\documentclass[leqno,11pt,a4paper]{amsart}
\usepackage{amsthm}
\usepackage{cite}
\usepackage{enumitem}
\usepackage{hyperref}
\usepackage{graphicx}
\usepackage{bm}
\usepackage{amssymb}
\usepackage{xspace}
\usepackage{xcolor}
\hypersetup{
  colorlinks,
  linkcolor={blue!50!black},
  citecolor={blue!50!black},
  urlcolor={blue!80!black}
}
\usepackage{multirow}
\usepackage[all]{xy}
\usepackage{tikz}
\usepackage{booktabs}
\usepackage{arydshln}
\usepackage{mleftright}
\usepackage{mathrsfs}

\setlist[enumerate]{labelsep=*, leftmargin=1.5pc}
\setlist[enumerate]{label=\normalfont(\roman*), ref=\roman*}
\usepackage[margin=2.7cm]{geometry}
\graphicspath{{images/}}
\theoremstyle{plain}
\newtheorem{thm}{Theorem}[section]
\newtheorem{pro}[thm]{Proposition}
\newtheorem{lem}[thm]{Lemma}

\newtheorem{conjecture}[thm]{Conjecture}
\theoremstyle{definition}
\newtheorem{dfn}[thm]{Definition}
\newtheorem{rem}[thm]{Remark}
\newtheorem{eg}[thm]{Example}
\newtheorem{algorithm}[thm]{Algorithm}

\newtheorem{cons}[thm]{Construction}

\DeclareMathOperator{\Pic}{Pic}

\DeclareMathOperator{\Div}{Div}
\DeclareMathOperator{\NE}{\overline{NE}}
\DeclareMathOperator{\Nef}{\overline{Amp}}
\DeclareMathOperator{\Eff}{Eff}

\DeclareMathOperator{\Spec}{Spec}

\DeclareMathOperator{\Gr}{Gr}

\DeclareMathOperator{\Id}{Id}

\DeclareMathOperator{\Newt}{Newt}

\DeclareMathOperator{\TV}{TV}

\DeclareMathOperator{\Bl}{Bl}

\newcommand{\bnu}{\bm{\nu}}

\newcommand{\bp}{\mathbf{p}}
\newcommand{\cA}{\mathcal{A}}

\newcommand{\cR}{\mathcal{R}}
\newcommand{\cP}{\mathcal{P}}
\newcommand{\cS}{\mathcal{S}}
\newcommand{\cV}{\mathcal{V}}

\newcommand{\cO}{\mathcal{O}}
\newcommand{\QQ}{{\mathbb{Q}}}
\newcommand{\RR}{{\mathbb{R}}}
\newcommand{\PP}{{\mathbb{P}}}

\newcommand{\ZZ}{{\mathbb{Z}}}
\renewcommand{\AA}{{\mathbb{A}}}

\newcommand{\CC}{\mathbb{C}}

\newcommand{\LL}{\mathbb{L}}
\newcommand{\TT}{\mathbb{T}}

\newcommand{\cX}{\mathcal{X}}

\newcommand{\conv}[1]{\operatorname{conv}\mleft({#1}\mright)}

\newcommand{\V}[1]{\operatorname{verts}\mleft({#1}\mright)}
\newcommand{\MM}[2]{{#1\text{--}#2}} 

\renewcommand{\tilde}{\widetilde}

\begin{document}
\author[T.\,Prince]{Thomas Prince}
\address{Mathematical Institute\\University of Oxford\\Woodstock Road\\Oxford\\OX2 6GG\\UK}
\email{thomas.prince@magd.ox.ac.uk}

\keywords{Fano manifolds, toric degenerations.}
\subjclass[2000]{14J45 (Primary), 14M25 (Secondary)}
\title[Cracked Polytopes]{From Cracked Polytopes to Fano Threefolds}
\maketitle
\begin{abstract}
	We construct Fano threefolds with very ample anti-canonical bundle and Picard rank greater than one from \emph{cracked polytopes} -- polytopes whose intersection with a complete fan forms a set of unimodular polytopes -- using Laurent inversion; a method developed jointly with Coates--Kasprzyk. We also give constructions of rank one Fano threefolds from cracked polytopes, following work of Christophersen--Ilten and Galkin. We explore the problem of classifying polytopes cracked along a given fan in three dimensions, and classify the unimodular polytopes which can occur as `pieces' of a cracked polytope.
\end{abstract}

\section{Introduction}

We explain how to construct an extensible database of Fano manifolds in each dimension. In particular, we develop a combinatorial framework, based on the notion of \emph{cracked polytopes} introduced in \cite{P:Cracked}. We show that this framework is flexible enough to obtain every Fano threefold with $-K_X$ very ample and $b_2 \geq 2$, famously classified by Mori--Mukai~\cite{Mori--Mukai:Erratum,Mori--Mukai:Kinosaki,Mori--Mukai:Manuscripta,Mori--Mukai:Tokyo,Mori--Mukai:Turin}. We show how one may extend these constructions to the rank one case -- adapting work of Christophersen--Ilten~\cite{CI14,CI16} -- and to cases for which $-K_X$ is not very ample.

To implement our method we first fix a unimodular rational fan $\Sigma$ of dimension $n$ containing $r$ rays. The \emph{ray map} of $\Sigma$ sends the $i$th element of the standard basis of $\ZZ^r$ to the primitive generator of the $i$th ray. The transpose of this map is an embedding of lattices and, tensoring with $\CC^\star$, defines an embedding of affine spaces. The fan $\Sigma$ also determines an embedded degeneration of $(\CC^\star)^n$ to a union of toric strata of $\CC^r$. The co-ordinate ring of the central fibre of this degeneration is given by a \emph{Stanley--Reisner ring} associated to the fan. Our prototypical example is the fan for $\PP^n$, which determines the embedded degeneration $\{x_1\cdots x_{n+1} = t\} \subset \CC^{n+1}$ obtained as $t \to 0$. Given such a fan $\Sigma$, our general procedure consists of two steps.

\begin{enumerate}
	\item Intersecting the fan $\Sigma$ with a lattice polytope $P$, we describe how the embedding of $(\CC^\star)^n$ determined by $\Sigma$ may be compactified to an embedding of the toric variety $X_P$ in a non-singular toric variety $Y$. This is based on \cite{P:Cracked} and joint work \cite{CKP17} with Coates and Kasprzyk. 
	\item The embedding of affine spaces determined by $\Sigma$ admits various possible deformations, and we explicitly construct embedded deformations in $Y$ by homogenizing the co-ordinate rings of such families.
\end{enumerate}

In this article the fans $\Sigma$ we consider are simple enough that we can deform the corresponding embeddings explicitly. However, in \S\ref{sec:gross_siebert} we outline a potentially sweeping generalisation using the work of Gross--Hacking--Keel \cite{GHK1} and Gross--Hacking--Siebert~\cite{GHS}. In particular, the authors construct mirror families to log Calabi--Yau varieties which deform the embeddings of the affine spaces and vertex varieties described above. In work in progress with Barrott and Kasprzyk we determine precisely when these families admit a fibrewise compactification in $Y$ in the two-dimensional setting.

The connection between mirror log Calabi--Yau families and Fano threefolds is also currently being investigated by Corti--Hacking--Petracci in \cite{CHP}. When fully established such work would guarantee the existence of a smooth Fano associated to each (mirror) \emph{Minkowski polynomial} (see \cite{CCGK,ACGK}). In that context the current work would form a bridge between these (log) deformation theoretic constructions and the constructions of Mori--Mukai; providing explicit toric degenerations -- embedded in a toric ambient space -- from which birational descriptions can deduced from the structure of the ambient space.

The current work fits into another program of research, directed toward a novel approach to Fano classification. In \cite{CCGK} Coates--Corti--Galkin--Kasprzyk identify (a number of) \emph{mirror Laurent polynomials} for each family of Fano threefolds. These constructions rely on the computation of the \emph{quantum period} (part of the small $J$-function) of each Fano threefold, which in turn relies on the existence of good models of these Fano varieties; either as toric complete intersections, or via representation theoretic constructions. We make heavy use of these constructions, noting that these constructions are usually compatible with Laurent inversion. We note that the connection between toric degenerations and mirror symmetry is further explored by Ilten--Lewis--Przyjalkowski~\cite{Ilten--Lewis--Przyjalkowski}.

Fixing a complete fan -- which we refer to as the \emph{shape} -- we say a polytope is cracked along $\Sigma$ if its intersection with each maximal cone of $\Sigma$ is unimodular, see Definition~\ref{dfn:cracked}. In \cite{CKP17} we show that embeddings of $X_P$ into toric varieties, compactifying the embedding of affine varieties described above, are described by \emph{scaffoldings}. Moreover, in \cite{P:Cracked} we show that embeddings of $X_P$ into non-singular toric ambient spaces $Y$ correspond to the combinatorial condition that the scaffolding is \emph{full}, see Definition~\ref{dfn:full} and Theorem~\ref{thm:smooth_ambient}.

\begin{thm}
	\label{thm:fano_from_cracked}	
	Every smooth Fano threefold with a very ample anti-canonical bundle and $b_2 \geq 2$ can be obtained by smoothing a Gorenstein toric Fano variety. In particular these can be constructed as deformations of toric embeddings provided by Laurent inversion, applied to a cracked polytope together with a full scaffolding $S$. Moreover, we may assume that the shape of the scaffolding $S$ appears in Table~\ref{tbl:shapes}.
\end{thm}

We recall that the ideal of $X_P$ in the homogeneous co-ordinate ring of the toric ambient space $Y$ is determined by the choice of shape $\Sigma$: for example, if $\TV(\Sigma)$ is a product of projective spaces, a full scaffolding with this shape realises $X_P$ as a toric complete intersection.

Extending the list of shapes given in Table~\ref{tbl:shapes} to include the varieties $Z_{2g-2}$ for $g \in \{2,8,9,10,12\}$ defined in \S\ref{sec:rank_one}, we obtain members of every family of Fano threefolds with very ample anti-canonical bundle from a cracked polytope and full scaffolding. We consider the Fano threefolds for which $-K_X$ is not very ample in \S\ref{sec:not_va}.

\begin{table}
	\[
	\begin{array}{cc|cc}
	Z & \rho(Z) & Z & \rho(Z) \\\hline
	pt & 0 & dP_7 & 3 \\
	\PP^1 & 1 & dP_6 & 4\\
	\PP^2 & 1 & Z_{10} = dP_7\times \PP^1 & 4 \\
	\PP^3 & 1  & dP_6\times \PP^1 & 5 \\
	\PP^1\times \PP^1 & 2 & Z_{12} & 5 \\
	\PP^2 \times \PP^1 & 2  & dP_5'\times\PP^1 & 6 \\
	\PP^1\times \PP^1\times \PP^1 & 3 & & \\
	\end{array}
	\]
	\caption{The various shape varieties used to construct Fano threefolds.}
	\label{tbl:shapes}
\end{table}

We suggest that four-dimensional cracked polytopes form classes of polytopes from which it is natural to algorithmically construct Fano fourfolds. We note, by way of example, that each of the $738$ families of Fano fourfolds which appear in \cite{CKP15} can be constructed from a polytope cracked along the fan of a product of projective spaces via a full scaffolding.

\subsection*{Acknowledgements}
We thank Tom Coates and Alexander Kasprzyk for our many conversations about Laurent inversion. The author is supported by a Fellowship by Examination at Magdalen College, Oxford.

\subsection*{Conventions}
Throughout this article $N$ will refer to an $3$-dimensional lattice, and $M := \hom(N,\ZZ)$ will refer to the dual lattice. Given a ring $R$ we write $N_R := N \otimes_\ZZ R$ and $M_R := M \otimes_\ZZ R$. For brevity we let $[k]$ denote the set $\{1,\ldots,k\}$ for each $k \in \ZZ_{\geq 1}$. We work over the field $\CC$ of complex numbers throughout this article. Given a reflexive polytope $P \subset N_\RR$, we assume throughout that $X_P$ is the toric variety associated to the fan of cones over faces of $P$. Cracked polytopes will always be contained in $M_\RR$; in particular if $Q$ is a polytope cracked along a fan $\Sigma$, $\Sigma$ is a fan in $M_\RR$. Given a variety $Y$, and an identification $\Pic(Y) \cong \ZZ^r$, we write $\cO(a_1,\ldots,a_r)$ for the line bundle of (multi) degree $a = (a_1,\ldots,a_r) \in \ZZ^r$.
\section{Cracked polytopes and Laurent Inversion}
\label{sec:cracked_polytopes}

The method \emph{Laurent inversion} -- introduced in \cite{CKP17} -- was developed to construct models of Fano manifolds embedded in toric varieties. To describe this method we first fix a splitting $N = \bar{N} \oplus N_U$ of $N$. We fix a Fano polytope $P \subset N_\RR$ and a smooth toric variety $Z$ (the \emph{shape}), such that $\bar{N}$ is the character lattice of the dense torus in $Z$. The central definition in the Laurent inversion construction is that of \emph{scaffolding}. Loosely, a scaffolding is a collection of polytopes associated to nef divisors on $Z$ whose convex hull is equal to $P$. From a scaffolding we construct a polytope $Q_S$ which projects to $P^\circ$. The toric variety $X_P$ embeds into the toric variety $Y_S$ associated to the normal fan of $Q_S$. Moreover, the corresponding ideal in the homogeneous co-ordinate ring of $Y_S$ is determined by $Z$. We then test explicit deformations of the equations cutting out $X_P$ in $Y_S$ to attempt to construct an embedded smoothing.

For general choices of $S$, the variety $Y_S$ may be highly singular: for example $Y_S$ need not be $\QQ$-Gorenstein. In \cite{P:Cracked} we explore the (restrictive) conditions on $S$ which ensure that $Y_S$ is non-singular, and introduce the following notion.

\begin{dfn}[{\!\!\cite[Definition~$2.1$]{P:Cracked}}]
	\label{dfn:cracked}
	Fix a convex polyhedron $P \subset M_\RR$ containing the origin in its interior, and a unimodular fan $\Sigma$. We say $P$ is \emph{cracked along $\Sigma$} if every tangent cone of $P \cap C$ is unimodular for every maximal cone $C$ of $\Sigma$.
\end{dfn}

\begin{rem}
	Let $P$ be a polytope cracked along a fan $\Sigma$. We do not assume that the minimal cone of the fan $\Sigma$ is not necessarily zero dimensional; these are sometimes called \emph{generalised fans}. The shape $Z$ is the toric variety associated to the quotient $\bar{\Sigma}$ of $\Sigma$ by its minimal cone. Slightly abusing terminology, we also say that $P$ is cracked along the fan $\bar{\Sigma}$.
\end{rem}

It follows from \cite[Proposition~$2.5$]{P:Cracked} that any cracked polytope is reflexive. In three dimensions the converse holds, in the sense that any reflexive polytope is cracked along \emph{some} complete unimodular fan. Indeed, consider the fan $\Sigma$ defined by taking the cone over every face of a maximal triangulation of the boundary of $P$; the polytopes obtained by intersecting maximal cones of $\Sigma$ with $P$ are all standard simplices. Some examples of cracked polytopes are displayed in Figure~\ref{fig:cracked_polytopes}. The polytope shown in the left-hand image of Figure~\ref{fig:cracked_polytopes} is cracked along the product of $\RR^2$ with the fan determined by $\PP^1$; while the polytope shown in the right-hand image is cracked along the fan determined by $\PP^1\times\PP^1\times\PP^1$.

	\begin{figure}
	\centering
	\begin{minipage}[b]{0.49\textwidth}
		\includegraphics[width=\textwidth]{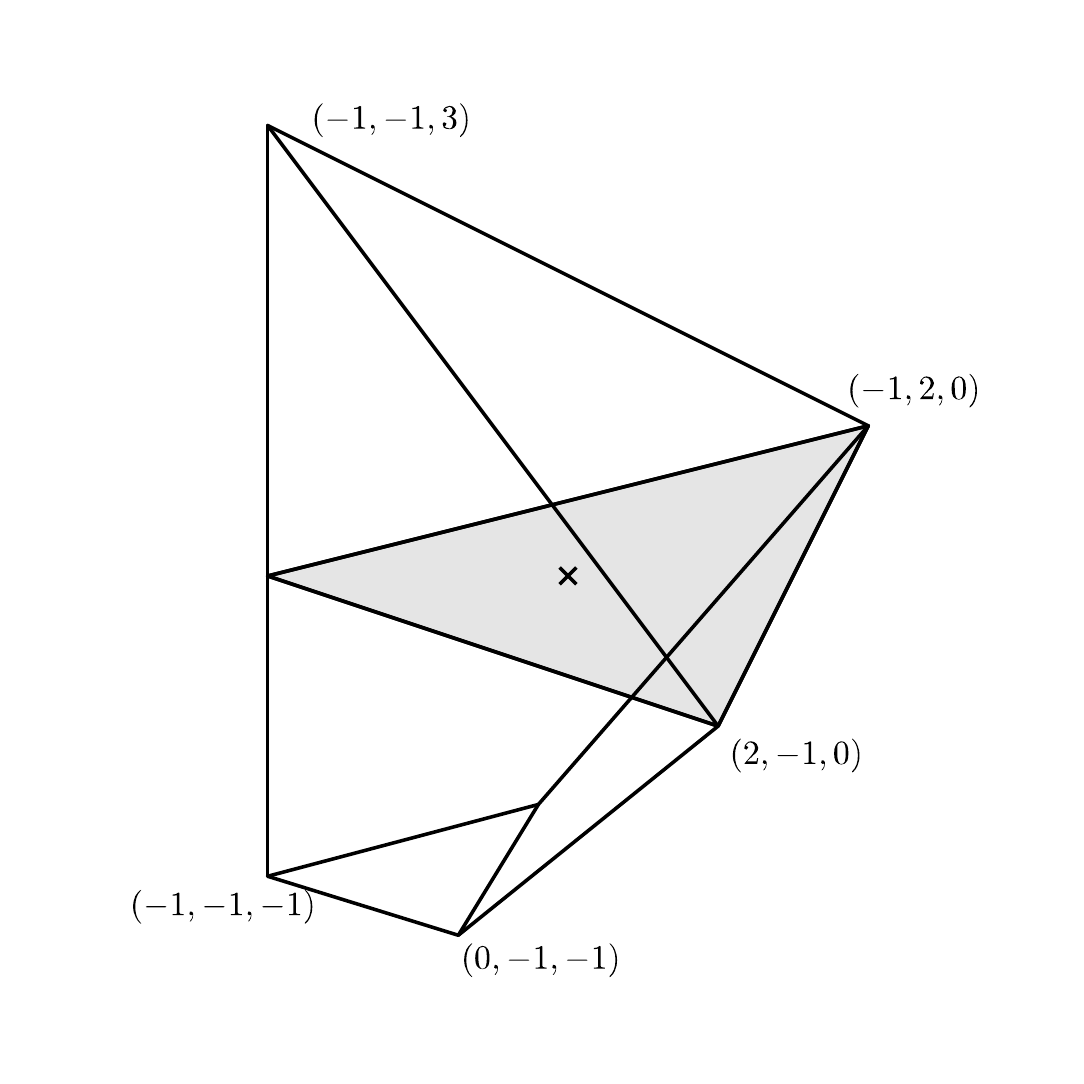}
	\end{minipage}
	\hfill
	\begin{minipage}[b]{0.49\textwidth}
		\includegraphics[width=\textwidth]{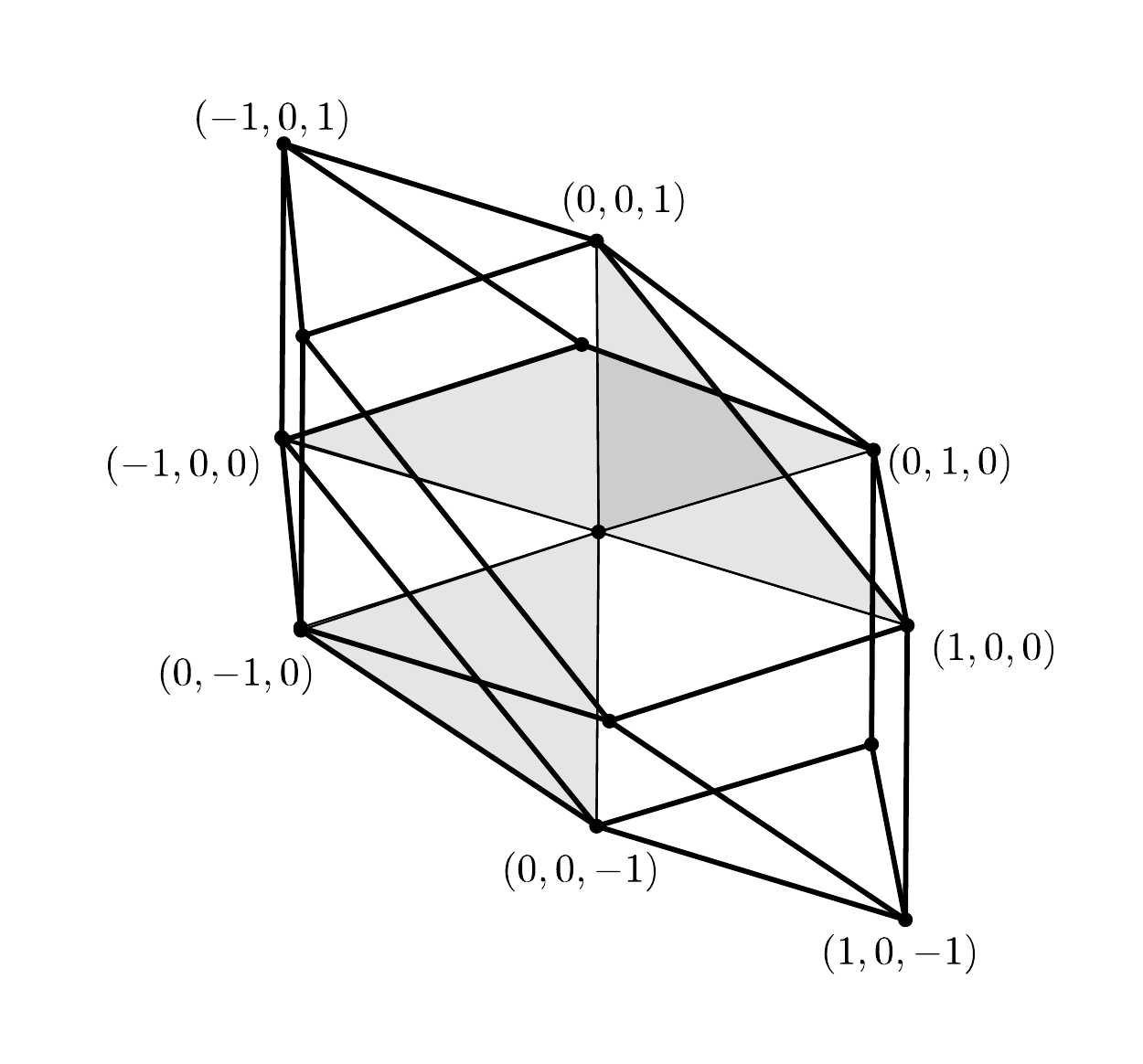}
	\end{minipage}
	\caption{Examples of cracked polytopes.}
	\label{fig:cracked_polytopes}
\end{figure}


\begin{dfn}[{\!\!\cite[Definition~$3.1$]{CKP17}}]
	\label{dfn:scaffolding}
	Fix a smooth projective toric variety $Z$ with character lattice $\bar{N}$. A \emph{scaffolding} of a polytope $P$ is a set of pairs $(D,\chi)$ -- where $D$ is a nef divisor on $Z$ and $\chi$ is an element of $N_U$ -- such that
	\[
	P = \conv{P_D + \chi \, \Big| \, (D,\chi) \in S}.
	\]
	We refer to $Z$ as the \emph{shape} of the scaffolding, and elements $(D,\chi) \in S$ as \emph{struts}. We also assume that there is a unique $s = (D,\chi)$ such that $v \in P_D+\chi$ for every vertex $v \in\V{P}$.
\end{dfn}

Scaffolding a polytope $P$ determines an embedding of $X_P$ into an ambient space $Y_S$. This is the main result of \cite{CKP17}; see also the treatment given in \cite[\S$3$]{P:Cracked}.

\begin{dfn}[{\!\!\cite[Definition~$A.1$]{CKP17}}]
	\label{dfn:ambient}
	Given a scaffolding $S$ of $P$ we define a toric variety $Y_S$, associated to the normal fan $\Sigma_S$ of the polytope $Q_S \subset \tilde{M}_\RR := (\Div_{T_{\bar{M}}}Z \oplus M_U)\otimes_\ZZ \RR$, itself defined by the inequalities
	\[
	\begin{cases}
	\big\langle (-D,\chi), - \big\rangle \geq -1 & \textrm{for all $(D,\chi) \in S$};\\
	\big\langle (0,e_i), - \big\rangle \geq 0 & \textrm{for $i \in [\ell]$};
	\end{cases}
	\]
	where $e_i$ denotes the standard basis of $\Div_{T_{\bar{M}}}Z \cong \ZZ^\ell$. 
\end{dfn}

We let $\rho$ denote the ray map of the fan $\bar{\Sigma}$ determined by $Z$, and set $\rho_s := (-D,\chi)$ for each $s = (D,\chi) \in S$. We also define a map of lattices,
\[
\xymatrix@R-1pc{
	\llap{$\theta := \rho^\star\oplus \Id \colon $} \bar{N} \oplus N_U \ar[rr]& & \Div_{T_{\bar{M}}}(Z) \oplus N_U \\
	N \ar@{=}[u] & & \tilde{N}. \ar@{=}[u]
}
\]
\begin{thm}[{\!\!\cite[Theorem~$5.5$]{CKP17}}]
	\label{thm:embedding}
	A scaffolding $S$ of a polytope $P$ determines a toric variety $Y_S$ and an embedding $X_P \to Y_S$. This map is induced by the map $\theta$ on the corresponding lattices of one-parameter subgroups.
\end{thm}

\begin{rem}
	We can provide an explicit generating set for the ideal of $X_P$ in the homogeneous co-ordinate ring of $Y_S$ using the map $\theta$. In particular, a hyperplane containing the image of $\theta$ defines a function $h$ on the set of ray generators of $\Sigma_S$. $X_P$ then satisfies the equation 
	\[
	\prod_{\{v:h(v) \geq 0\}}z_v^{h(v)} - \prod_{\{v:h(v) < 0\}} z_v^{h(v)} = 0,
	\]
	where products are taken over the ray generators of $\Sigma_S$, and $z_v$ is the homogeneous co-ordinate on $Y_S$ corresponding to the ray generated by $v$.
\end{rem}

Recall that each facet $F$ of $P^\circ$ is dual to a vertex $F^\star$ of $P$, contained in a cone $\sigma$ of $\Sigma$. Taking $\sigma$ is minimal among such cones, $\sigma$ corresponds to a non-singular toric stratum $Z(\sigma)$ of the toric variety $\TV(\bar{\Sigma})$. It is shown in \cite[Proposition~$2.8$]{P:Cracked} that the facet $F$ of $P^\circ$ is a \emph{Cayley sum} $P_{D_1} \star \cdots \star P_{D_k}$, where $\{D_i : 1 \leq i \leq k\}$ is a set of nef divisors on $Z(\sigma)$, and $k = \dim(\bar{\sigma})+1$. We call a face of $P^\circ$ \emph{vertical} if it is contained in a factor $P_{D_i}$ of some facet $F = P_{D_1} \star \cdots \star P_{D_k}$ and some $i \in \{1,\ldots,k\}$.

\begin{dfn}[{\!\!\cite[Definition~$4.1$]{P:Cracked}}]
	\label{dfn:full}
	Given a Fano polytope $P \subset N_\RR$ cracked along a fan $\Sigma$ in $M_\RR$ we say a scaffolding $S$ of $P$ with shape $Z := \TV(\bar{\Sigma})$ is \emph{full} if every vertical face of $P$ is contained in a polytope $P_D+\chi$ for a unique element $(D,\chi) \in S$.
\end{dfn}

We show in \cite{P:Cracked} that full scaffoldings on cracked polytopes give rise to embeddings $X_P \to Y_S$ where $Y_S$ is smooth in a neighbourhood of $X_P$.

\begin{thm}[{\!\!\cite[Theorem~$1.1$]{P:Cracked}}]
	\label{thm:smooth_ambient}
	Fix a polytope $P \subset M_\RR$, and a rational fan $\Sigma$ in $M_\RR$ such that the toric variety $Z := \TV(\bar{\Sigma})$ is smooth and projective. Given a scaffolding $S$ of $P$ with shape $Z$, we have that the target of the corresponding embedding is smooth in a neighbourhood of the image of $X_P$ if and only if $P$ is \emph{cracked} along $\Sigma$ and $S$ is \emph{full}.
\end{thm}

\subsection{Torus quotients}
\label{sec:secondary_fan}

Every $n$-dimensional toric variety $X$ (over $\CC$) may be described as the quotient of a Zariski open set of affine space $\CC^{n+r}$ by a complex torus $\TT := (\CC^\star)^r$. Recalling that, if $X$ is determined by a fan in $N$ whose rays generators $\nu_1,\ldots,\nu_{n+r}$ form a spanning set of $N$, we have an exact sequence
\[
\xymatrix{
	0 \ar[rr] & & \LL \ar[rr] & & \ZZ^{n+r} \ar^{\bnu}[rr] & & N \ar[rr] & & 0 
}
\]
where $\bnu\colon e_i \to \nu_i$ for each $i \in \{1,\ldots, n+r\}$. The character lattice $\LL^\star$ of $\TT$ fits into the dual sequence,
\[
\xymatrix{
	0 \ar[rr] & & M \ar[rr] & & (\ZZ^{n+r})^\star \ar^R[rr] & & \LL^\star \ar[rr] & & 0.
}
\]
Moreover we recall that if $X$ is smooth there is a canonical identification $\LL^\star \cong \Pic(X)$, while if $X$ is $\QQ$-factorial there is a canonical identification of $\LL^\star_\RR := \LL\otimes_\ZZ\RR \cong \Pic(X)_\RR$. The map $R\colon (\ZZ^{n+r})^\star \to \LL^\star$ is called the \emph{weight data} for the toric variety. Recall that the possible fans in $N$, with rays generated by a subset of $\{\nu_1,\ldots,\nu_{n+r}\}$, and such that the associated toric variety is projective, are indexed by the cones of a fan contained in the effective cone $\Eff(X) \subset \Pic(X)_\RR$. This fan is called the \emph{secondary fan} or \emph{GKZ decomposition}.

Fixing a maximal cone (or \emph{chamber}) $\sigma$ in the secondary fan, the corresponding toric variety can be described as the torus quotient
\[
X_\sigma = \left(\CC^{n+r} \setminus Z(\sigma)\right) / \TT,
\]
where $\TT := (\LL^\star \otimes_\ZZ \CC^\star)$, the weights of the torus action are given by $R$, and the $Z(\sigma)$ is the \emph{irrelevant locus}. Choosing a point (or \emph{stability condition}) $\omega$ in the interior of $\sigma$, the irrelevant locus is defined by setting
\[
Z(\sigma) := V\left( x_{i_1}\cdots x_{i_r} : \omega \in \langle R_{i_1}, \ldots, R_{i_r} \rangle \right),
\]
where $R_i = R(e_i)$ for each standard basis vector $e_i$, $i \in [n+r]$. Some of the constructions described in \S\ref{sec:constructions} make use of stability conditions contained in a codimension one cone  (or \emph{wall}) in the secondary fan.

We can use the GIT presentation of a toric variety to streamline the construction of the variety $Y_S$ from a scaffolding $S$. To do this we first assume that, that there is a basis $B = \{b_i \in N_U : i \in [\dim N_U] \}$, such that $\{(0,b) : b \in B\} \subseteq S$. If this condition holds the cone generated by
\[
B \cup \{e_i : i \in [\dim \Div_{T_{\bar{M}}}(Z)]\},
\]
where the vectors $e_i$ form the standard basis in the based lattice $\Div_{T_{\bar{M}}}(Z)$, defines a smooth torus invariant point in $Y_S$. We assume for the remainder of this section that every scaffolding satisfies this condition. We next explain how to form a weight matrix and stability condition which determine the variety $Y_S$ directly from the scaffolding $S$. This construction follows \cite[Algorithm~$5.1$]{CKP17}.

\begin{cons}
	\label{cons:scaff_weights}
	Given a scaffolding $S$ with shape $Z$ of a polytope $P$, index the elements of $S$ by $[s]$, and let $(D_i,\chi_i)$ denote the $i$\textsuperscript{th} element of $S$. It follows from our assumptions on $S$ that the ray matrix of $\Sigma_S$ is in echelon form
	\[
  \left(
	\begin{array}{c:ccc}
	\multirow{2}{*}{$\quad I_n\quad$} & -D_1 & \cdots & -D_{r} \\
	& \chi_1 & \cdots & \chi_r \\
	\end{array}
	\right),
	\]
	where $[s] \setminus [r]$ indexes the elements $(D_i,\chi_i) \in S$ of the form $(0,b_i)$, for a basis $\{b_i : i \in [\dim N_U]\}$, and $n = \dim \tilde{N}$. Thus $R$, the transpose of the kernel matrix, is given by
	\[
	R = 
	\left(
	\begin{array}{c:cc}
	\multirow{3}{*}{$\quad I_r\quad$} & -\chi_1 & D_1 \\
	    & \vdots & \vdots \\
	& -\chi_r & D_r \\
	\end{array}\right).
	\]
	The variety $Y_S$ is defined using the a polarising torus invariant divisor given by the sum of all rays corresponding to elements of $S$. The (multi) degree of this divisor is given by the sum of the first $s$ columns of $R$. That is, the stability condition used to define $Y_S$ is given by the sum of $(1,\ldots,1)^T$ with the columns of the matrix $(\chi_1,\ldots,\chi_r)^T$
	
	If $Z$ is a product of $c$ projective spaces, there is a partition of the columns of $R$ containing the vectors $D_i \in \Div_{T_{\bar{M}}}(Z)$. In particular, the standard basis in $\Div_{T_{\bar{M}}}(Z)$ partitions into $c$ sets $C_1,\ldots,C_c$, such that $C_i$ consists of divisors pulled back from the standard projection to the $i$\textsuperscript{th} projective space factor. For each $i \in [c]$ the degree of the line bundle $L_i$ cutting out $X_P$ in $Y_S$ is given by the sum of the columns in $C_i$. In particular, there is a distinguished \emph{binomial} $z^{m_1} - z^{m_2}$ in $L_i$, where $m_1$ is the sum of standard basis vectors in $(\ZZ^{n+r})^\star$ corresponding to the columns of $C_i$, and $m_2$ is the unique lift of $L_i \in \LL^\star$ to $(\ZZ^r)^\star$: the subspace of $(\ZZ^{n+r})^\star$ corresponding to the first $r$ columns of $R$. It is shown in \cite{CKP17}, see also \cite[\S$3$]{P:Cracked}, that $X_P$ is the vanishing locus of these $c$ binomials. 
\end{cons}

\begin{eg}
	\label{eg:anti_canonical}
	Fix a $3$-dimensional reflexive polytope $P$, and let $Z$ be a crepant resolution of the toric variety determined by the normal fan of $P$. In particular, $\bar{N} = N$ and $N_U = \{0\}$. Let $S := \{(D,0)\}$, where $D \in |-K_Z|$ is the toric boundary of $Z$. Hence $P = P_{D}$, and the corresponding $1 \times n$  weight matrix $R$ is equal to $\begin{pmatrix}1& 1& \cdots & 1\end{pmatrix}$, where $n = 1+\dim \Div_{T_{M}}(Z)$ columns. The stability condition is equal to $1 \in \LL^\star \cong \ZZ$, and hence $Y_S \cong \PP^{n-1}$. This is nothing but the anti-canonical embedding of $X_P$ into projective space.
\end{eg}

\begin{eg}
	\label{eg:dP6}
		
	\begin{figure}
		\includegraphics{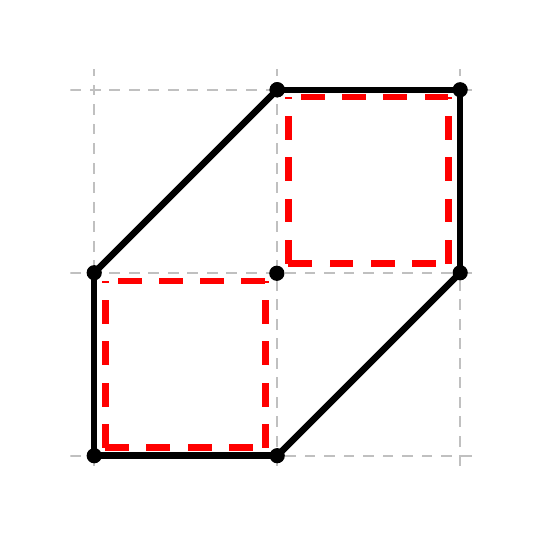}
		\caption{The scaffolding corresponding to the embedding of $dP_6$ in $\PP^2\times \PP^2$.}
		\label{fig:dP6}
	\end{figure}

	In \cite[Example~$3.5$]{CKP17} we consider two distinct scaffoldings for the polygon $P$ associated with the toric del Pezzo surface of degree six. One of these is illustrated in Figure~\ref{fig:dP6}.	The scaffolding illustrated in Figure~\ref{fig:dP6} has shape $Z = \PP^1\times \PP^1$ and -- letting $D_{i,a}$ denote the pullback of $\{a\} \subset \PP^1$ along the $i${\textsuperscript{th}} projection for each $i \in \{1,2\}$ and $a \in \{0,\infty\}$ -- we define 
	\[
	S := \{\{D_{1,0}+D_{2,0}\},\{D_{1,\infty}+D_{2,\infty}\}\}.
	\]
	Applying Construction~\ref{cons:scaff_weights} to $S$ we obtain the weight matrix
	\[
	\left(
	\begin{array}{c:cccc}
	\multirow{2}{*}{$\quad I_2\quad$} & 1 & 0 & 1 & 0 \\
	 & 0 & 1 & 0 & 1
	\end{array}\right).
	\]
	and stability condition $\omega = (1,1)$. This is a GIT presentation of the toric variety $\PP^2\times \PP^2$. The variety $X_P \cong dP_6$ is the vanishing locus of the binomials $x_1y_1 = x_0y_0$ and $x_2y_2=x_0y_0$, where $x_i$ and $y_i$ denote homogeneous co-ordinates on the $\PP^2$ factors.
\end{eg}
\section{Rank one Fano threefolds}
\label{sec:rank_one}

Toric degenerations of rank one Fano manifolds have been obtained by Ilten and Christophersen~\cite{CI14,CI16}, using the deformation theory of Stanley--Reisner rings developed by Altmann--Christophersen~\cite{AC10,AC04}. Using these results -- and the work of Galkin~\cite{G07} on small toric degenerations -- we obtain cracked polytopes $P$ corresponding to each of the $15$ rank one Fano threefolds $X$ with very ample anti-canonical bundle. In particular, we describe degenerations of these $15$ Fano threefolds $X$ to the toric varieties $X_P$. We remark that, since the toric degenerations in this case occur in the anti-canonical embedding, the use of cracked polytopes in this context is rather trivial; see Example~\ref{eg:anti_canonical}.

\begin{table}
	\[\def\arraystretch{1.5}
	\begin{array}{cccc|cccc}
	\textrm{Fano} & \textrm{PALP ID} & \textrm{Equations} & \textrm{Shape} & \textrm{Fano} & \textrm{PALP ID} & \textrm{Equations} & \textrm{Shape} \\\hline
	\PP^3 & 0 & - & pt & V_8 & 4250 & \begin{cases} x_1x_2-x_0^2 \\ x_3x_4-x_0^2 \\ x_5x_6-x_0^2  \end{cases} & \PP^1\times\PP^1\times \PP^1\\
	Q^3 & 1 & x_1x_2-x_0^2 & \PP^1 & V_{10} & 4142 & \S\ref{sec:higher_genus} & dP_7\times\PP^1 \\
	B_5 & 245 & \S\ref{sec:B5} & dP_7 & V_{12} & 3868 & \S\ref{sec:higher_genus} & Z_{12} \\
	B_4 & 433 & \begin{cases} x_1x_2-x_0^2 \\ x_3x_4-x_0^2 \end{cases} & \PP^1\times\PP^1 & V_{14} & 3297 & \S\ref{sec:higher_genus} & Z_{14} \\
	B_3 & 741 & x_1x_2x_3-x_0^3 & \PP^2 & V_{16} & 3033 & \S\ref{sec:higher_genus} & Z_{16}\\
	B_2 & 427 & \S\ref{sec:B2} & Z_2 & V_{18} & 2702 & \S\ref{sec:higher_genus} & Z_{18} \\
	V_4 & 4311 & x_1x_2x_3x_4-x_0^3 & \PP^3 & V_{22} & 1942 & \S\ref{sec:higher_genus} & Z_{22}\\
	V_6 & 4286 & \begin{cases} x_1x_2-x_0^2 \\ x_3x_4x_5-x_0^2 \end{cases} & \PP^1\times\PP^2 & & &.
	\end{array}
	\]
	\caption{Rank one Fano threefolds}
	\label{tbl:rank_one}
\end{table}

To specify the toric varieties $Z_{2n}$ for $n \in \{6,7,8,9,11\}$ which appear in Table~\ref{tbl:rank_one}, we set $Z_{10} := dP_7 \times \PP^1$, and let $\ell^a_1,\ldots,l^a_5$ denote the torus invariant divisors of $dP_7\times \{a\} \subset Z_{10}$ for each $a \in \{0,\infty\}$.

\begin{itemize}
	\item  $Z_{12}$ is the blow up of $Z_{10} := dP_7\times\PP^1$ in a toric invariant line $\ell^0_1 \subset Z_{10}$.
	\item $Z_{14}$ is the blow up of $Z_{12}$ in the strict transform (and pre-image) of $\ell^\infty_2 \subset Z_{10}$.
	\item $Z_{16}$ is the blow up of $Z_{14}$ in the strict transform of the line $\ell^0_5 \subset Z_{10}$.
	\item $Z_{18}$ is the blow up of $Z_{16}$ in the strict transform of the line $\ell^\infty_3 \subset Z_{10}$.
\end{itemize}

The fans determined these varieties define triangulations of the sphere via radial projection. The sequence of blow up maps described induces the \emph{starring} operations on these triangulations described in \cite{CI14}. We define the variety $Z_{22}$ to be a crepant resolution of the toric variety determined by the normal fan of the reflexive polytope with ID $1941$. Similarly, we define the variety $Z_2$ to be a crepant resolution of the toric variety determined by the normal fan of the (self-dual) reflexive polytope with ID $427$.

The Fano variety $\PP^3$ is toric, while $Q^3,B_2,B_3,B_4,V_4,V_6$, and $V_8$ are well known to be toric complete intersections. These admit toric degenerations to the varieties defined by the equations given in Table~\ref{tbl:rank_one}. To describe the scaffolding associated to each of these Fano threefolds, let $d$ be the dimension of the shape variety $Z$, set $\bar{N} := \ZZ^d$ and $N_U := \ZZ^{3-d}$. Letting $\{e_1,\ldots,e_{3-d}\}$ denote the standard basis of $N_U$, we define 
\[
S := \{(0,e_1),\ldots,(0,e_{3-d}),(D,\chi))\},
\]
where $D \in |-K_Z|$ is the toric boundary of $Z$, and $\chi = (-1,\ldots,-1) \in N_U$. This scaffolding is illustrated in the case $B_3$ in Figure~\ref{fig:wPS} (setting $a=1$ and $b=3$).

\subsection{Pfaffian equations and $B_5$}
\label{sec:B5}

The Fano threefold $B_5$ is a linear section of the Grassmannian $\Gr(2,5)$. We make heavy use of the fact that the ideal of the image of the Pl\"{u}cker embedding
\[
\Gr(2,n) \hookrightarrow \PP^{{n \choose 2}-1}
\]
is generated by $4 \times 4$ Pffafians of a skew-symmetric $n \times n$ matrix; entries of which are the Pl\"{u}cker co-ordinates of $\Gr(2,n)$. Hyperplane sections can then be obtained by replacing entries with linear combinations of a subset of the Pl\"{u}cker co-ordinates. For example, $B_5$ can be described as the Pfaffians of the matrix

\[
\begin{pmatrix} x_0 & x_1 & x_2 & x_0 \\
& tx_0 & x_3 & x_4 \\
& & x_0 & x_5 \\
& & & tx_0 
\end{pmatrix}
\]
for a fixed value of $t\neq 0$. Varying $t$ defines a flat family, the central fibre of which is the projective cone over a toric variety with two ordinary double points, obtained from $dP_5$ by moving the four points at which $\PP^2$ is blown up to two pairs of infinitely close points, and contracting the pair of resulting $-2$ curves. Setting $t=0$ recovers five equations generating the ideal of a toric variety in $\PP^5$. This toric variety is isomorphic to $X_P$, where $P$ denotes the toric variety with ID $741$. The embedding $X_P \to \PP^5$ is the embedding of $X_P$ determined by the scaffolding $S= \{(0,1),(D,0)\}$, where $1 \in N_U \cong \ZZ$ and $D \in -K_Z$ (recalling that $Z = dP_7$) is the toric boundary of $Z$.

\subsection{Higher genus Fano threefolds}
\label{sec:higher_genus}

The varieties $V_{2n-2}$ for $n \in \{6,7,8,9,10,12\}$ are linear sections of the \emph{Mukai varieties}  $M_n$~\cite{M88}. Toric degenerations of these are related -- by work of Ilten--Christophersen~\cite{CI14} -- to the \emph{convex deltahedra} in the cases $n < 12$, while varieties in the family $V_{22}$ admit a toric degeneration to a variety with ordinary double point singularities, see~\cite{G07}.

Given a Fano toric variety $Z$, let its \emph{dual} $Z^\star$ be toric variety associated to the normal fan of the convex hull of the ray generators of the fan determined by $Z$.

\begin{pro}
	The toric varieties $V_{2n-2}$ admit toric degenerations to the Fano toric varieties $Z^\star_{2n-2}$ dual to $Z_{2n-2}$ for each $n \in \{6,7,8,9,10,12\}$.
\end{pro}
\begin{proof}
	If $n < 12$ we recover the triangulations $T_n$ of $S^2$ used in \cite{CI14} to construct degenerations of Fano threefolds by removing the origin from $N_\RR \cong \RR^3$ and radially projecting the fan $\Sigma_n$ determined by $Z_{2n-2}$. The result then follows immediately from \cite[Proposition~$2.3$]{CI14}. In the case $n=12$ we observe that $Z^\star_{22}$ contains only ordinary double point singularities, and hence admits a smoothing. It is shown in \cite{G07} that the general fibre of this smoothing is a member of the family $V_{22}$.
\end{proof}

In the cases $n \in \{6,7,8\}$ we can provide an explicit description of the toric degeneration.
\begin{enumerate}
	\item $V_{10}$: varieties in this family can be described by the Pfaffians of a $5\times 5 $ skew-symmetric matrix, and one quadric equation. We can form a toric degeneration following \S\ref{sec:B5}.
	\item $V_{12}$: varieties in this family can be described via a system of $9$ Pfaffian equations, and we refer to the treatment of \MM{2}{21} in \S\ref{sec:constructions} for a description of a toric degeneration using the same shape variety.
	\item $V_{14}$: varieties in this family can be described as the vanishing of the $4 \times 4$ Pfaffians of a $6\times 6$ skew matrix. An explicit toric degeneration is given by the $4\times 4$ Pfaffians of the matrix \eqref{eq:six_by_six} below.
\end{enumerate}

The vanishing $4\times 4$ Pfaffians of the matrix
\begin{equation}
\label{eq:six_by_six}
\begin{pmatrix}
-x_1 & x_2 & tf_1 & x_3 & x_4 \\
 & tg_1 & x_5 & x_1 & x_6 \\
 & & x_7 & x_2 & x_0 \\
 & & & x_0 & x_8 \\ 
& & & & th_1 \\
\end{pmatrix}
\end{equation}
define a toric degeneration of $V_{14}$, a general linear section of $\Gr(2,6)$, where $x_i$ are homogeneous co-ordinates on $\PP^9$ and $f_1$, $g_1$, and $h_1$ are general linear forms on $\PP^9$. The scaffolding $S$ in each case is equal to the singleton set $\{(D,0)\}$, where $D$ is the toric boundary of $Z$.

\subsection{The quartic hypersurface in $\PP(1,1,1,1,2)$}
\label{sec:B2}

Recall that the toric variety $Y_S$ defined by a full scaffolding of a cracked polytope $P$ is non-singular in a neighbourhood of the image of $P$. This excludes certain constructions of Fano manifolds as hypersurfaces as weighted projective spaces. In particular, consider the scaffolding (with shape $\PP^2$) of the polytope $P$ with ID $3312$ illustrated in Figure~\ref{fig:wPS}, setting $(a,b) = (1,4)$. We have that $\bar{N} \cong \ZZ^2$, $N_U \cong \ZZ$, and $S = \{(0,1),(D_0+D_1+2D_2,-1)\}$; where $D_i := \{x_i=0\} \subset \PP^2$. Computing the corresponding weight matrix we find
\begin{align*}
R = \left(
\begin{array}{c|c|c}
I_r & \chi & D
\end{array}
\right) &= 
\left(
\begin{array}{c|c|ccc}
1 & 1 & 1 & 1 & 2
\end{array}
\right).
\end{align*}
Thus $X_P$ is the vanishing locus of a section of $\cO(4)$ in $\PP(1^4,2) := \PP(1,1,1,1,2)$. Notice that $P^\circ$ is \emph{not} cracked along the fan of $\PP^2$. To obtain a construction from a \emph{cracked} polytope we first embed $\PP(1^4,2)$ into $\PP^{10}$ via the linear system defined by sections of $\cO(2)$. Sections of $\cO(2)$ define the integral points of a polytope in $\ZZ^4$ given by the convex hull of the points given by the columns of the matrix
\[
\begin{array}{ccccc}
0&2&0&0&0 \\
0&0&2&0&0 \\
0&0&0&2&0 \\
0&0&0&0&1
\end{array}.
\]
The quartic equation $x_0x_1x_2x_3=y^2$ defines a projection of this polytope to the reflexive polytope $P$ with ID $427$. This polytope is self-dual, and we take the scaffolding of $P$ with shape $Z$ given by a crepant resolution of $X_P$, covering $P$ with a single strut. This scaffolding corresponds to the anti-canonical embedding of $X_P$ into $\PP^{10}$, which is the intersection of the image of the Veronese embedding of $\PP(1^4,2)$ with a (binomial) quadric. Deforming this quadric deforms $X_P$ to a general quartic hypersurface in $\PP(1^4,2)$.
\section{Constructions of Fano manifolds}
\label{sec:constructions}

There are $98$ Fano threefolds with very ample anti-canonical bundle. In the previous section we described constructions from cracked polytopes of the $15$ of these which have Picard rank one. We now explain constructions in the remaining $83$ cases. In particular, for each of these $83$ Fano threefolds $X$, we exhibit a fan $\Sigma$ and polytope $P$ cracked along $\Sigma$ such that -- for some full scaffolding of $S$ with shape $Z := \TV(\bar{\Sigma})$ -- the toric variety $X_P$ admits an embedded smoothing in $Y_S$ to $X$.

\subsection*{Examples from `Quantum periods for $3$-dimensional Fano manifolds'}

Explicit constructions of Fano threefolds are provided in \cite{CCGK}. The authors use these constructions to compute (part of) the $J$-function of each Fano threefold using either the Quantum Lefschetz principle, or the Abelian-non Abelian correspondence. In particular, each Fano threefold $X$ is exhibited either as a complete intersection in a weak Fano toric variety, or as the degeneracy locus of a map of homogeneous vector bundles.

\begin{pro}
	\label{pro:CCGK_examples}
	Fix a Fano threefold $X$, such that the model of $X$ in \cite{CCGK} describes $X$ as the vanishing locus of a section of a split vector bundle $\Lambda = L_1 \oplus \cdots \oplus L_c$ on a toric variety $Y$. In addition, we insist that the divisor
	\[
	-K_Y-L_1 - \cdots - L_c
	\]
	is ample. There is a reflexive polytope $P$, shape variety $Z$, and full scaffolding $S$ of $P$ such that $Y_S \cong Y$, and $X_P$ admits an embedded smoothing to $X$ in $Y$.
\end{pro}
\begin{proof}
	Tables~\ref{tbl:rank_2}, \ref{tbl:rank_3}, and~\ref{tbl:rank_4} list binomial equations cutting out toric varieties to which Fano varieties in the various families satisfying our hypotheses degenerate. The leading monomial in each case is square-free and defines a subset of the columns $C_i$ of the weight matrix listed in \cite{CCGK} for each $i \in \{1,\ldots,c\}$. In every case the sets $C_i$ are pairwise disjoint, and disjoint from a subset $C$ of columns which define a basis of $\Pic(Y)$. Reversing Construction~\ref{cons:scaff_weights}, we can obtain a scaffolding from the weight matrices given in \cite{CCGK} and the binomial expressions listed in Tables~\ref{tbl:rank_2}, \ref{tbl:rank_3}, and~\ref{tbl:rank_4}. The rank one complete intersection cases are listed in Table~\ref{tbl:rank_one}.

	It follows from \cite[Theorem~$1.1$]{P:Cracked}, and smoothness of $Y_S$, that the polytope $P^\circ$ is cracked along the fan determined by $Z$, and $S$ is full.
\end{proof}

\begin{eg}
	Consider a Fano threefold $X$ in the family \MM{2}{18}. $X$ is a double cover of $\PP^1 \times \PP^2$, branched in a divisor with bidegree $(2,2)$. The construction in \cite{CCGK} describes this Fano threefold as a hypersurface in the projectivisation of a rank $3$ split vector bundle on $\PP^2$.

	Consider the scaffolding $S$ with shape $Z=\PP^2$ illustrated in the left hand image in Figure~\ref{fig:2-18}. That is, $\bar{N} \cong \ZZ^2$, $N_U \cong \ZZ$, and $S = \{(D_1+D_2,0),(D_0+D_2,-1)\}$, where $D_i = \{x_i=0\}$ for homogeneous co-ordinates $(x_0:x_1:x_2)$ on $\PP^2$. This scaffolding exhibits $X_P$ as the hypersurface given by the vanishing locus of the binomial $zy_2x_3 - y_1^2x_1^2$, in the toric variety with weight matrix
	\[
	\begin{array}{cccccc}
	y_1 & x_1 & x_2 & x_3 & y_2 & z \\ \midrule
	1 & 0 & 0 & 0 & 1 & 1  \\
	0 & 1 & 1 & 1 & 0 & 1
	\end{array}
	\]
	such that the class $(2,1)$ is ample. Note that the weight matrix -- up to a permutation of the columns -- and stability condition $\omega = (2,1)$ are identical to those appearing in \cite[p.$40$]{CCGK}. Thus the general member of the linear system $\cO(2,2)$ is a Fano threefold in the family \MM{2}{18}.

	\begin{figure}
		\begin{tabular}{cc}
			\includegraphics[scale = 0.6]{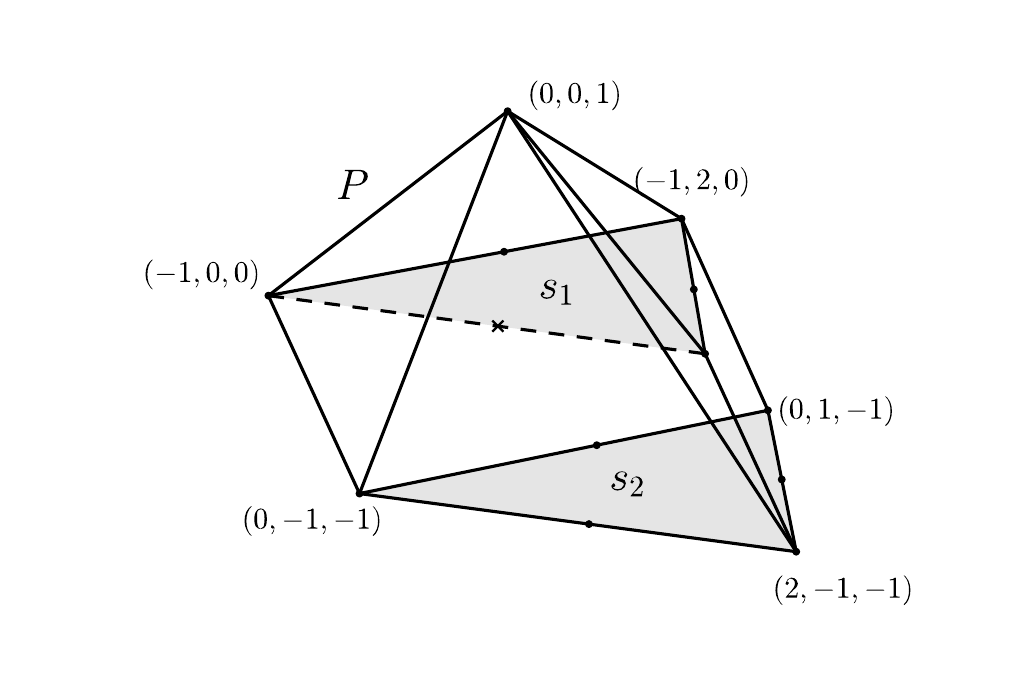}\newline &
			\includegraphics[scale = 0.5]{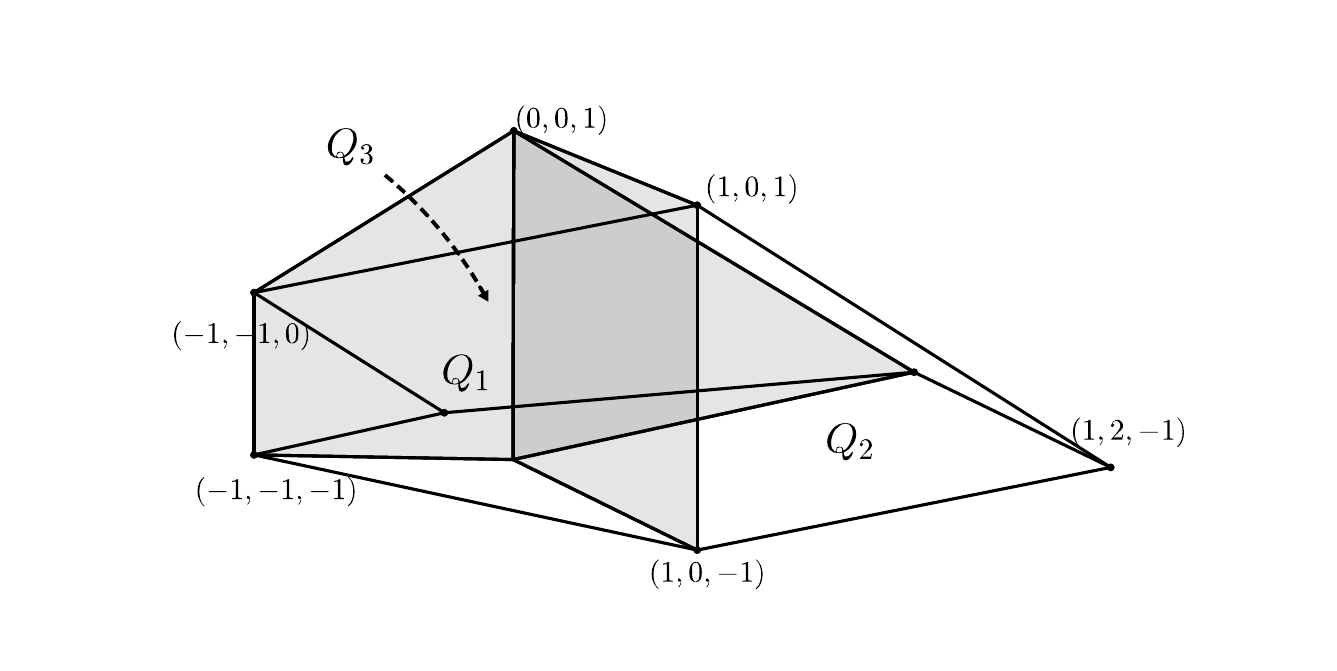}\newline \\
		\end{tabular}
		\caption{Constructing $\MM{2}{18}$ via Laurent inversion.}
		\label{fig:2-18}
	\end{figure}
\end{eg}

Of the $83$ Fano threefolds with very ample anti-canonical divisor and $b_2 \geq 2$, $67$ of the constructions given in \cite{CCGK} coincide with constructions from full scaffoldings on cracked polytopes. We summarise these constructions in Tables~\ref{tbl:rank_2}, \ref{tbl:rank_3}, and~\ref{tbl:rank_4}. The column \emph{Equations} in each table describes a generating set for the ideal in the homogeneous co-ordinate ring of the ambient variety $Y$ described in \cite{CCGK}. The first monomial of each binomial is always square-free, and may be used to identify columns of the weight matrix defined by $Y$. If $Y$ is a product of projective spaces the co-ordinates are not named in \cite{CCGK}, and we name these $x_0,\ldots,x_m$ for the first projective space factor $\PP^m$, $y_0,\ldots,y_n$ for the second, etc.

\begin{table}
\[
\def\arraystretch{2.0}
\begin{array}{ccc|ccc}
\textrm{Fano} & \textrm{Equations} & \textrm{Shape} & \textrm{Fano} & \textrm{Equations} & \textrm{Shape} \\\hline
\MM{2}{4} &  x_1y_1y_2y_3 -x_0y_0^3 & \PP^3 & \MM{2}{23} &  \begin{cases} xs_3s_4-s_0x_5\\ s_1s_2-s^2_0 \end{cases}  & \PP^2\times \PP^1\\
\MM{2}{5} & s_1xx_3x_4 - x_2^3 & \PP^3 & \MM{2}{24} & x_1y_1y_2-x_0y_0^2 & \PP^2 \\
\MM{2}{6} & x_1x_2y_1y_2 - x_0^2y_0^2 & \PP^3 & \MM{2}{25} & x_1y_1y_2 - x_0y_0^2 & \PP^2 \\
\MM{2}{7} & \begin{cases} x_1y_1y_2 - x_0y^2_0 \\ y_3y_4 - y_0^2 \end{cases} & \PP^1\times\PP^2 & \MM{2}{27} & x_iy_i-x_0y_0, i \in [2] & \PP^1\times \PP^1 \\
\MM{2}{9} & \begin{cases} x_1y_1 - x_0y_0 \\ x_2x_3y_2-x_0^2y_0 \end{cases}  & \PP^1\times\PP^2 & \MM{2}{28} & s_1s_2s_3x - s_0y & \PP^3 \\
\MM{2}{10} &  \begin{cases} x_4x_5-x_2^2 \\ xx_3s_1-x_2^2 \end{cases}  & \PP^1\times\PP^2 & \MM{2}{29} & x_3x_4-x^2_2 & \PP^1\\
\MM{2}{11} & s_0s_1xx_4-s_2x_3^2 & \PP^3 & \MM{2}{30} & s_1s_2x-s_0x_4 & \PP^2 \\
\MM{2}{12} & x_iy_i-x_0y_0, i \in [3]  & \PP^1\times\PP^1\times\PP^1 & \MM{2}{31} & s_2x_4-s_0x_3 & \PP^1 \\
\MM{2}{13} & \begin{cases} x_1y_1-x_0y_0 \\ x_2y_2-x_0y_0  \\ y_3y_4-y_0^2 \end{cases}& \PP^1\times\PP^1\times\PP^1 & \MM{2}{32} & x_1y_1-x_0y_0 & \PP^1 \\
\MM{2}{15} & s_0s_1s_2x-s_3^2x_4 & \PP^3 & \MM{2}{33} & - & pt \\
\MM{2}{16} & \begin{cases} s_0s_1x-s_2x_3 \\ x_4x_5-x_3^2 \end{cases} & \PP^2\times\PP^1 & \MM{2}{34} & - & pt \\
\MM{2}{18} & x_1y_1w-x_0^2y_0^2  & \PP^2 & \MM{2}{35} & - & pt \\
\MM{2}{19} &  \begin{cases} s_1x_5-s_0x_4 \\ xs_2s_3-s_0x_4\end{cases} & \PP^2\times\PP^1 & \MM{2}{36} & - & pt
\end{array}
\]
	\caption{Scaffolding constructions for Picard rank $2$ Fano threefolds.}
	\label{tbl:rank_2}
\end{table}

\begin{table}
\[\def\arraystretch{2.0}
	\begin{array}{ccc|ccc}
\textrm{Fano} & \textrm{Equations} & \textrm{Shape} & \textrm{Fano} & \textrm{Equations} & \textrm{Shape} \\\hline
\MM{3}{3} & x_1y_1z_1z_2-x_0y_0z_0^2 & \PP^3 & \MM{3}{19} & s_1xx_4 - x_3^2 & \PP^2 \\
\MM{3}{6} & x_2^2y_0-s_1xx_3y_1 & \PP^3 & \MM{3}{20} & s_1t_3 - s_0t_2 & \PP^1 \\
\MM{3}{7} & \begin{cases} x_1y_1z_1-x_0y_0z_0 \\ y_2z_2-y_0z_0 \end{cases} & \PP^2\times\PP^1 & \MM{3}{21} & y_1s - tx_0y_2^2 & \PP^1 \\
\MM{3}{8} & s_1xy_1y_2-x_2y_0^2 & \PP^3 & \MM{3}{22} & x_1s - ty_0^2 & \PP^1 \\
\MM{3}{9} & y_1xs_1s_2 - y_0^2 & & \MM{3}{23} & s_2v - s_1x_0u & \PP^1 \\
\MM{3}{10} & s_1t_3xy-x_4^2 & \PP^3 & \MM{3}{24} & x_2y_1 - s_0xy_0 & \PP^1 \\
\MM{3}{11} & s_1s_2xy_1-s_0x_3y_0 & \PP^3 & \MM{3}{25} & - & pt \\
\MM{3}{12} &  \begin{cases} s_3xy_1-x_1y_0 \\ x_2y_2-x_1y_0 \end{cases} & \PP^2\times\PP^1 & \MM{3}{26} & - & pt \\
\MM{3}{13} & \begin{cases} x_1y_1 - x_0y_0 \\ x_2z_1 - x_0z_0 \\ y_2z_2 - y_0z_0 \end{cases}  & \PP^1\times\PP^1\times\PP^1 & \MM{3}{27} & - & pt \\
\MM{3}{15} & s_1s_2y - s_0t_3z & \PP^2 & \MM{3}{28} & - & pt \\
\MM{3}{17} & x_1y_1z_1-x_0y_0z_0 & \PP^2 & \MM{3}{29} & - & pt \\
\MM{3}{18} & s_1xx_3y_0-s_0x_2y_1 & \PP^3 & \MM{3}{30} & - & pt \\
& & & \MM{3}{31} & - & pt
\end{array}
\]
		\caption{Scaffolding constructions for Picard rank $3$ Fano threefolds.}
	\label{tbl:rank_3}
\end{table}

\begin{table}
	\[\def\arraystretch{2.0}
	\begin{array}{ccc|ccc}
	\textrm{Fano} & \textrm{Equations} & \textrm{Shape} & \textrm{Fano} & \textrm{Equations} & \textrm{Shape} \\\hline
	\MM{4}{1} & x_1y_1z_1w_1 - x_0y_0z_0w_0 & \PP^3 & \MM{4}{7} & \begin{cases} y_1u_1-x_0u_0 \\ z_1u_2 - x_0u_0 \end{cases} & \PP^1\times\PP^1 \\
	\MM{4}{3} & y_0y_1 - s_0^2t_0^2x^2 & \PP^1 & \MM{4}{8} & x_2y_2 - s_0xt_0y & \PP^1 \\
	\MM{4}{4} & x_1y_1v - x_0y_0z_0u & \PP^2 & \MM{4}{9} & z_1u - x_0y_0v & \PP^1 \\
	\MM{4}{5} & x_3y_4 - x_2^2y^2 & \PP^1 & \MM{4}{10}-\MM{4}{13} & - & pt 
	\end{array}
	\]
	\caption{Scaffolding constructions for Picard rank $4$ Fano threefolds.}
	\label{tbl:rank_4}
\end{table}

We now provide constructions from cracked polytopes of the $15$ Fano threefolds whose construction in \cite{CCGK} is not directly related to a full scaffolding of a cracked polytope. In six cases (\MM{2}{14},  \MM{2}{17}, \MM{2}{20}, \MM{2}{21}, \MM{2}{22}, \MM{2}{26}) the corresponding construction in \cite{CCGK} does not describe the Fano threefold as a toric complete intersection. In the remaining nine cases (\MM{2}{8}, \MM{3}{1}, \MM{3}{4}, \MM{3}{5}, \MM{3}{14}, \MM{3}{16}, \MM{4}{2}, \MM{4}{6}, \MM{5}{1}) the construction given in \cite{CCGK} expresses the Fano threefold $X$ as the vanishing locus of a section of split vector bundle $\Lambda = L_1 \oplus \cdots \oplus L_c$ on a toric variety $Y$, such that $L := -K_Y-\sum_i{L_i}$ is nef but not ample. In the latter case the embedding cannot come from a scaffolding $S$, since Construction~\ref{cons:scaff_weights} uses $L$ to polarise the ambient space.

\begin{rem}
	The construction given in \cite[p.$58$]{CCGK} expresses varieties in the family \MM{3}{2} using a hypersurface for which $L$ is not ample. However, in the remarks on the construction given on \cite[p.$59$]{CCGK}, the authors describe a second construction using a toric variety $G$. This toric variety does coincide with a toric ambient space obtained from a full scaffolding of a cracked polytope.
\end{rem}

\begin{rem}
	Note that the numbering for the rank $4$ Fano threefolds replicates that in \cite{CCGK}, which differs from the original list of Mori--Mukai by the insertion of the family \MM{4}{2} which was omitted from the original classification (some lists instead append this family as \MM{4}{13}).
\end{rem}

\subsection*{Rank $2$, number $8$}
Varieties in the family $\MM{2}{8}$ are either, 
\begin{enumerate}
	\item\label{it:2-8a} the double cover of $B_7$ (the blow-up of $\PP^3$ at a point) with branch locus a member $B$ of $|-K_{B_7}|$ such that $B \cap D$ is non-singular, where $D$ is the exceptional divisor of the blow-up $B_7 \to \PP^3$, or;
	\item the specialisation of \eqref{it:2-8a} where $B \cap D$ is reduced but singular.
\end{enumerate}

We make use of the construction given in \cite{CCGK}, which embeds Fano threefolds in the family \MM{2}{8} as hypersurfaces of bi-degree $(2,4)$ in the toric variety $Y$, defined by the weight matrix
\[
\begin{matrix}
y & x_0 & z & x_1 & x_2 & x_3 \\
\midrule
1 & 1 & 1 & 0 & 0 & 0 \\
0 & 1 & 2 & 1 & 1 & 1 \\
\end{matrix}
\]
and a choice of stability condition in the chamber $\langle (0,1),(1,2)\rangle$. The coincidence of these two constructions is proved in \cite[p$.31$]{CCGK}.

Consider the scaffolding $S = \{D\}$ of the reflexive polytope $P$ with PALP ID $3262$, with shape $Z = \PP^1\times dP'_5$. Here we take $dP'_5$ to be the blow up of $\PP^1\times \PP^1$ at three of its torus invariant points, and $D \in |-K_Z|$ is the toric boundary of $Z$.

This scaffolding corresponds to the anti-canonical embedding $X_P \to \PP^9$, see Example~\ref{eg:anti_canonical}. To exhibit an explicit smoothing in this embedding we consider another scaffolding of $P$ -- with shape $Z' = \PP^2\times\PP^1$ -- shown in Figure~\ref{fig:2-8}. Note that $P^\circ$ is not cracked along the fan determined by $Z'$. The scaffolding $S'$ defines an embedding $X_P \to Y_{S'}$ where $Y_{S'}$ is the toric variety defined by weight matrix
\[
\begin{matrix}
y & x_0 & z_0 & z_1 & x_1 & x_2 & x_3 \\
\midrule
1 & 1 & 1 & 1 & 0 & 0 & 0 \\
0 & 1 & 2 & 2 & 1 & 1 & 1 \\
\end{matrix}
\]
and stability condition $\omega = (1,2)$ -- note that $\omega$ is contained in a wall. The toric variety $X_P$ is the vanishing locus of a section of $E := \cO(1,2)\oplus\cO(2,4)$.

\begin{figure}
	\includegraphics{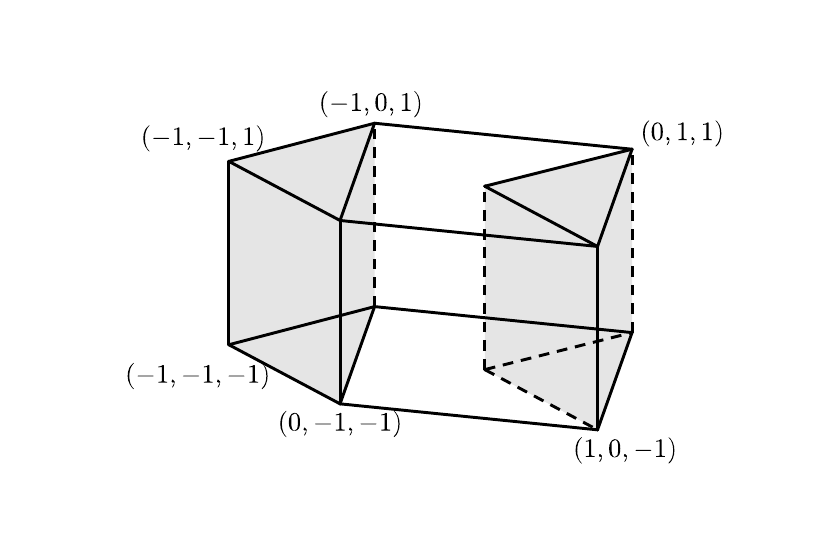}
	\caption{Scaffolding used to construct \MM{2}{8}}
	\label{fig:2-8}
\end{figure}

\begin{lem}
	\label{lem:2-8_part_1}
	The vanishing locus of a general section of $E$ is a Fano threefold \MM{2}{8}.
\end{lem}
\begin{proof}
	General sections of $E$ do not vanish at the torus invariant point defined by the vanishing of all co-ordinates except $z_1$. There is a projection from this point to the toric variety $Y'$, the toric variety defined by the same weight matrix as $Y$, but stability condition $\omega = (1,2)$. The wall spanned by $(1,2)$ is a flipping wall, and the birational transformation induced by crossing this wall is given by (the cone on) a Pachner move in the fan determined by $Y$. The intermediate variety has the non-$\QQ$ factorial point given by the vanishing of all homogeneous co-ordinates (labelled as for $Y$) except $z$. The image of the vanishing locus $X$ of a general section of $E$ in $Y'$ misses this singularity. Hence the resolution of $Y'$ induced by moving the stability condition from $(1,2)$ into the chamber $\langle (1,2),(0,1) \rangle$ restricts to an isomorphism of $X$, and the result follows from \cite[p.$31$]{CCGK}.
\end{proof}

Consider the embedding $\varphi_{\cO(1,2)} \colon Y_{S'} \to \PP(H^0(Y_{S'},\cO(1,2)))^\star = \PP^{10}$. Composing $\varphi_{\cO(1,2)}$ with the embedding $\iota \colon X_P \to Y_{S'}$, the pull-back of the line bundle $\cO_{\PP^{10}}(1)$ is the anti-canonical class on $X_P$ by adjunction. Moreover $X_P$ is the intersection of $\varphi_{\cO(1,2)}(Y_{S'})$ with a quadric and a hyperplane in $\PP^{10}$. In particular, restricting to this hyperplane, we obtain the anti-canonical embedding of $X_P$ in $\PP^9$. Restricting to members of a general pencil of hyperplanes -- and intersecting with a general pencil of quadrics -- we see that $X_P$ deforms in $\PP^9$ to a variety in the family \MM{2}{8}.

\subsection*{Rank $2$, number $14$}
This example is the first of a sequence of examples -- along with $\MM{2}{20}$, $\MM{2}{22}$, and $\MM{2}{26}$ -- to make use of polytopes cracked along the fan of $Z := dP_7$. The corresponding embeddings are defined using  the five $4\times 4$ Pfaffians of a $5\times 5$ matrix of polynomials in the homogeneous co-ordinate ring of a toric variety. Varieties in the family $\MM{2}{14}$ are the blow up of $B_5$ (a three dimensional linear section of $\Gr(2,5)$) in an elliptic curve which is the intersection of two hyperplane sections.

Consider the polytope $P$ with PALP ID $3027$ together with the scaffolding with shape $Z$ displayed in Figure~\ref{fig:2-14}. We have that $\bar{N} \cong \ZZ^2$, $N_U \cong \ZZ$, and $S = \{(0,1),(D,0),(D,-1)\}$, where $D$ is the toric boundary of $Z = dP_7$.

\begin{figure}
	\includegraphics{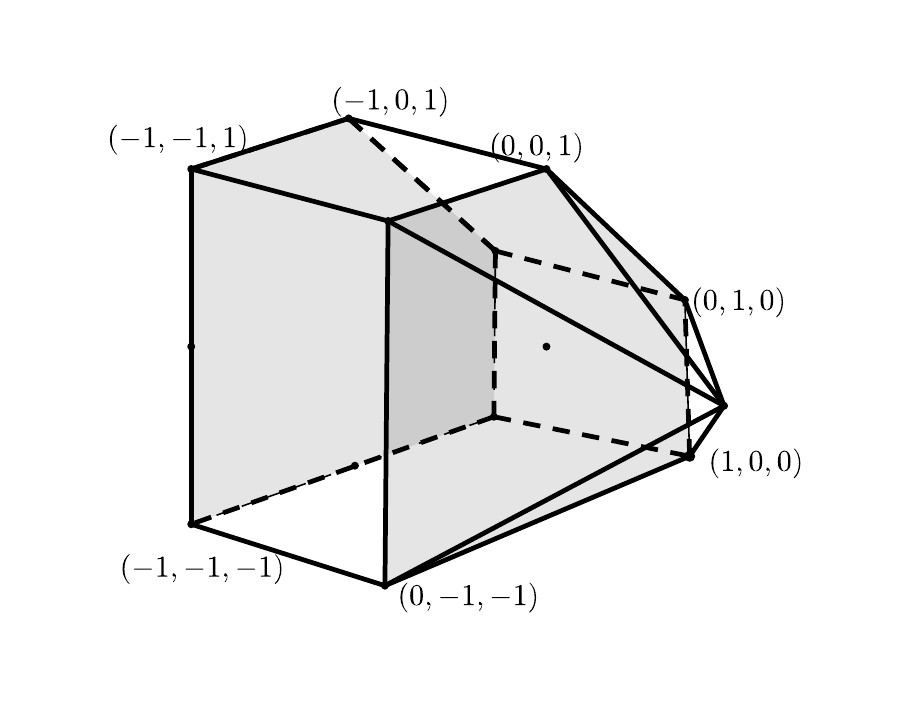}
	\caption{Scaffolding used to construct $\MM{2}{14}$}
	\label{fig:2-14}
\end{figure}

The variety $Y_S$ is determined by the weight matrix
\[
\begin{matrix}
x_0 & x_1 & x_2 & x_3 & x_4 & x_5 & x_6 & y \\
\midrule
1 & 1 & 1 & 1 & 1 & 1 & 1 & 0 \\
0 & 0 & 1 & 1 & 1 & 1 & 1 & 1 \\
\end{matrix}
\]
and stability condition $\omega = (2,1)$. The variety $Y_S$ is consequently the blow up of $\PP^6$ in a codimension $2$ linear subspace. The ideal of $X_P$ in $Y_S$ is obtained by homogenizing the $4\times 4$ Pfaffians of the skew-symmetric matrix
\[
\begin{pmatrix}
x_0y & x_1 & x_2 & x_0y \\
& 0 & x_3 & x_4 \\
& & x_0y & x_5 \\
& & & 0
\end{pmatrix}.
\]

\begin{figure}
	\includegraphics[scale=0.7]{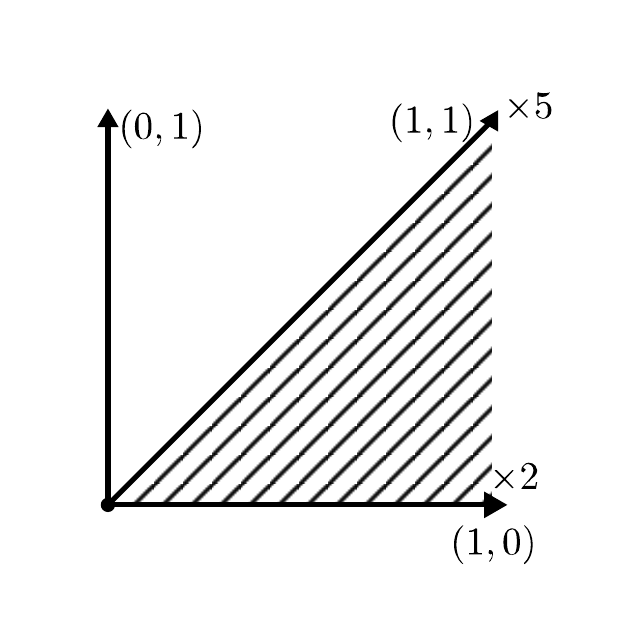}
	\caption{Secondary fan for the variety $Y_S$ used in the construction of $\MM{2}{14}$.}
	\label{fig:sec_fan2-14}
\end{figure}

Consider the contraction $Y_S \to \PP^6$, and observe that the intersection of the image of $X_P$ with the centre $V := \{x_0=x_1=0\}$ is a cycle of five $(-1)$-curves. Replacing the two $0$ entries with general linear forms this cycle of $(-1)$-curves becomes a (codimension $3$) non-singular curve of genus one; blowing up $V$ produces a flat family deforming $X_P$ to a Fano threefold in the family $\MM{2}{14}$.

\subsection*{Rank $2$ number $17$}

Varieties in the family $\MM{2}{17}$ are the blow up of a quadric threefold in an elliptic curve of degree $5$. We consider the polytope $P$ with PALP ID $1527$, together with the scaffolding $S$ shown in Figure~\ref{fig:2-17} using the shape variety $Z = \PP^1 \times dP_7$.

\begin{figure}
	\includegraphics[scale=1.2]{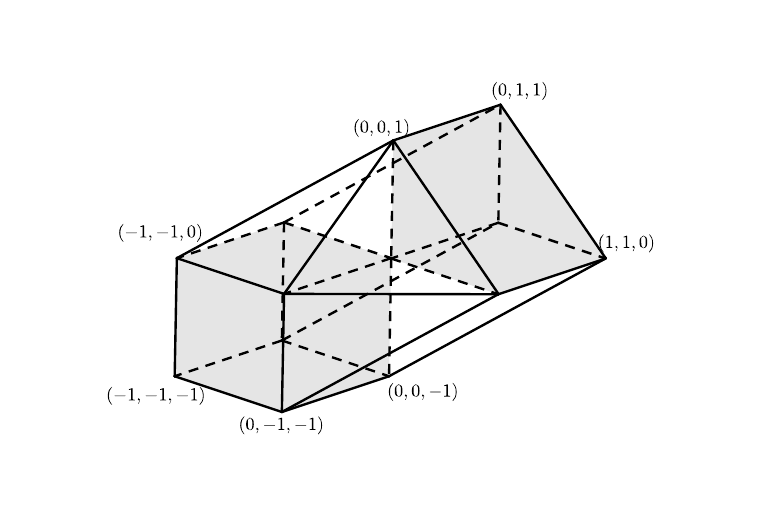}
	\caption{Scaffolding used to construct $\MM{2}{17}$}
	\label{fig:2-17}
\end{figure}

The scaffolding $S$ determines the toric variety $Y_S \cong \PP^4 \times \PP^3$. Letting $x_0,\ldots,x_4$ and $y_0,\ldots,y_3$ denote homogeneous co-ordinates on the respective projective space factors, $X_P$ is the vanishing locus of the binomial $x_0y_0=x_1y_1$, and the five $4\times 4$ Pfaffians of the skew-symmetric matrix
\[
\begin{pmatrix}
y_0 & y_2 & y_3 & y_0 \\
& tf_1 & x_2 & x_4 \\
& & x_0 & x_3 \\
& & & tf_2
\end{pmatrix},
\]
where $t=0$ and $f_i$ are general linear equations in $x_0,\ldots,x_4$. One of these five Pfaffians describes the threefold $x_2x_3-x_0x_4=0$ in $\PP^4$, while the other $4$ equations have bidegree $(1,1)$. It is shown in \cite[p.$38$]{CCGK} that varieties in the family \MM{2}{17} may be obtained as the vanishing loci of general sections of the bundle
\[
E := (S^\star \boxtimes \cO_{\PP^3}(1))\oplus(\det S^\star \boxtimes \cO_{\PP^3}(1)) \oplus (\det S^\star \boxtimes \cO_{\PP^3})
\]
on the variety $\Gr(2,4) \times \PP^3$. The Grassmannian $\Gr(2,4) \subset \PP^5$ is a quadric fourfold, while sections of the line bundle $\det S^\star$ define hyperplane sections in $\PP^5$. Moreover, the binomial $x_0y_0=x_1y_1$ defines a section of the bundle obtained by pulling back $(\det S^\star \boxtimes \cO_{\PP^3}(1))$ to the product of a hyperplane section in $\PP^5$ with $\PP^3$. We claim that the remaining four Pfaffian equations define a section of the pull-back of $(S^\star \boxtimes \cO_{\PP^3}(1))$ to this hyperplane section. Representing a point in $\Gr(2,4)$ as the row-space of a $2 \times 4$ matrix
\[
M = 
\begin{pmatrix}
y_{1,1} & y_{1,2} & y_{1,3} & y_{1,4} \\
y_{2,1} & y_{2,2} & y_{2,3} & y_{2,4}
\end{pmatrix},
\]
a section of the bundle $S^\star$ is determined by a vector $z = (z_1,z_2,z_3,z_4) \in \CC^4$, and this section vanishes precisely when $z$ lies in the row space of $M$. This happens when the maximal minors of the matrix
\[
\bar{M} =
\begin{pmatrix}
	z_1 & z_2 & z_3 & z_4 \\
	y_{1,1} & y_{1,2} & y_{1,3} & y_{1,4} \\
	y_{2,1} & y_{2,2} & y_{2,3} & y_{2,4}
\end{pmatrix}
\]
vanish. Writing the $2\times 2$ minors of $M$ (the Pl\"{u}cker co-ordinates) as $x_0,\ldots,x_5$ we have that sections of $S^\star$ are defined by four equations of degree $1$ in the variables $\{x_i : i \in \{0,\ldots,5\}\}$ and constants $\{z_i : i \in \{1,\ldots,4\}\}$. Replacing each $z_i$ with the homogeneous co-ordinate $y_{i-1}$ we recover the $4$ remaining Pfaffian equations found above, up to a linear relation eliminating $x_5$. That is, $X_P$ admits an embedded flat deformation to a variety in the family \MM{2}{17}.

\subsection*{Rank $2$ number $20$}

Varieties in the family $\MM{2}{20}$ are the blow up of $B_5$ (a three dimensional linear section of $\Gr(2,5)$) in a twisted cubic. Consider the polytope $P$ with PALP ID $1909$ together with the scaffolding with shape $Z = dP_7$ displayed in Figure~\ref{fig:2-20}.

\begin{figure}
	\includegraphics{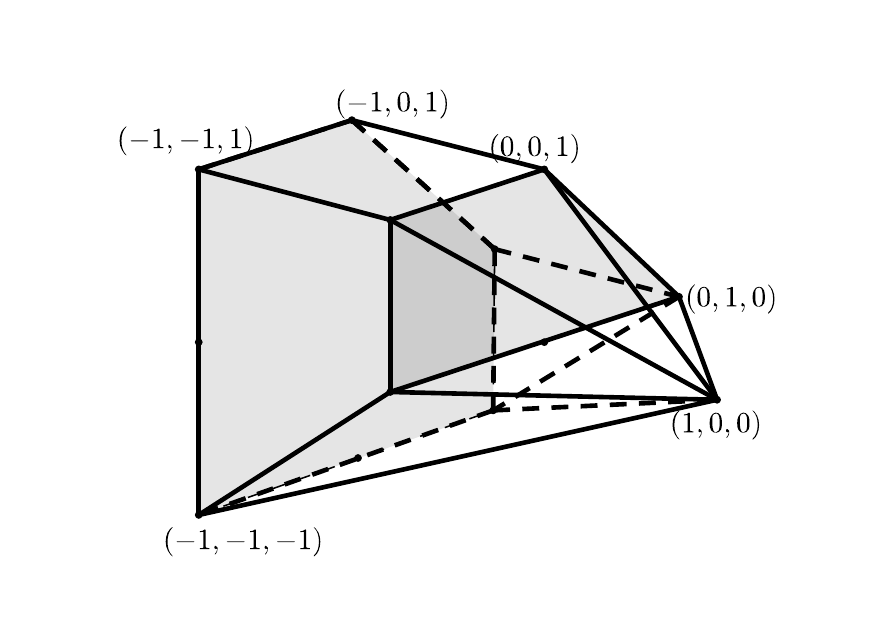}
	\caption{Scaffolding used to construct $\MM{2}{20}$}
	\label{fig:2-20}
\end{figure}

The corresponding toric variety $Y_S$ is isomorphic to $Bl_{\PP^3}\PP^6$. Moreover, the variety $X_P$ is the blow up of the vanishing locus of the five $4\times 4$ Pfaffians of
\[
\begin{pmatrix}
x_0 & x_1 & x_2 & x_0 \\
& 0 & x_3 & x_4 \\
& & x_0 & x_5 \\
& & & 0
\end{pmatrix},
\]
where $x_0,\ldots,x_6$ are homogeneous co-ordinates on $\PP^6$, in the locus $\{x_0=x_1=x_6\}$. Note that the ideal $x_2x_4 = x_3x_5 = x_2x_5 = 0$ defines a (degenerate) twisted cubic. Replacing the two zero entries in the above matrix with general homogeneous elements of degree one we obtain a flat deformation of $X_P \hookrightarrow Y_S$ to the blow up of $B_5$ in a twisted cubic.

\subsection*{Rank $2$ number $21$}

Varieties in the family $\MM{2}{21}$ are the blow up of a quadric threefold in a rational curve of degree $4$. These are shown in \cite[p.$43$]{CCGK} to be zero loci of sections of the vector bundle
\[
E = (S^\star\boxtimes \cO_{\PP^4}(1))^{\oplus 2}\oplus (\det S^\star \boxtimes \cO_{\PP^4})
\]
on $\Gr(2,4)\times \PP^4$. Consider the polytope with PALP ID $702$, with the scaffolding shown in Figure~\ref{fig:2-21}. This scaffolding has shape $Z = Z_{12}$, the shape used in the construction of Fano threefolds in the family $V_{12}$. The ambient space $Y_S$ defined by this scaffolding is isomorphic to $\PP^4\times\PP^4$ with co-ordinates $x_0,\ldots,x_4$ and $y_0,\ldots,x_4$ respectively. The equations cutting out $X_P$ in $Y_S$ can be read off as relations between labelled lattice points in Figure~\ref{fig:rels2-21}. In particular if $u_1+v_1 = u_2+v_2$, where $u_i$ and $v_i$ are lattice points labelled with variables $z_i$ and $w_i$ for each $i \in \{1,2\}$, points in $X_P$ satisfy the equation $z_1w_1 = z_2w_2$. There are nine such binomial equations, which can be written as the $4\times 4$ Pfaffians of the following pair of matrices (setting $t=0$),

\begin{figure}
	\includegraphics{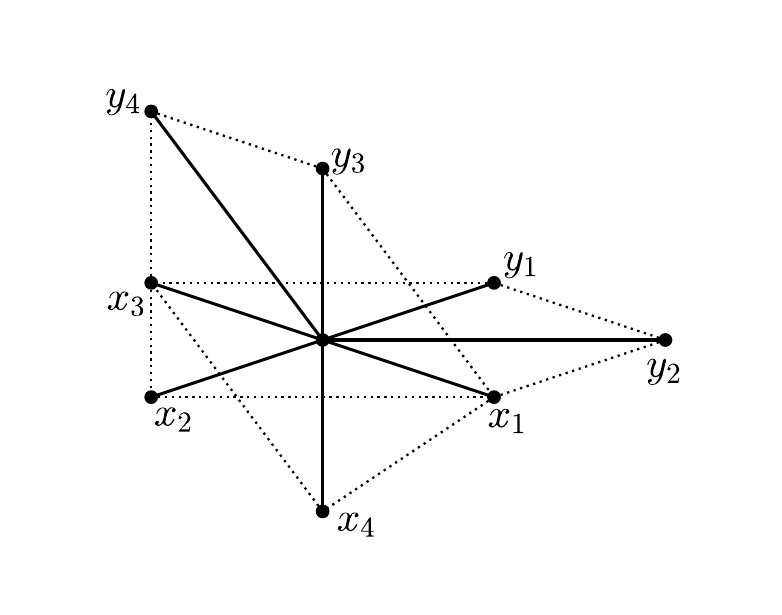}
	\caption{Ray generators of $Z_{12}$.}
	\label{fig:rels2-21}
\end{figure}

\begin{align*}
\begin{pmatrix}
y_0 & ty_3 & y_2 & y_1 \\
& x_2 & x_0 & x_3 \\
& & x_1 & x_0 \\
& & & tx_4
\end{pmatrix},
&&
\begin{pmatrix}
y_0 & ty_1 & y_3 & y_4 \\
& x_4 & x_1 & x_0 \\
& & x_0 & x_3 \\
& & & -tx_2
\end{pmatrix}.
\end{align*}

Note that these matrices share the Pfaffian $x_1x_3 - x_0^2 + tx_2x_4$, which defines a toric degeneration of a quadric threefold.

\begin{figure}
	\includegraphics{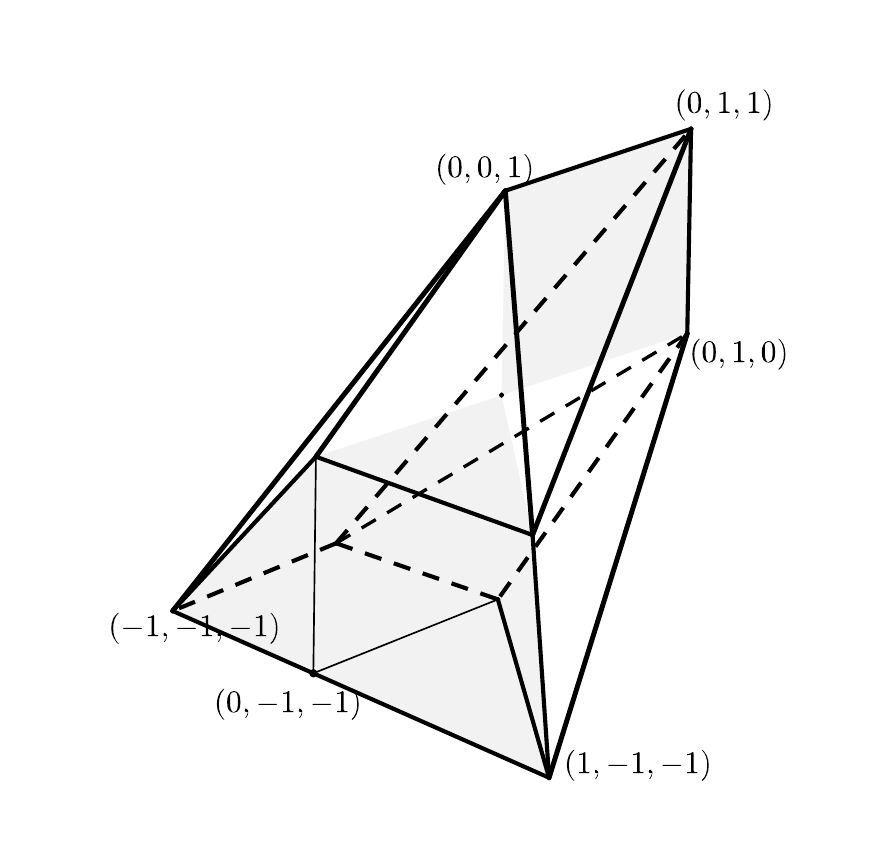}
	\caption{Scaffolding used to construct $\MM{2}{21}$}
	\label{fig:2-21}
\end{figure}

Following the treatment of the variety \MM{2}{17}, we observe that each set of five Pfaffian equations defines a section of (the pullback to a hyperplane section of) $S^\star \boxtimes \cO_{\PP^4}(1)$. Thus the general member of the family given by the set of $9$ Pfaffian equations is isomorphic to a Fano threefold in the family \MM{2}{21}.

\subsection*{Rank $2$ number $22$}
Varieties in the family $\MM{2}{22}$ are the blow up of $B_5$ in a conic. Consider the polytope $P$ with PALP ID $1856$, and the scaffolding with shape $Z = dP_7$ displayed in Figure~\ref{fig:2-22}. The variety $Y_S$ is the blow up of $\PP^6$ in a plane; the toric variety determined by the weight matrix
\[
\begin{matrix}
y & x_0 & x_1 & x_2 & x_3 & x_4 & x_5 & x_6 \\
\midrule
1 & 0 & 0 & 0 & 0 & 1 & 1 & 1 \\
0 & 1 & 1 & 1 & 1 & 1 & 1 & 1 \\
\end{matrix}
\]
and stability condition $\omega = (1,2)$. $X_P$ is cut out by the five $4\times 4$ Pfaffians of
\[
\begin{pmatrix}
x_0 & tf_{0,1} & x_2 & x_3 \\
& x_4 & x_5 & x_0y \\
& & x_0y & x_6 \\
& & & tf_{1,1}
\end{pmatrix}
\]
where $f_{i,j}$ is a generic polynomial of bi-degree $(i,j)$, and $t=0$. The ambient variety $Y_S$ is obtained from $\PP^6$ with co-ordinates $x_0,\ldots,x_6$ by blowing up the plane $\Pi := \{x_0=x_1=x_2=x_3=0\}$. The Pfaffian equations defining $X_P$ pull back to the single equation $x_5x_6=tx_4f_{1,1}$ on this locus. Hence, for general values of $t$, the equations define the blow up of $B_5$ (cut out by $5$ Pfaffian equations in $\PP^6$) in a non-degenerate conic.

\begin{figure}
	\includegraphics{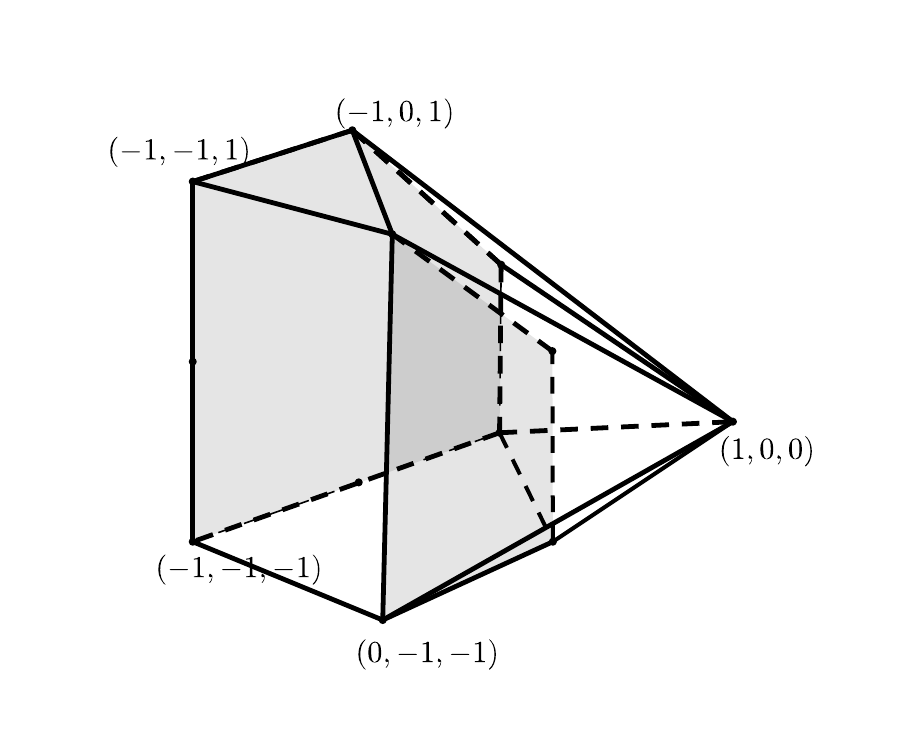}
	\caption{Scaffolding used to construct $\MM{2}{22}$}
	\label{fig:2-22}
\end{figure}

\subsection*{Rank $2$ number $26$}

Varieties in the family $\MM{2}{26}$ are the blow up of $B_5$ in a line. Consider the polytope $P$ with PALP ID $1433$ and scaffolding with shape $Z = dP_7$ displayed in Figure~\ref{fig:2-26}. The variety $Y_S$ is the blow up of $\PP^6$ in the line with homogeneous co-ordinates $\{x_4,x_5\}$. Consider the one-parameter family
\[
\begin{pmatrix}
x_0 & x_1 & x_2 & x_0 \\
& tf_1 & x_3 & x_4 \\
& & x_0 & x_5 \\
& & & tg_1
\end{pmatrix},
\]
where $f_1$ and $g_1$ are general linear forms with no terms in $x_4$ or $x_5$. Varying $t$, this family contains the line with co-ordinates $x_4$ and $x_5$ for all values of $t$. Blowing up this line we obtain a flat family embedded in $Y_S \times \AA^1_t$ with central fibre $X_P$, and general fibre a Fano threefold in the family $\MM{2}{26}$.

\begin{figure}
	\includegraphics{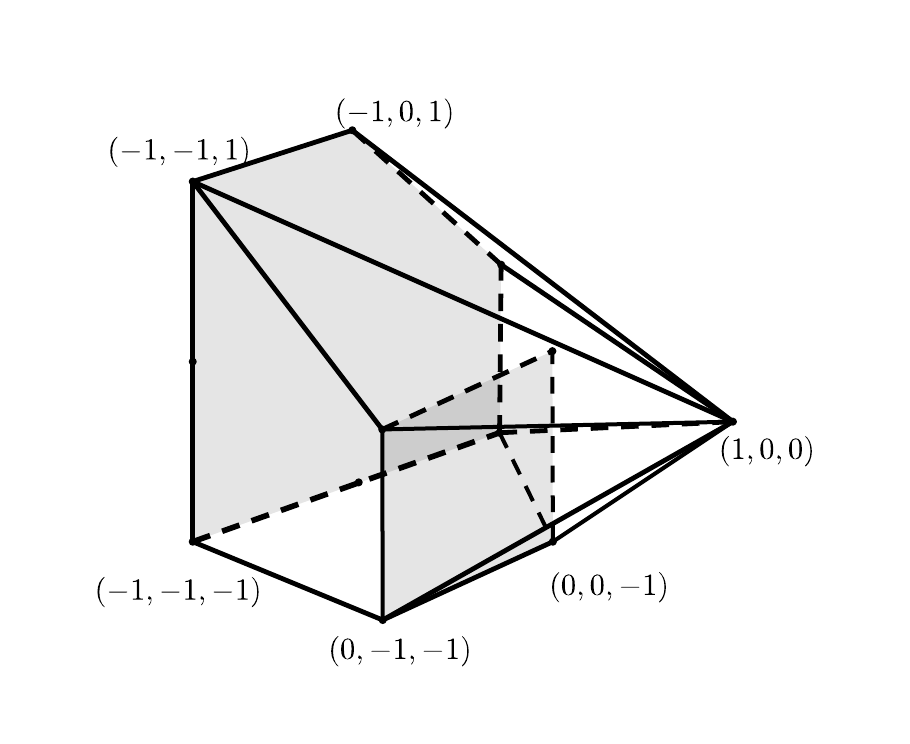}
	\caption{Scaffolding used to construct $\MM{2}{26}$}
	\label{fig:2-26}
\end{figure}

\subsection*{Rank $3$ number $1$}

Varieties in the family $\MM{3}{1}$ are double covers of $\PP^1\times\PP^1\times\PP^1$ branched along a divisor of tri-degree $(2,2,2)$. Our treatment of this family is similar to that of \MM{2}{8}. Consider the Fano polygon $P$ with PALP ID $3874$, illustrated in Figure~\ref{fig:3-1}. We give $P$ the `anti-canonical' scaffolding; covering $P$ with the polyhedron of sections of the toric boundary on the shape variety $Z = \PP^1 \times dP_6$. This scaffolding reproduces the anti-canonical embedding $X_P \to \PP^8$, see Example~\ref{eg:anti_canonical}. We exhibit an explicit smoothing by factoring the anti-canonical embedding through a map to a toric variety obtained from a non-full scaffolding of $P$. Figure~\ref{fig:3-1} shows a scaffolding $S'$ of $P$ with shape $Z' := \PP^1\times \PP^2$. The scaffolding $S'$ consists of three elements, and defines the toric variety $Y_{S'}$ with weight matrix

\[
\begin{matrix}
x_0 & y_0 & z_0 & x_1 & y_1 & z_2 & w_0 & w_1 \\
\midrule
1 & 0 & 0 & 1 & 0 & 0 & 1 & 1 \\
0 & 1 & 0 & 0 & 1 & 0 & 1 & 1 \\
0 & 0 & 1 & 0 & 0 & 1 & 1 & 1 \\
\end{matrix}
\]
and stability condition $\omega = (1,1,1)$. The hypersurface $X_P$ is the vanishing locus of the binomial $w_0w_1 = x_0^2y_0^2z_0^2$, a section of the line bundle $L_1$ with tri-degree $(2,2,2)$ -- and $x_1y_1z_1 = x_0y_0z_0$ -- a section of the line bundle $L_2$ tri-degree $(1,1,1)$. Note that the variety $Y_{S'}$ is not $\QQ$-factorial along the line on which $x_0 = y_0 = z_0 = x_1 = y_1 = z_2 = 0$. General linear sections though this non-isolated singularity are isomorphic to the affine cone $V$ over $\PP^1\times\PP^1\times\PP^1$, polarised by the line bundle of tri-degree $(1,1,1)$.

Consider a general section $s$ of $E = L_1\oplus L_2$, and its vanishing locus $X$. Projecting away from the point at which all co-ordinates except $w_0$ vanish, $X$ is an isomorphism onto its image in a toric variety $Y'$. The variety $F$ which appears in the construction in \cite[p.57]{CCGK} is obtained from the variety $Y'$ by a making one of the three possible small resolutions of the singularity $V$. Since the variety $X$ does not intersect the singular locus of $Y'$ this resolution restricts to an isomorphism of $X$. The rest of the example follows our treatment of the family \MM{2}{8}: the complete linear system determined by $L_2$ defines an embedding $Y_{S'} \to \PP^{9}$ and varying a quadric section in the anti-canonical embedding of $Y_{S'}$ smooths $X_P$.

\begin{figure}
	\includegraphics{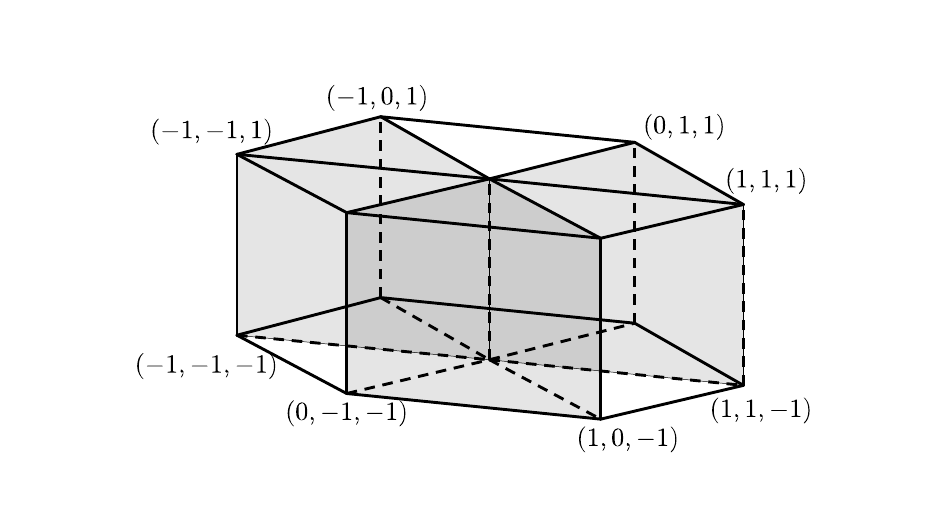}
	\caption{Scaffolding used to construct $\MM{3}{1}$}
	\label{fig:3-1}
\end{figure}

\subsection*{Rank $3$ number $4$}

Fano threefolds in this family are obtained by blowing up the fibre of the projection map $X_{\MM{2}{18}} \to \PP^2$, where $X_{\MM{2}{18}}$ is a double cover of $\PP^2\times \PP^1$ branched in a divisor of bidegree $(2,2)$. In \cite[p.$60$]{CCGK} it is shown that varieties in this family may be obtained as hypersurfaces of tri-degree $(2,2,2)$ contained in the toric variety $Y$ defined by the weight matrix
\[
\begin{matrix}
x_0 & x_1 & y & z & t_0 & t_1 & w \\
\midrule
1 & 1 & 1 & 0 & 0 & 0 & 1 \\
0 & 0 & 1 & 1 & 0 & 0 & 1 \\
0 & 0 & 0 & 0 & 1 & 1 & 1 \\
\end{matrix}
\]
together with a stability condition $\omega$ in the chamber $\langle (1,0,0),(1,1,0),(1,1,1)\rangle$. We compare these toric hypersurfaces to the threefolds obtained by scaffolding the polytope $P$ with PALP ID $2602$ shown in Figure~\ref{fig:3-4}. This scaffolding has shape $Z = \PP^1\times \PP^1$, and hence defines a codimension $2$ toric complete intersection in the toric variety $Y_S$ with weights:
\[
\begin{matrix}
a_0 & a_1 & b_0 & b_1 & b_2 & c_0 & c_1 \\
\midrule
1 & 1 & 1 & 1 & 1 & 0 & 0 \\
0 & 0 & 1 & 1 & 1 & 1 & 1 \\
\end{matrix}
\]
and stability condition $\omega = (2,1)$. Let $X$ be the vanishing locus of a general section of the vector bundle $E := \cO(1,2)\oplus\cO(2,2)$. Note that the line bundle $\cO(1,2)$ is not nef on $Y_S$. We define a Segre type map $\phi \colon Y \to Y_S$, setting
\[
\phi \colon (x_0,x_1,y,z,t_0,t_1,w) \mapsto (x_0,x_1,w,yt_0,yt_1,zt_0,zt_1).
\]
It is easily verified that this map is homogeneous, and that $\phi^\star \colon \Pic(Y_S) \to \Pic(Y)$ is given by the matrix $\begin{pmatrix} 1&0&0\\0&1&1\end{pmatrix}^T$. Hence the stability condition $\omega = (2,1)$, is mapped into the wall spanned by $(1,0,0)$ and $(1,1,1)$. Let $Y'$ be the toric variety defined by weight matrix $M$ and stability condition $(2,1,1)$. 

\begin{figure}
	\includegraphics{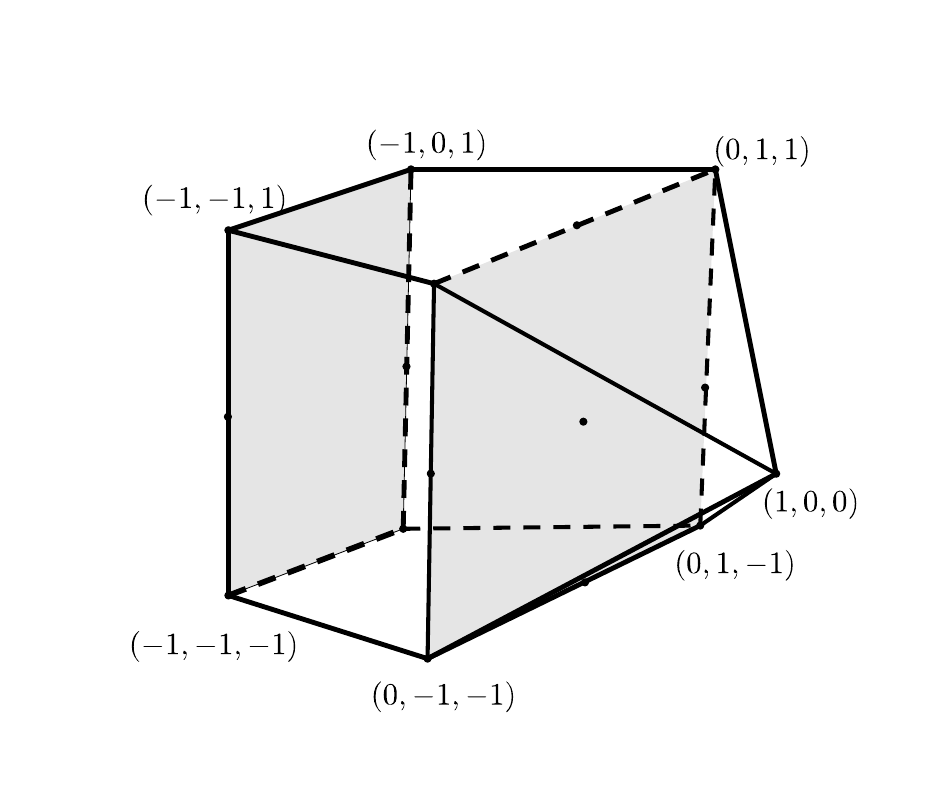}
	\caption{Scaffolding used to construct \MM{3}{4}}
	\label{fig:3-4}
\end{figure}

\begin{lem}
	\label{lem:3-4}
	$Y$ is a small resolution a non-isolated singularity of $Y'$ which is disjoint from the divisor $\{w=0\}$. 
\end{lem}
\begin{proof}
	There is a morphism $\pi \colon Y_S \to \PP_1$ expressing $Y_S$ as a $\PP^4$ bundle over $\PP^1$, with co-ordinates $(a_0:a_1)$. Similarly, $Y$ and $Y'$ admit projections to the $\PP^1$ with co-ordinates $(x_0:x_1)$. Each of these projections commute with the inclusion $\iota \colon Y' \hookrightarrow Y_S$. Given a point $a \in \PP^1$, the intersection $\pi^{-1}(a) \cap \iota(Y')$ is isomorphic to the projective closure of the (affine) ODP singularity in $\PP^4$ with co-ordinates $(b_0:b_1:b_2:c_0:c_1)$. The (smooth) variety $Y$ is obtained by making one of the two possible small resolutions of this line of conifold singularities. Note however that, for any fibre of $\pi$, the divisor $\{b_0=0\}$ is disjoint from the singular locus of $Y'$. Since $\iota^\star b_0 = w$, the locus $w=0$ is disjoint from the singular locus of $Y'$.
\end{proof}

Note that $Y'$ is a hypersurface in the class $\cO(1,2)$, cut out by $\det\begin{pmatrix}b_1 & c_0 \\ b_2 & c_1\end{pmatrix}$. Moreover, we have that $\phi^\star(2,2) = (2,2,2)$; hence, by Lemma~\ref{lem:3-4}, any hypersurface cut out by a member of the linear system $(2,2,2)$ on $Y$ is the vanishing locus of a section of $E$ on $Y_S$.

\subsection*{Rank $3$ number $5$}

It was shown in \cite[p.$62$]{CCGK} that varieties in the family \MM{3}{5} are codimension $2$ complete intersections in the toric variety $Y$, determined by the weight matrix
\[
\begin{matrix}
x_0 & x_1 & y_0 & y_1 & y_2 & z_0 & z_1 & t\\
\midrule
1 & 1 & 0 & 0 & 0 & 1 & 1 & 0 \\
0 & 0 & 1 & 1 & 1 & 1 & 1 & 0 \\
0 & 0 & 0 & 0 & 0 & 1 & 1 & 1
\end{matrix}
\]
and a stability condition in the chamber $\langle (1,0,0),(0,1,0),(1,1,1)\rangle$. Varieties $X$ in the family \MM{3}{5} are obtained as zero loci of sections of the bundle $\cO(1,2,1)^{\oplus 2}$. The secondary fan for $Y$ is illustrated in Figure~\ref{fig:sec_fan3-5}. Consider the scaffolding of the polytope $P$ with PALP ID $1836$ shown in Figure~\ref{fig:3-5}. The variety $Y_S$ is determined by the weight matrix
\[
\begin{matrix}
a_0 & a_1 & b_0 & b_1 & b_2 & c_0 & c_1 \\
\midrule
1 & 1 & 0 & 0 & 0 & 1 & 1 \\
0 & 0 & 1 & 1 & 1 & 1 & 1 \\
\end{matrix}
\]

\begin{figure}
	\includegraphics{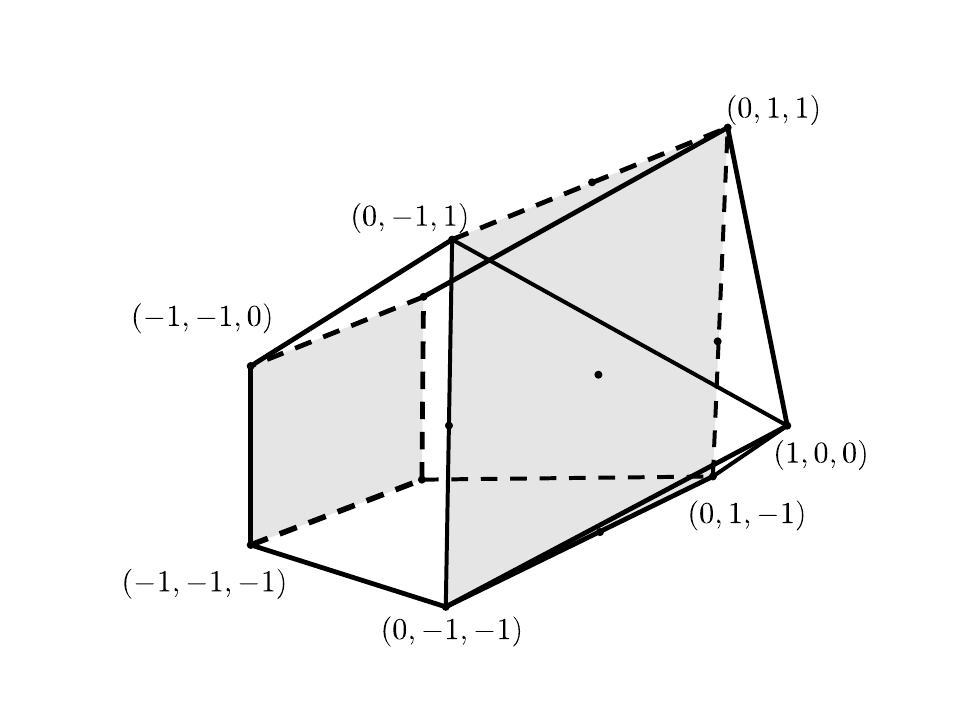}
	\caption{Scaffolding used to construct \MM{3}{5}}
	\label{fig:3-5}
\end{figure}

\noindent and stability condition $(2,1)$. The toric variety $X_P$ is cut out of $Y_S$ by a pair of binomial sections of $\cO(1,2)$. Observe that the linear system $(1,2)$ is not nef on $Y_S$, and has base locus $B = \{b_0=b_1=b_2=0\}$. We claim that $Y$ is obtained from $Y_S$ by blowing up $B$. It is clear that the weight matrix defining $Y$ is the same as the defining the toric variety $\Bl_BY_S$. Moreover, the map defined by setting
\[
\phi \colon (x_0, x_1, y_0, y_1, y_2, z_0,z_1, t) \mapsto (x_0t, x_1t, y_0,y_1,y_2,z_1,z_2)
\]
has pull-back defined by the matrix
\[
[\phi^\star] = \begin{pmatrix}
1 & 0 \\
0 & 1 \\
1 & 0
\end{pmatrix}.
\]

\begin{figure}
	\includegraphics{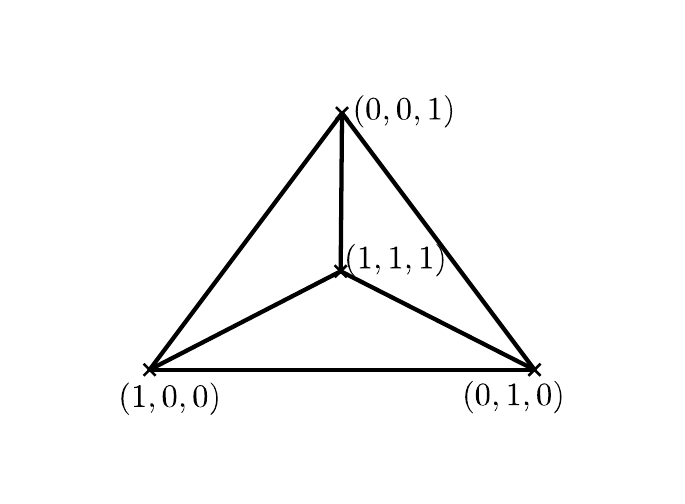}
	\caption{Secondary fan of the toric variety $Y$, used to construct \MM{3}{5}.}
	\label{fig:sec_fan3-5}
\end{figure}

Hence, considering the ample class $\omega = (2,1)$, $\phi^\star\omega = (2,1,2)$, it remains to analyse the effect of crossing the wall in the secondary fan of $Y$ generated by $(1,0,0)$ and $(1,1,1)$. We observe that moving the stability condition into this wall contracts the divisor $t=0$ (defining the ray generated by $(0,0,1)$) to the locus $\{y_0=y_1=y_2=0\}$.

We claim that vanishing loci of general sections of $E := \cO(2,1)^{\oplus 2}$ are smooth. If so, the blow-up of the base locus is an isomorphism on general sections, as the restriction of the base locus to a general fibre is a Cartier divisor. Smoothness follows directly from the Jacobian condition. Indeed, sections of $E$ are of the form
\[
c_0f_1 + c_1g_1 + a_0f_2 + a_1g_2
\]
where $f_j$ and $g_j$ are homogeneous polynomials of degree $j \in \{1,2\}$ in $b_0,b_1,b_2$. Taking two such sections the corresponding Jacobian matrix, evaluated at $b_0=b_1=b_2$ and -- without loss of generality -- $a_0=c_0=1$, has the form $\begin{pmatrix} 0 & L\end{pmatrix}$; a block matrix consisting of a $2\times 2$ zero block and a $2\times 3$ matrix $L$ of linear forms in $c_1$. Since the locus $F$ in $\PP^5$ where a $2\times 3$ matrix drops rank has codimension $2$, any projective line in this space which misses $F$ determines a matrix $L$ which does not drop rank.

\subsection*{Rank $3$ number $14$}

We consider the reflexive polytope $P$ with PALP ID $142$, together with the scaffolding $S$ with shape $Z = \PP^1$ shown in Figure~\ref{fig:3-14}. This scaffolding expresses $X_P$ as a hypersurface of tri-degree $(3,1,1)$ in the toric variety $Y_S$ with weight matrix
\[
\begin{matrix}
x_0 & x_1 & x_2 & y_0 & y_1 & z_0 & z_1 \\
\midrule
1 & 1 & 1 & 0 & 1 & 0 & 2 \\
0 & 0 & 0 & 1 & 1 & 0 & 0 \\
0 & 0 & 0 & 0 & 0 & 1 & 1 \\
\end{matrix}
\]
and stability condition $\omega = (3,1,1)$. Note that $Y_S$ is not $\QQ$-factorial at the point $w := \{x_0=x_1=x_2=y_0=z_0=0\}$. However, since the monomial $y_1z_1$ defines a section of $\cO(3,1,1)$ -- and this does not vanish along $w$ -- a general hypersurface $X$ with tri-degree $(3,1,1)$ misses this locus. Moving $\omega \in \Pic(Y_S)_\RR$ to $(4,1,1)$ induces a resolution of this singularity which restricts to an isomorphism of $X$, and recovers the ambient space considered in \cite[p.$70$]{CCGK}. Hence, by the argument given in \cite[p.$70$]{CCGK}, the hypersurface $X$ is isomorphic to a Fano variety in the family \MM{3}{14}.

\begin{rem}
	We could also construct varieties in this family using the scaffolding $S'$ obtained by combining the two struts containing the origin in $N_\RR$ into a single line segment of length two. This produces an embedding $X_P \to Y_S$, where $Y_S$ is given by the weight matrix
	\[
	\begin{matrix}
	s & x_0 & x_1 & x_2 & y & z \\
	\midrule
	1 & 0 & 0 & 0 & 1 & 1 \\
	0 & 1 & 1 & 1 & 1 & 2 \\
	\end{matrix}
	\]
	and stability condition $\omega = (1,3)$. $X_P$ is the vanishing locus of the binomial $yz = s^2x_0^3$. 
\end{rem}

\begin{figure}
	\includegraphics{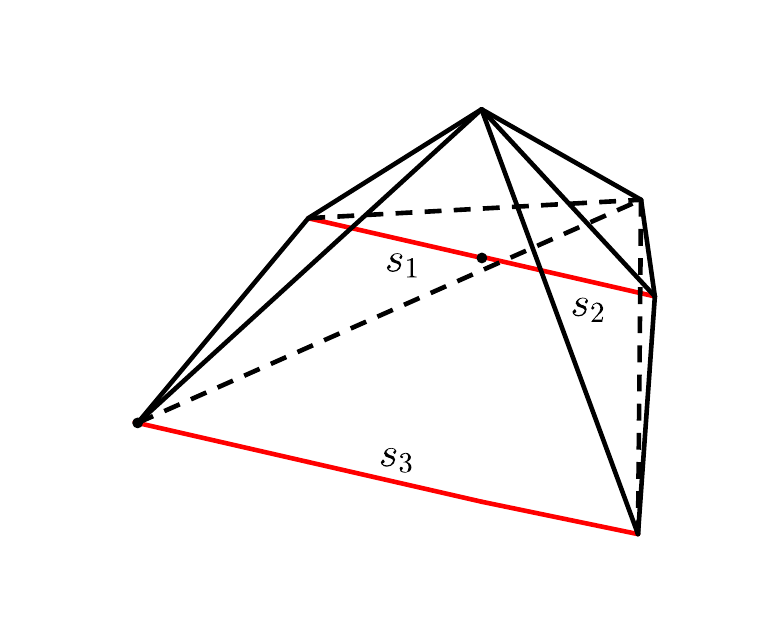}
	\caption{Scaffolding used to construct \MM{3}{14}}
	\label{fig:3-14}
\end{figure}

\subsection*{Rank $3$ number $16$}

Varieties in this family are obtained by blowing up $B_7 = \Bl_{pt}\PP^3$ with centre the strict transform of a twisted cubic passing through the centre of the blow-up $B_7 \to \PP^3$.

We can recover the construction used in \cite[p.$71$]{CCGK} using a scaffolding of a reflexive polytope. Indeed, consider the polytope $P$ with PALP ID $1091$, together with the scaffolding $S$ displayed in Figure~\ref{fig:3-16} with shape $Z = \PP^1\times \PP^1$. Note that this scaffolding is not full, and $P^\circ$ is not cracked along the fan defined by $Z$. The toric variety $Y_S$ is determined by the weight matrix
\[
\begin{matrix}
x_0 & x_1 & x_2 & y_0 & z & s & z_0 & z_1 \\
\midrule
1 & 1 & 1 & 0 & 0 & 1 & 0 & 0 \\
0 & 0 & 0 & 1 & 0 & 1 & 1 & 1 \\
0 & 0 & 0 & 0 & 1 & 0 & 1 & 1 \\
\end{matrix}
\]
together with the stability condition $\omega = (2,2,1)$. The toric variety $X_P$ is defined by the vanishing of a pair of binomial sections of $\cO(1,1,1)$. A stability condition which lies in the cone spanned by $\langle (1,0,0), (1,1,0), (1,1,1)\rangle$ determines the toric variety $\hat{Y}_S$ used in \cite{CCGK} to construct Fano varieties in \MM{3}{16}. However $\omega$ lies in the wall spanned by a pair of these vectors. Moving $\omega$ into the chamber used in \cite{CCGK} resolves the singular locus $\{x_0=x_1=x_2=y_0=z_0=z_1=0\}$. However general sections of $\cO(1,1,1)$ do not vanish along this point, and hence the intersection of two general divisors of tri-degree $(1,1,1)$ are isomorphic to varieties in the family \MM{3}{16}.

In order to provide a construction using a \emph{cracked} polytope, we consider the scaffolding $S'$ of $P$ with shape $Z = dP_6$, also shown in Figure~\ref{fig:3-16}.

\begin{figure}
	\centering
	\begin{minipage}[b]{0.49\textwidth}
		\includegraphics[width=\textwidth,  trim = 10mm 10mm 10mm 10mm, clip]{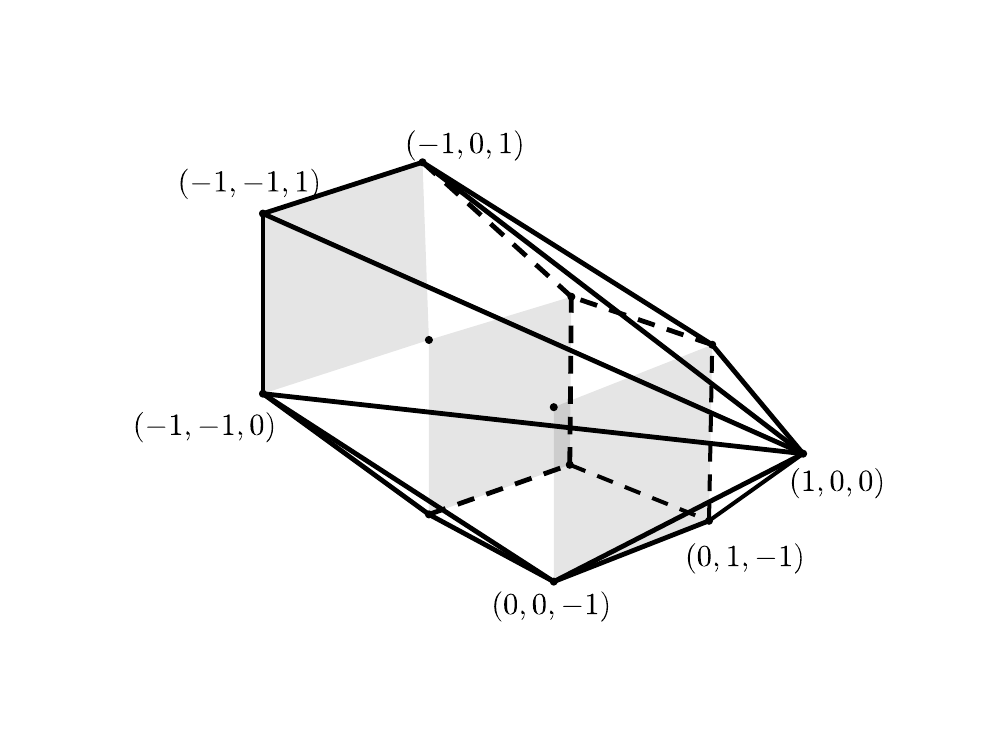}
	\end{minipage}
	\hfill
	\begin{minipage}[b]{0.49\textwidth}
		\includegraphics[width=\textwidth, trim = 10mm 10mm 10mm 10mm, clip ]{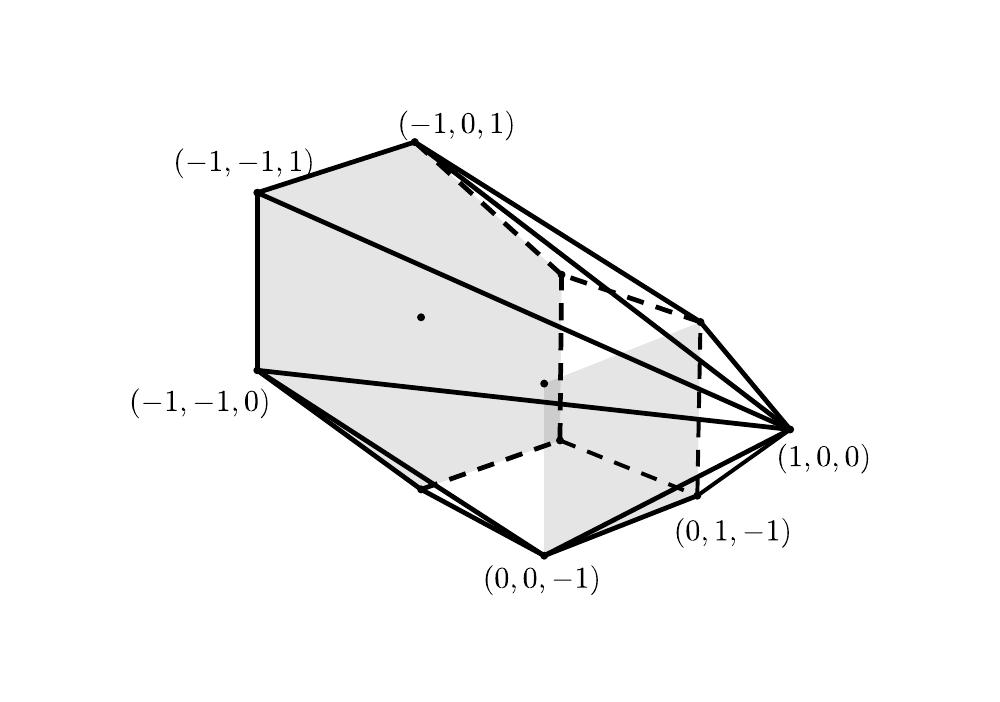}
	\end{minipage}
	\caption{Scaffolding used to construct \MM{3}{16}}
	\label{fig:3-16}
\end{figure}

The scaffolding $S'$ defines the weight matrix
\[
\begin{matrix}
x_0 & x_1 & x_2 & x_3 & y_0 & y_1 & y_2 & y_3 & y \\
\midrule
1 & 1 & 1 & 1 & 1 & 1 & 1 & 1 & 0 \\
0 & 0 & 0 & 0 & 1 & 1 & 1 & 1 & 1 \\
\end{matrix}
\]
and stability condition $(2,1)$. Let $Y$ denote the toric variety determined by the weight matrix
\[
\begin{matrix}
1 & 1 & 1 & 1 & 1 & 1 & 1 & 1 & 1 & 1 & 0 \\
0 & 0 & 0 & 0 & 1 & 1 & 1 & 1 & 1 & 1 & 1
\end{matrix}
\]
and stability condition $(2,1)$. Note that general sections of $\cO(1,1)^{\oplus 2}$ define subvarieties of $Y$ isomorphic to $Y_{S'}$. There is a map $\theta \colon Y_S \hookrightarrow Y$ -- analogous to the Segre embedding map $\PP^2\times \PP^2 \to \PP^8$ -- sending
\[
(x_0, x_1, x_2, y_0, z, s, z_0, z_1) \mapsto (x_0y_0, x_1y_0, x_2y_0, s, x_0z_0, x_0z_1, x_1z_0,x_1z_1, x_2z_0, x_2z_1, z).
\]
We have that $\theta^\star \cO(1,0) = \cO(1,1,0)$, while $\theta^\star \cO(0,1) = \cO(0,0,1)$. Hence the ample line bundle $\cO(2,1)$ pulls-back to $\cO(2,2,1)$. This class is not ample on $\hat{Y}_S$ and the image of the induced morphism $\hat{Y}_S \to Y$ factors through the contraction $\hat{Y}_S \to Y_S$. Indeed, we have the commutative diagram of embeddings
\[
\xymatrix{
X_P \ar[r] & Y_S \ar[r] & Y \\
X_P \ar[r] \ar@{=}[u] & Y_{S'}  \ar[ur] & 
}.
\]
We can deform $X_P$ in $Y_S$ by moving the section of $\cO(1,1,1)^{\oplus 2}$ cutting out $X_P$. In other words, we obtain varieties in the family \MM{3}{16} in $Y_{S'}$ in codimension $4$ by embedding $Y_S \to Y$ and moving the sections used to cut out $Y_{S'}$.

\subsection*{Rank $4$, number $2$} 
Varieties in this family are obtained from $\PP^1\times \PP^1\times \PP^1$ by blowing up a curve of tri-degree $(1,1,3)$.

We consider the polytope with PALP ID $1080$, together with the scaffolding shown in Figure~\ref{fig:4-2}, with shape $Z = \PP^2$. This scaffolding describes $X_P$ as a hypersurface of tri-degree $(1,1,2)$ in the toric variety $Y_S$ determined by the weight matrix
\[
\begin{matrix}
x_0 & x_1 & y_0 & y_1 & z_0 & z_1 & z_2 \\
\midrule
1 & 1 & 0 & 0 & 0 & 0 & 0 \\
0 & 0 & 1 & 1 & 0 & 0 & 1 \\
0 & 0 & 0 & 0 & 1 & 1 & 1
\end{matrix}
\]
and stability condition $\omega = (1,2,1)$. The variety $Y_S$ is the projectivisation of the bundle $\cO^{\oplus 2} \oplus \cO(0,1)$ on $\PP^1\times\PP^1$. Note that the line bundle $\cO(1,1,2)$ is not nef, and that its base locus is section of the projection $Y_S \to \PP^1\times \PP^1$ defined by $z_0=z_1=0$. Blowing up this base locus we obtain the variety $F$ considered in \cite[p.$82$]{CCGK}. To check smoothness of general hypersurfaces in this linear system, note that general sections of $L$ have the form
\[
f = z_0^2f_{1,1} + z_0z_1g_{1,1} + z_1^2h_{1,1} + z_0z_2f_1 + z_1z_2g_1,
\] 
where $f_{1,1}$ and $g_{1,1}$ are polynomials of bidegree $(1,1)$ in $x_0,x_1,y_0,y_1$, while $f_1$, $g_1$ are linear polynomials in $x_0,x_1$. Restricting the Jacobian to the locus $z_0=z_1=0$, we see that the locus $\{f=0\}$ is singular precisely when $f_1=g_1=0$. However this locus is empty for general choices of $f_1$ and $g_1$.

\begin{figure}
	\includegraphics{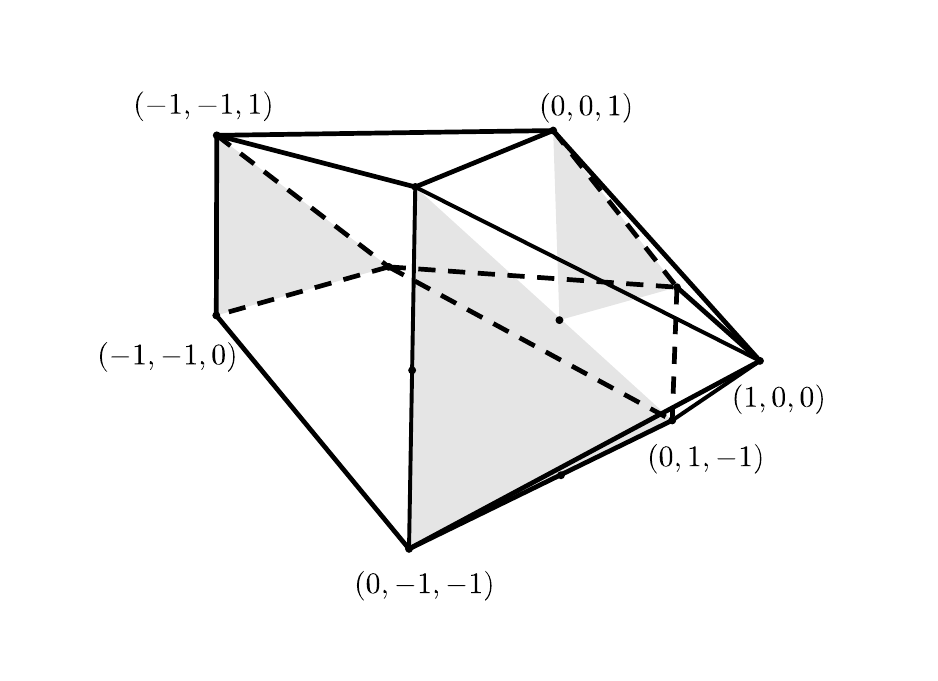}
	\caption{Scaffolding used to construct \MM{4}{2}}
	\label{fig:4-2}
\end{figure}

Since the restriction of the base locus of this linear system to a smooth member $X$ is a Cartier divisor in $X$, its blow up is an isomorphism. Hence such hypersurfaces $X$ are members of the family \MM{4}{2}, and $X_P$ is the central fibre of a toric degeneration in this family.

\subsection*{Rank $4$ number $6$}

Varieties $X$ in the family $\MM{4}{6}$ are obtained by blowing up $\PP^2\times \PP^1$ in curves of bidegree $(1,2)$ and $(0,1)$ respectively. Consider the polytope $P$ with PALP ID $425$, together with the scaffolding $S$ with shape $\PP^1\times\PP^1$ illustrated in Figure~\ref{fig:4-6}.

\begin{figure}
	\includegraphics{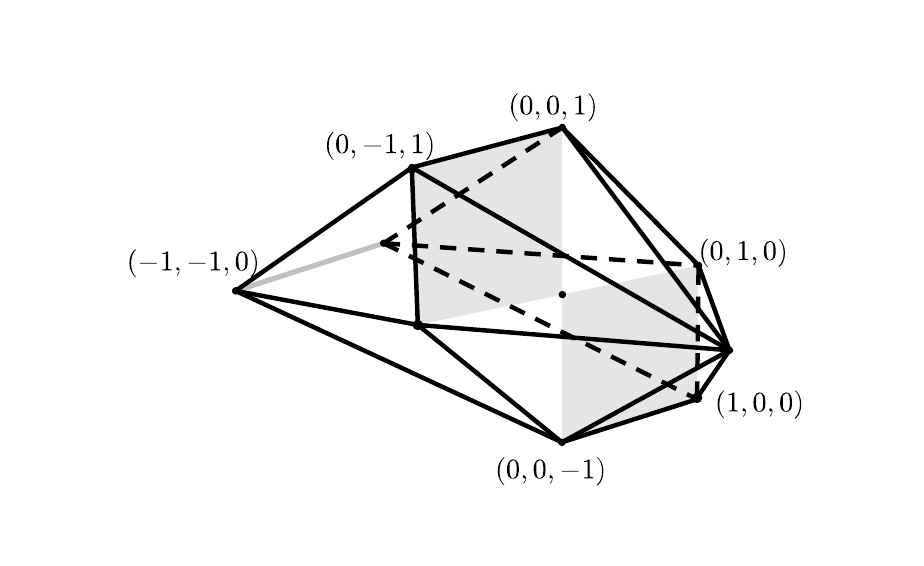}
	\caption{Scaffolding used to construct \MM{4}{6}}
	\label{fig:4-6}
\end{figure}

The toric variety $Y_S$ is defined by the weight matrix
\[
\begin{matrix}
s_0 & s_1 & s_2 & y_0 & y_1 & x_0 & x_1 & x_2 \\
\midrule
1 & 1 & 1 & 0 & 0 & 0 & 0 & 0 \\
0 & 0 & 0 & 0 & 0 & 1 & 1 & 1 \\
0 & 0 & 1 & 1 & 1 & 0 & 0 & 0
\end{matrix}
\]
and stability condition $\omega = (1,1,2)$. The secondary fan of $Y_S$ is illustrated in Figure~\ref{fig:sec_fan4-6}. 

\begin{figure}
	\includegraphics{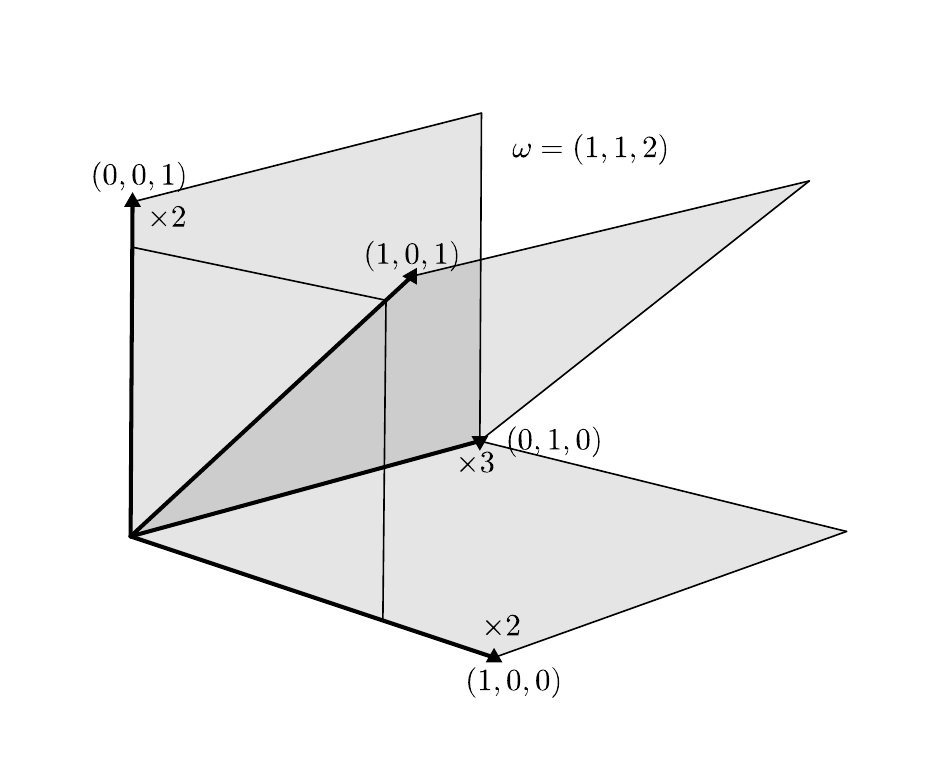}
	\caption{Secondary fan of the toric variety $Y_S$}
	\label{fig:sec_fan4-6}
\end{figure}

The variety $Y_S$ is isomorphic to $\PP_{\PP^1\times\PP^2}(\cO^{\oplus 2}\oplus \cO(1,0))$; and the two chambers in the secondary fan correspond to isomorphic varieties -- despite the presence of a non-trivial flopping locus. The projection $\pi \colon Y_S \to \PP^2\times \PP^1$ corresponds to projecting out the variables $s_i$ for all $i \in \{0,1,2\}$. The toric variety $X_P$ is cut out of $Y_S$ by the binomial equations
\begin{align*}
s_2x_2 = s_0x_0y_0 && s_1x_1 = s_0x_0.
\end{align*}
These are sections of the line bundles $L_1$ and $L_2$, with weights $(1,1,1)$ and $(1,1,0)$ respectively. Note that $L_1$ is nef while $L_2$ is not.

Let $X$ be the vanishing locus of a general section $s = l_1+l_2$ of $E := L_1\oplus L_2$. The section $l_1 \in \Gamma(Y_S,L_1)$ has the general form $s_0f_{1,1} + s_1g_{1,1} + s_2h_1$, where $f_{1,1}$ and $g_{1,1}$ have bi-degree $(1,1)$ in $x_0,x_1,x_2$ and $y_0,y_1$ respectively; while $h_1$ has bi-degree $(1,0)$. Similarly $l_2$ has the general form $s_0f_1 + s_1g_1$, where $f_1$ and $g_1$ have bi-degree $(1,1)$.

Fibres of the restriction of $\pi$ to $X$ are given by the kernel of the matrix
\[
\begin{pmatrix}
f_{1,1} & g_{1,1} & h_1 \\
f_1 & g_1 & 0 .
\end{pmatrix}
\]
That is, $\pi$ is a graph away from the locus at which this matrix has rank $\leq 1$. This locus in $\PP^2\times\PP^1$ has two connected components, one given by $h_1 = f_{1,1}g_1-g_{1,1}f_1=0$, a curve of bidegree $(1,2)$, and the other by $f_1=g_1=0$, a curve of degree $(0,1)$. Thus the morphism $\pi$ exhibits $X$ as a Fano threefold in the family $\MM{4}{6}$.

\subsection*{Rank $5$ number $1$}

Varieties in this family are obtained by first blowing up a quadric in a conic -- obtaining a variety $V$ in the family \MM{2}{29} -- and blowing up $V$ in three exceptional lines. Consider the scaffolding $S$ of the polytope with PALP ID $1082$ with shape $\PP^2$, illustrated in Figure~\ref{fig:5-1}.

\begin{figure}
	\includegraphics{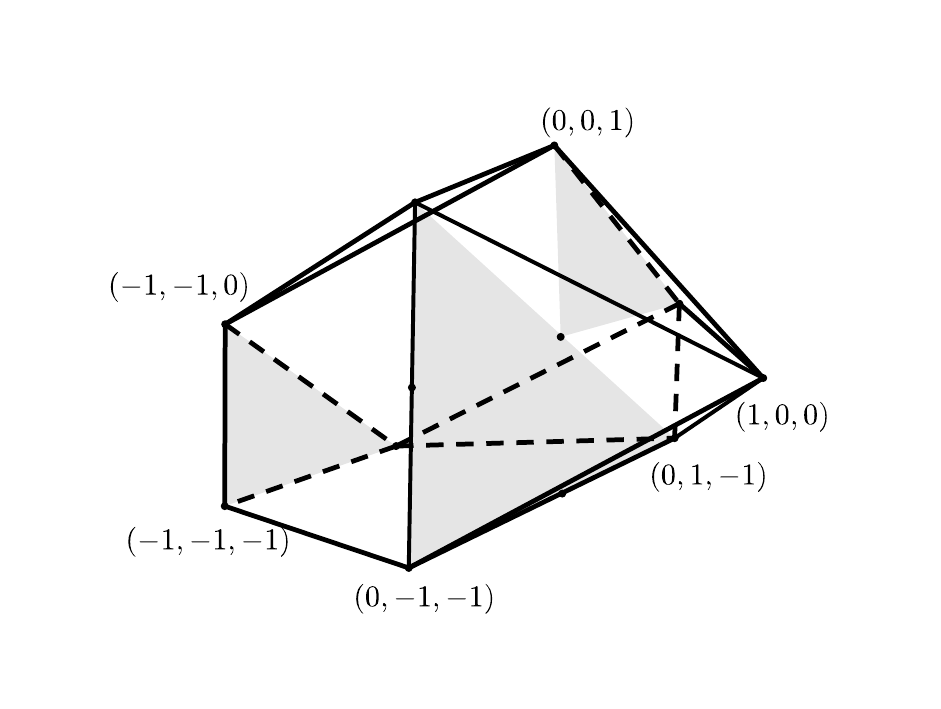}
	\caption{Scaffolding used to construct \MM{5}{1}}
	\label{fig:5-1}
\end{figure}

That is, we consider general hypersurfaces $X$ of tri-degree $(1,2,1)$ in the toric variety $Y_S$ defined by the weight matrix
\[
\begin{matrix}
s_0 & s_1 & x_0 & x_1 & x_2 & y_0 & y_1 \\
\midrule
1 & 1 & 0 & 1 & 1 & 0 & -1 \\
0 & 0 & 1 & 1 & 1 & 0 & 0 \\
0 & 0 & 0 & 0 & 0 & 1 & 1
\end{matrix}
\]
and stability condition $(2,1,1)$. The variety $Y_S$ admits a map to $\PP^1$ (with co-ordinates $(s_0:s_1)$), giving $Y_S$ the structure of a $\PP^2\times \PP^1$ fibre bundle. The variety $X$ also admits a morphism to $\PP^1$, whose fibres are surfaces of bi-degree $(2,1)$ in $\PP^2\times\PP^1$. Projecting $\PP^2\times\PP^1$ to $\PP^2$ we see that any such smooth fibre is the blow up of $\PP^2$ in four (general) points; that is, isomorphic to the del Pezzo surface $dP_5$.

Hypersurfaces of tri-degree $(1,2,1)$ have general form
\[
y_0x_0(x_0f_1 + p_1) + y_1(x_0^2f_2 + x_0f_1q_1 + p_2),
\]
where $p_i,q_i \in \CC[x_1,x_2]$ and $f_i \in \CC[s_0,s_1]$ are homogeneous polynomials of degree $i$ for each $i \in \{1,2\}$. Let $X$ denote the vanishing locus of this polynomial. Note that $X$ contains the surface $\{x_0=y_1=0\}$. Fixing a point $(s_0,s_1) \in \PP^1$, the $dP_5$ fibre of the projection $X \to \PP^1$ is obtained by blowing up the intersection points of the conics $C_1 := \{x_0(x_0f_1 + p_1)=0\}$ and $C_2 := \{(x_0^2f_2 + x_0f_1q_1 + p_2)=0\}$ in $\PP^2$ (with homogeneous co-ordinates $(x_0:x_1:x_2)$). First consider the case $x_0=p_2=0$. Choosing a general $p_2$, we find two distinct reduced points $\alpha_1$, $\alpha_2$ in $C_1 \cap C_2$; these are independent of the choice of $s = (s_0,s_1) \in \PP^1$. The other two solutions depend on $s$, and lie in the line $(x_0f_1 + p_1)=0$. Note that we may choose co-ordinates such that $C_1$ is defined by $\{x_0x_1=0\}$.

Hence we can construct four surfaces, each isomorphic to $\PP^1\times \PP^1$, contained in $X$: two surfaces -- $S_1$ and $S_2$ -- swept out by $\{\alpha_i\} \times \PP^1_{(y_0:y_1)}$, the surface $S_3$ swept out by $C_1$ over $\PP^1_{(s_0:s_1)}$, and the base locus $S_4 = \{x_0=y_1=0\}$. Each of these surfaces restrict to  exceptional curves in the $dP_5$ fibres. Note that fibres of $X \to \PP^1$ are not all smooth -- there are two singular fibres -- but they are smooth in a neighbourhood of $\bigcup_{i \in [4]} S_i$.  Hence -- applying a relative version of Castelnuovo's criterion -- we can have a morphism $X \to X'$ which contracts the disjoint surfaces $S_1$, $S_2$, and $S_3$ to sections of the induced morphism $\pi \colon X' \to \PP^1_{(s_0:s_1)}$. The smooth fibres of $\pi$ are isomorphic to $\PP^1\times\PP^1$, while singular fibres have a single nodal singularity; these are isomorphic to $\PP(1,1,2)$. The surface $S_4$ is the strict transform of a surface $S'_4$, which intersects every fibre $F$ in a smooth section of $-\frac{1}{2}K_F$.

Letting $\rho(X)$ denote the Picard rank of $X$, we have that $\rho(X) = \rho(X')+3$. Since $X_P$ -- and hence $X$ -- has degree $28$, we can conclude from the classification of Fano $3$-folds that if $\rho(X') \geq 2$, $X$ is in the family \MM{5}{1}. This is easily seen from the Leray spectral sequence
\[
H^i(\PP^1,R^j\pi_\star\QQ) \Rightarrow H^{i+j}(X',\QQ);
\]
indeed -- since $H^1(F,\QQ) = 0$ for all fibres $F$ of $\pi$ -- we have $b^2(X') = 1 + h^0(\PP^1,R^2\pi_\star\QQ)$. However $h^0(\PP^1,R^2\pi_\star\QQ) \geq 1$ since the surface $S'_4$ defines a non-trivial class in $H^2(F,\QQ)$ for every fibre $F$.

\begin{rem}
	Comparing our construction with that made by Mori--Mukai~\cite{Mori--Mukai:Manuscripta}, they first consider the blow up of a quadric threefold in a conic. Restricting the projection $\PP^4 \dashrightarrow \PP^1$ this blow-up defines $X'$, a quadric surface bundle over $\PP^1$ with two singular fibres (with singularities are disjoint from the exceptional locus). Note that the exceptional locus distinguishes a conic $C$ in each fibre of $\pi$. To obtain varieties in \MM{5}{1} we then blow-up $X'$ in three exceptional lines. These lines are sections of the map $X' \to \PP^1$ defined by a triple of points on the distinguished conic $C$ in each fibre. That is, the surface $S_4$ is the strict transform of the exceptional locus obtained by the blow-up of the quadric threefold; while $S_i$, $i \in \{1,2,3\}$ are obtained by blowing up exceptional lines.
\end{rem}

\subsection{Products}
\label{sec:products}

The remaining non-toric Fano threefolds $X$ with $-K_X$ very ample are products of non-toric del Pezzo surfaces with $\PP^1$. That is, $dP_k\times \PP^1$ for $k \in \{3,4,5\}$. We can easily construct toric degenerations of these from degenerations of $dP_k$ for each $k$. Fix a reflexive polygon $Q$ such that $Q^\circ$ is cracked along the fan of a shape variety $Z'$, together with a scaffolding $S'$ of $Q$ with shape $Z'$. We can produce a scaffolding $S$ of $\conv{Q,(0,0,1),(0,0,-1)}$ with shape $Z := Z'\times \PP^1$ by setting $S = \{(\pi_1^\star(D),\chi) : (D,\chi) \in S'\} \cup \{\pi_2^\star D\}$ where $D$ is the toric boundary of $\PP^1$, and $\pi_i$ is the $i$\textsuperscript{th} projection from $Z'\times \PP^1$. The example of $dP_3\times \PP^1$, together with a scaffolding with shape $Z = \PP^2 \times \PP^1$ is illustrated in Figure~\ref{fig:proddPn}, setting $a=1$ and $b=3$. We thus produce toric degenerations embedded in the following spaces.
\begin{enumerate}
	\item $dP_3\times \PP^1 \to \PP^3\times \PP^2$,
	\item $dP_4\times \PP^1 \to \PP^4\times \PP^2$,
	\item $dP_5\times \PP^1 \to \PP^1 \times \PP^2 \times \PP^2$.
\end{enumerate}

\begin{figure}
	\includegraphics{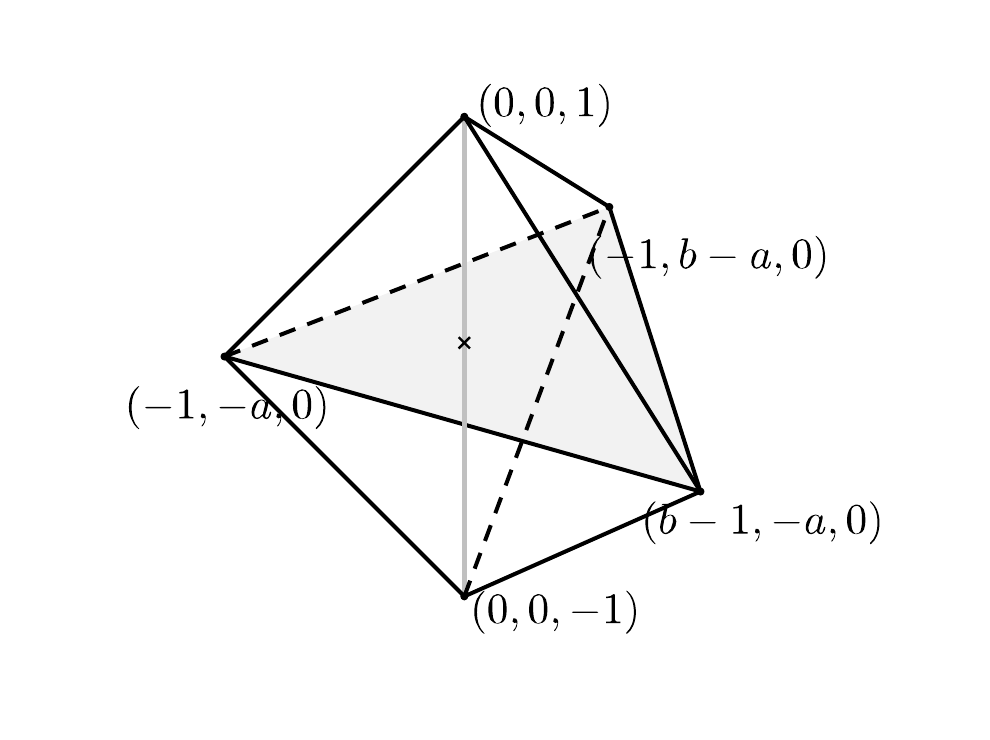}
	\caption{Scaffolding used to construct $dP_n \times \PP^1$ for $n \leq 3$.}
	\label{fig:proddPn}
\end{figure}

\subsection{$-K_X$ not very ample}
\label{sec:not_va}

There are $7$ families of Fano threefolds $X$ for which $-K_X$ is not very ample. These fall into three distinct groups. We first consider the varieties
\begin{enumerate}
	\item $B_1$, a sextic in $\PP(1,1,1,2,3)$ ; and, 
	\item $V_2$, a sextic in $\PP(1,1,1,1,3)$.
\end{enumerate}
Writing $x_i$ for homogeneous co-ordinates of degree $1$, and $y$, $z$ for those of degree $2$ and $3$ respectively, $B_1$ degenerates to the toric hypersurface $x_2yz = x_0^6$; while $V_2$ degenerates to the toric variety $x_1x_2x_3z = x_0^6$. These toric varieties correspond to scaffoldings of non-reflexive toric varieties with shape $\PP^2$ and $\PP^3$ respectively. The scaffolding used to construct $B_1$ is illustrated in Figure~\ref{fig:wPS} in the case $(a,b) = (2,6)$. The details of these constructions follow those described in \S\ref{sec:B2}.

\begin{figure}
	\includegraphics{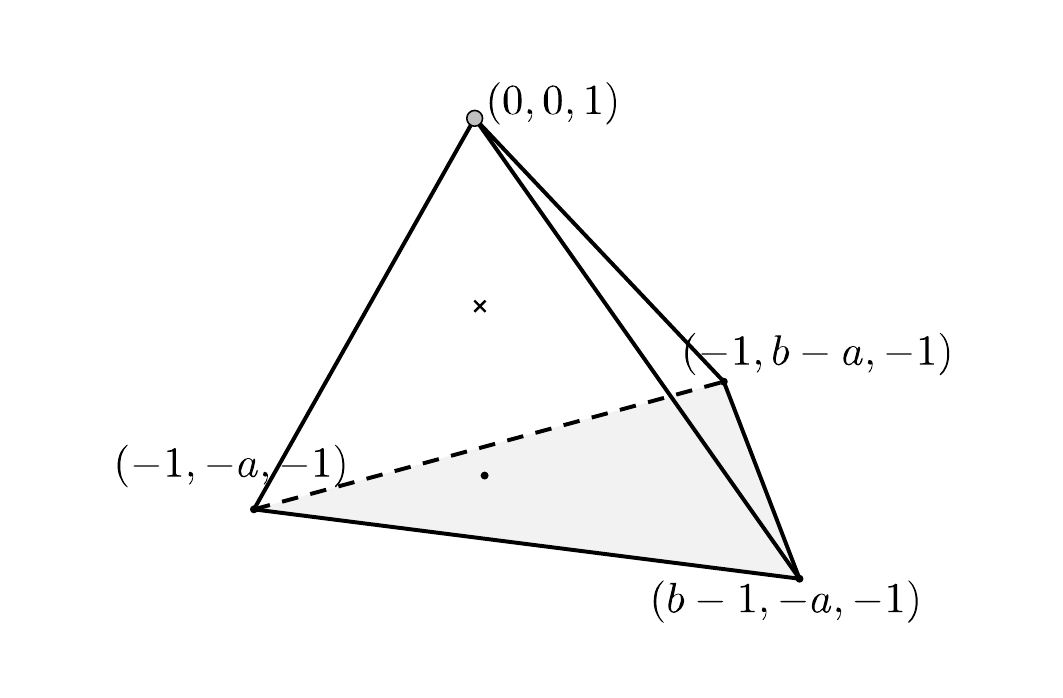}
	\caption{The scaffolding used to construct $B_i$ for each $i \in [3]$.}
	\label{fig:wPS}
\end{figure}

The second group consists of the following three families of Picard rank $2$ Fano threefolds.
\begin{enumerate}
	\item \MM{2}{1}, the blow up of $B_1$ is an elliptic curve formed by intersecting two members of $-\frac{1}{2}K_{B_1}$.
	\item \MM{2}{2}, a double cover of $\PP^1\times\PP^2$ branched along a divisor of bidegree $(2,4)$.
	\item \MM{2}{3}, the blow up of $V_2$ is an elliptic curve formed by intersecting two members of $-\frac{1}{2}K_{V_2}$.
\end{enumerate}

In each case a toric complete intersection construction is given in \cite{CCGK}, and each construction admits a toric degeneration to an embedding described by Laurent inversion. The corresponding scaffoldings have shapes $\PP^2\times\PP^1$, $\PP^3$, and $\PP^2\times \PP^1$ respectively. Letting $(x_0:x_1:x_2:y:z)$ be homogeneous co-ordinates on $\PP(1,1,1,2,3)$, and $(s_0:s_1)$ be co-ordinates on $\PP^1$, varieties in the family $\MM{2}{1}$ degenerate to the toric variety given by the binomial equations
\[
\begin{cases}
x_2yz = x_0^6 \\
x_1s_1 = x_0s_0
\end{cases}
\]
in $\PP(1,1,1,2,3)\times \PP^1$. Varieties in the family \MM{2}{3} degenerate to the toric variety given by the binomial equations
\[
\begin{cases}
x_2x_3y = x_0^4 \\
x_1s_1 = x_0s_0
\end{cases},
\]
where $(x_0:x_1:x_2:x_3:y)$ are homogeneous co-ordinates on $\PP(1,1,1,1,2)$. Finally, varieties in the family \MM{2}{2} degenerate to the hypersurface $x_1y_1y_2w = x_0^2y_0^4$ in the variety $F$ described in \cite[p.25]{CCGK}.

Finally, we have the following two families of products
\begin{enumerate}
	\item $dP_2\times \PP^1$, recalling that $dP_2$ is a quartic in $\PP(1,1,1,2)$ ; and, 
	\item $dP_1 \times \PP^1$, recalling that $dP_1$ is a sextic in $\PP(1,1,2,3)$.
\end{enumerate}
Let $Q_1$ and $Q_2$ denote the polygons associated to the toric varieties given by the binomials $\{x_1x_2y = x_0^4\}$ and $\{x_0^6 = x_1yz\}$ respectively. $Q_1$ and $Q_2$ are triangles and the corresponding scaffolding (with shape $\PP^2$) covers each of these with a single strut. Hence we can scaffold $\conv{Q_i,(0,0,1),(0,0,-1)}$ with a pair of struts  -- following the constructions made in \S\ref{sec:products} -- embedding $dP_2\times \PP^1 \to \PP(1,1,1,2) \times \PP^2$ and $dP_1\times \PP^1 \to \PP(1,1,2,3) \times \PP^2$. These scaffoldings are illustrated in Figure~\ref{fig:proddPn}, setting $(a,b) = (1,4)$ and $(a,b) = (2,6)$ respectively.

\section{Classifying cracked $3$-topes}
\label{sec:classification}

We consider the combinatorial problem of classifying cracked polytopes, and present an algorithm to obtain such a classification in three dimensions.

\subsection{One dimensional shape variety}
\label{sec:P1_shape}

We refer to polytopes cracked along the fan of $\PP^1$ as \emph{cracked in half}, since their intersection with a pair of half spaces form unimodular polytopes. This class of polytopes is explored in greater detail -- and in the four dimensional setting -- in \cite{CKP:P1}.

Since polytopes cracked in half are reflexive \cite[Proposition~$2.5$]{P:Cracked}, we can proceed from the classification of reflexive $3$-topes. Given a reflexive polytope $P \subset M_\RR$, we define $V_P$ to be the vector space spanned by the vertices $v \in P$ such that the tangent cone $C_v$ to $P$ at $v$ is not unimodular. If $P$ is cracked along $\PP^1$ these must lie in a proper linear subspace of $M_\RR$. Moreover, by \cite[Proposition~$2.8$]{P:Cracked}, no facet of $P^\circ$ contains an interior point. We use Magma to search for reflexive polytopes meeting both these conditions, and obtain a list of $91$ reflexive $3$-topes. In $73$ cases $V_P$ is two-dimensional, and hence unique determines the direction of the line segments used to scaffold $P^\circ$. The remaining polytopes contain a square facet and admit two possible full scaffoldings.

Testing which of these $91$ polytopes are cracked in half, we find there are $82$ three dimensional polytopes cracked along the fan of $\PP^1$ and we list these reflexive polytopes in Table~\ref{tbl:cracked_in_half}. These polytopes are specified by their index in the Kreuzer--Skarke list of reflexive $3$-topes. Note that, as elsewhere, we index this list from zero. The column \emph{Fano} indicates the families Fano threefolds $X$ for which there is a mirror Minkowski (as defined in \cite{CCGGK,CCGK}) polynomial $f$ such that $\Newt(f)$ is isomorphic to the reflexive polytope with the indicated ID. Note that in each case there is at most one such family of Fano threefolds. Applying Laurent inversion to a full scaffolding on $P$ with shape $Z = \PP^1$, we obtain $X_P$ as a Fano hypersurface. We expect to recover $X$ by passing to a general hypersurface, although we have only partial results in this direction.

\begin{pro}[\!\!\cite{P18}]
	For each $P$ in Table~\ref{tbl:cracked_in_half} with no associated Fano threefold, $X_P$ is not smoothable.
\end{pro}
\begin{proof}
	The list of reflexive $3$-topes with no associated Fano in  Table~\ref{tbl:cracked_in_half} is a subset of the list of non-smoothable Fano threefolds which appears in work of Petracci~\cite[p.$10$]{P18}.
\end{proof}

\begin{pro}[\!\!\cite{G07}]
	For each polytope $P$ indexed in Table~\ref{tbl:cracked_in_half} such that each torus invariant point of $X_P$ is either a smooth point, or an ordinary double point, $X_P$ smooths to the indicated Fano manifold.
\end{pro}
\begin{proof}
	By Namikawa's results~\cite{N97} all such toric varieties admit a smoothing. The invariants of the smoothed varieties were computed by Galkin in \cite{G07}.
\end{proof}

\begin{table}	
	\begin{tabular}{cc|cc|cc}
		PALP ID & Fano & PALP ID & Fano & PALP ID & Fano \\
		\midrule
		$1$ & $Q^3$ & $69$ & \MM{2}{31} & $202$ & \MM{3}{14} \\
		$3$ & $Q^3$ & $71$ & \MM{2}{29} & $204$ & \MM{3}{23} \\
		$13$ & \MM{2}{30} & $72$ & \MM{3}{26} & $206$ & \MM{3}{24} \\
		$14$ & - & $73$ & \MM{3}{25} & $207$ & \MM{3}{20} \\
		$15$ & - & $74$ & \MM{3}{19} & $211$ & \MM{3}{18} \\
		$17$ & \MM{3}{27} & $75$ & \MM{3}{22} & $213$ & \MM{3}{21} \\
		$18$ & \MM{2}{29} & $76$ & \MM{3}{23} & $214$ & \MM{4}{10} \\
		$19$ & \MM{3}{31} & $77$ & \MM{3}{24} & $215$ & \MM{4}{8} \\
		$20$ & \MM{2}{31} &  $78$ & \MM{3}{24} & $216$ & \MM{4}{9} \\
		$21$ & \MM{2}{32} & $79$ & \MM{3}{20} & $217$ & \MM{4}{8} \\
		$22$ & \MM{2}{32} & $80$ & \MM{3}{28} & $288$ & \MM{3}{9} \\
		$23$ & \MM{2}{34} & $130$ & \MM{2}{29} & $340$ & \MM{3}{18} \\
		$33$ & \MM{2}{28} &  $142$ & \MM{3}{14} & $343$ & \MM{3}{9} \\
		$45$ & \MM{2}{31} & $170$ & \MM{2}{29} & $345$ & \MM{4}{9} \\
		$51$ & \MM{3}{28} & $177$ & \MM{3}{23} & $353$ & \MM{3}{9} \\
		$54$ & \MM{2}{28} & $179$ & \MM{4}{10} & $373$ & \MM{3}{9} \\
		$56$ & \MM{3}{19} & $180$ & \MM{4}{12} & $392$ & \MM{3}{18} \\
		$57$ & - &  $183$ & \MM{3}{21} & $403$ & \MM{3}{20} \\
		$58$ & - & $185$ &\MM{3}{14} & $407$ & \MM{5}{2} \\
		$59$ & \MM{4}{13} & $189$ & \MM{4}{12} & $408$ & \MM{4}{6} \\
		$60$ & - & $190$ & \MM{4}{11} & $425$ & \MM{4}{6} \\
		$61$ & \MM{4}{11} & $191$ & - & $426$ & \MM{4}{5} \\
		$62$ & \MM{3}{18} & $192$ & - & $682$ & \MM{4}{5} \\
		$63$ & \MM{3}{22} & $193$ & \MM{5}{2} & $683$ & \MM{4}{3} \\
		$64$ & - & $194$ & \MM{5}{3} & $726$ & \MM{4}{5} \\
		$65$ & - & $195$ & \MM{4}{5} & $727$ & \MM{4}{3} \\
		$66$ & \MM{4}{9} & $196$ & - & $734$ & \MM{4}{3} \\
		$68$ & \MM{2}{28} & & & & \\
	\end{tabular}
	\caption{Reflexive polytopes cracked in two.}
	\label{tbl:cracked_in_half}
\end{table}

Assuming the toric Fano varieties associated to the reflexive polyhedra listed in Table~\ref{tbl:cracked_in_half} all smooth as indicated, there are $22$ non-toric Fano threefolds obtained from polytopes cracked along the fan of $Z = \PP^1$; these are:

\begin{align*}
 Q^3,\MM{2}{29},\MM{2}{30},\MM{2}{31},\MM{2}{32},\MM{3}{14}, \\
 \MM{3}{18},\MM{3}{19},\MM{3}{20},\MM{3}{21},\MM{3}{22},\MM{3}{23},\\
 \MM{3}{24},\MM{4}{3},\MM{4}{5},\MM{4}{6},\MM{4}{8},\MM{4}{9},\\
 \MM{4}{10},\MM{4}{11},\MM{4}{12},\MM{4}{13}.
\end{align*}

\subsection{Classification algorithm}

We present the general form of an algorithm which we can use to classify three dimensional polytopes cracked along a given two dimensional fan $\Sigma$. Fixing a choice of $Z$, and letting $\Sigma$ denote the corresponding fan, we first divide cases among possible \emph{wrapping polyhedra}.

\begin{dfn}
	Given a polytope $P$ cracked along a fan $\Sigma$, let $C_v$ denote the tangent cone to $P$ at a point $v \in P$. The \emph{wrapping polyhedron} of $P$ is the intersection of cones $C_v$ as $v$ varies over the primitive ray generators of $\Sigma$.
\end{dfn}

Note that the set of primitive ray generators is empty in the case $Z = \PP^1$, and need not be a subset of the vertex set of $P$ for any choice of shape $Z$.

\begin{lem}
	Fix a shape variety $Z$ determined by a fan $\Sigma$ in $M_{\RR}$ and a ray $\rho \in \Sigma[1]$. Let $Z_\rho$ denote the codimension one torus invariant subvariety of $Z$ determined by $\rho$. There is a canonical inclusion, with bounded image, from the set wrapping polyhedra of reflexive polytopes $P$ cracked along $\Sigma$ to the set of lattice points in the cone 
	\[
	\prod_{\rho \in \Sigma[1]}{\left\{ \Nef(Z_\rho) \times (M_\RR/\RR \rho) \right\}}.
	\]
\end{lem}
\begin{proof}
	Fix a splitting $M \cong \ZZ v \oplus M_\rho$, and let $\Sigma_\rho$ denote the fan in $M_\rho$ determined by $Z_\rho$. The tangent cone at $v$ to a wrapping polyhedron for $\Sigma$ determines -- and is determined by -- a piecewise linear function $\theta \colon (M_\rho)\otimes_\ZZ \RR \to M_\RR$ which is linear on each cone of $\Sigma_\rho$, sends $0 \mapsto v$, and sends the cones of $\Sigma_\rho$ into their corresponding cones in $\Sigma$. The connected component of the complement of the image of $\theta$ which contains the origin must be a convex set. Such maps $\theta$ are in bijection with points in $\Nef(Z_\rho) \times (M_\RR/\RR v) \subset \Div_{T_{M_\rho}}(Z_\rho) \cong \ZZ^r$, for some $r \in \ZZ_{\geq 0}$. Hence the set of possible wrapping polyhedra is contained in the cone required.

	To show this region is bounded, first note that each ray $\tau$ of $\Sigma_\rho$ corresponds to a cone in $\Sigma$ of dimension $2$; generated by $v$ and some $v' \in M$. Since $v'$ must be in the same connected component as the origin of $M_\RR \setminus \theta((M_\rho)\otimes_\ZZ \RR)$, the co-ordinate of $\theta$, regarded as an element of $\ZZ^r$, corresponding to $\tau$ is bounded. Each pair $(\rho,\tau)$, where $\rho \in \Sigma[1]$ and $\tau \in \Sigma_\rho[1]$ defines a linear inequality satisfied by any tuple of piecewise linear maps $\theta$ which define a wrapping polyhedron. The intersection of these half spaces with $\Nef(Z_\rho) \times (M_\RR/\RR \rho)$ defines a polytope, $\cR_\Sigma$, which contains the image of each wrapping polyhedron.
\end{proof}


Recall that a polytope is called \emph{hollow} if it contains no lattice points in its interior.

\begin{dfn}
	Let $Q$ be a unimodular hollow polytope in $M_\RR$. We call $Q$ a \emph{(reflexive) piece} if $0 \in Q$ and, for any facet $F$ of $Q$ with primitive inner normal vector $w$, either $0 \in F$, or $w(F) = -1$. 
\end{dfn}

The set of reflexive pieces has an obvious iterative structure: faces of reflexive pieces which contain the origin are themselves reflexive pieces. Thus the classification of reflexive pieces of dimension $n$ makes use of the classification in dimensions $< n$. If $Q$ is a $3$-tope there are four cases, depending on the minimal dimension $d$ of the face of $Q$ containing the origin. In particular either

\begin{enumerate}
	\item $Q$ is a reflexive polytope;
	\item the origin is the unique interior point of a facet of $Q$;
	\item the origin is the unique relative interior lattice point of an edge of $Q$, or;
	\item the origin is a vertex of $Q$, and every edge of $Q$ containing $v$ has lattice length $1$.
\end{enumerate}

Note that this generalises both the notion of reflexive polytope (the first case) and the notion of \emph{top}~\cite{CF98} (the second case).

Assuming that the minimal face of $Q$ containing the origin has dimension $d$, we say that a piece $Q$ has \emph{type} $3-d$. Given a smooth cone $\sigma$ with minimal face $\tau$ of dimension $d$, we call a reflexive piece $Q'$ contained in a two-dimensional face of $\sigma$ a \emph{panel} if $Q' \cap \tau$ has dimension $d$. Fixing a function $\bp$ from two-dimensional faces of $\sigma$ to panels contained in $\sigma$, we can consider the set of pieces $P$ of type $3-d$ such that every polygon in the image of $\bp$ is a facet of $P$. Let  $\cP(\bp)$ denote this set of pieces. Given an element $\varphi \in \cR_\Sigma$, let $\cS(\varphi)$ denote the set of functions from the collection of two dimensional cones $\tau$ of $\Sigma$ to panels contained in $\tau$ which are contained in the wrapping polyhedron defined by $\varphi$, and have one dimensional intersection with the boundary of this polyhedron.

\begin{algorithm}
	\label{alg:classification}
	Fix a complete fan $\Sigma$ in $N$ such that the dimension of the minimal cone of $\Sigma$ is at most one.
	\begin{enumerate}
		\item Compute the integral points in the polytope $\cR_\Sigma$.
		\item Exploit symmetries of $\Sigma$ to obtain a minimal subset $R$ of $\cR_\Sigma$, containing a representative of every isomorphism class of cracked polytope in $N_\RR$.
		\item Compute the set $\cS(\varphi)$ for each point $\varphi \in R$, and iterate over this set of functions.
		\item \label{it:pieces} For each $\varphi \in R$, $\bp \in \cS(\varphi)$, and maximal cone $\sigma \in \Sigma$, let $\bp_\sigma$ be the  restriction of $\bp$ to the two dimensional faces of $\sigma$. There is a finite subset $\cA(\varphi,\bp,\sigma)$ of $\cP(\bp_\sigma)$ such that, for each polytope $Q$ in this subset, $\langle w,v\rangle \geq -1$ for all inner normal vectors $w$ to facets of $Q$ which do not contain the origin, and vertices $v$ of polygons in the image of $\bp$ (note that this image is a strict superset of $\bp_\sigma$).
		\item For each function from the set of maximal cones $\sigma$ in $\Sigma$ to  $\coprod_{\sigma}\cA(\varphi,\bp,\sigma)$ such that the image of $\sigma$ is contained in $\cA(\varphi,\bp,\sigma)$, test whether the union of the polytopes in the image is itself a convex reflexive and cracked polytope.
	\end{enumerate}
\end{algorithm} 

\subsection{Classifying Pieces}

In order to implement Algorithm~\ref{alg:classification} in dimension $n$ we require a database of pieces in dimension $\leq n$. We now treat the classification of pieces in dimension $\leq 3$. Note that the classification in dimension $n$ divides into cases depending on the dimension $k$ of the minimal face containing $Q$. The cases $k=n$ and $k=n-1$ form known classes: indeed, if $k=n$, the corresponding pieces are polar dual to \emph{smooth} polytopes, which have a well-known classification up to dimension $8$ by \O{}bro~\cite{O07}. If $k=n-1$ the definition of reflexive piece coincides precisely with the notion of a \emph{top}~\cite{CF98,DD+14} which is also a unimodular polytope; we call such polytopes \emph{unimodular tops}.

In dimension one there are two possible cases, depending on the dimension $k$ of the minimal face of $P$ containing the origin:
\begin{itemize}
\item If $k=1$, $P = \conv{-1,1}$ is a line segment of length two.
\item If $k=0$, $P = \conv{0,1}$.
\end{itemize} 

It is well-known that hollow polytopes in dimension two are either Cayley polytopes or equal to $T := \conv{(0,0),(2,0),(0,2)}$ up to integral affine linear transformations. Hence we have three cases for pieces $P$ in $\RR^2$, depending on the dimension $k$ of the minimal face of $P$ containing the origin:
\begin{itemize}
	\item If $k=2$, $P$ is a reflexive polytope, of which five are unimodular.
	\item If $k=1$, $P = T$ or a quadrilateral isomorphic to
	\[
	\conv{(0,-1),(0,1),(1,-1),(1,m)},
	\]
	for some $m \in \ZZ_{\geq 0}$.
	\item If $k=0$, $P$ is isomorphic to
	\[
	\conv{(0,0),(0,1),(0,1),(1,m)},
	\]
	for some $m \in \ZZ_{\geq 0}$.
\end{itemize}

In dimension three we have four possible cases depending on $k$. In the case $k=3$, $P$ is a unimodular reflexive polytope, of which there are $18$. If $k=2$, $P$ is a \emph{unimodular top}. We do not describe the classification of unimodular tops in dimension $3$, as the algorithm given in \S\ref{sec:P1_shape} to treat the case $Z=\PP^1$ does not rely on this classification. Moreover, this classification is contained in that of all three dimensional tops made by Bouchard--Skarke~\cite{BS03}.

Assume next that $k=1$; that is, assume that the origin lies in an edge $E$ of the piece $P \subset \RR^3$. Fixing a vertex $v \in E$, and making a change of co-ordinates, we can assume that the edges incident to $v$ are parallel to the co-ordinate lines, $E$ has direction $e_3$, and $v = (0,0,-1)$. Since $E$ is itself a reflexive piece of dimension one, $(0,0,1)$ is a vertex of $P$. Let $F_1$ and $F_2$ denote the facets of $P$ containing the origin. For each $i \in \{1,2\}$, $F_i$ contains an edge $E_i$ incident to $(0,0,1)$ with direction vectors $(1,0,\alpha_1)$ and $(0,1,\alpha_2)$ respectively, such that,  by the unimodularity of $F_i$, $\alpha_i \geq -1$. Assume without loss of generality that $\alpha_1 \geq \alpha_2$. Since $F_i$ is a reflexive piece for each $i \in \{1,2\}$, if $\alpha_i > -1$ we have that 
\[
F_i = \conv{e_3,-e_3,e_i-e_3,e_i+(\alpha_i+1)e_3};
\]
while if $\alpha_i = -1$ we have that additional possibility that $F_i \cong T$. Let $\alpha := (\alpha_1,\alpha_2) \in \ZZ^2$, and, fixing a value of $l \in \ZZ_{\geq 0}$, define the Cayley polytopes $P(\alpha,l,1)$ and $P(\alpha,l,2)$ to be the convex hulls of the points given by the columns of the matrices
\[
\begin{array}{cccccccc}
0 & 0 & 1 & 0 & 1 & 0 & l & l \\
0  & 0 & 0 & 1 & 0 & 1 & 1 & 1 \\
1 & -1 & -1 & -1 & (\alpha_1+1) & (\alpha_2+1) & -1  & (\alpha_2 + l\alpha_1+1)
\end{array}
\]
and
\[
\begin{array}{cccccccc}
0 & 0 & 1 & 0 & 1 & 0 & 1 & 1 \\
0  & 0 & 0 & 1 & 0 & 1 & l & l \\
1 & -1 & -1 & -1 & (\alpha_1+1) & (\alpha_2+1) & -1  & (\alpha_1 + l\alpha_2+1)
\end{array}
\]
respectively.

\begin{lem}
	\label{lem:uni_cayley}
	Let $P_i$, $i \in \{1,\ldots k\}$ be a collection of $d$-dimensional lattice polytopes in $\RR^d$. If $P := P_1\star \cdots \star P_k \subset \RR^{d+k}$ is a unimodular polytope, there is a non-singular projective toric variety $Z$ such that $P_i$ is the polyhedron of sections of an ample divisor $D_i$ on $Z$ for all $i \in \{1,\ldots k\}$.
\end{lem}
\begin{proof}
	Since $P_{i_1}\star P_{i_2}$ is a face of $P$ for any $i_1, i_2 \in \{1,\ldots k\}$, we assume without loss of generality that $k=2$. Since $P_1$ is unimodular, its normal fan defines a non-singular projective toric variety $Z$. We claim that $P_2 = P_D$ for some ample divisor on $Z$.
 
	Note that $\V{P} = \V{P_1}\coprod \V{P_2}$. Moreover, each vertex of $P_1$ is contained in $d$ edges of $P_1$ and $(d+1)$ edges of $P$. Hence, fixing a facet $F$ of $P$ different from $P_1$ and $P_2$, $F \cap P_1$ is equal to a facet $G$ of $P_1$. $G$ contains $(d-1)$ edges of $P_1$ incident to $v$.
	
	The normal fan of $P$ consequently contains a ray for each facet of $P_1$ (or $P_2$), as well as rays  $\rho_1$, $\rho_2$ dual to $P_1$ and $P_2$ respectively. Moreover, each vertex of $P_1$ is dual to a maximal cone, generated by $\rho_1$ and rays corresponding to facets of $P_1$ containing $v$. Since the same applies to vertices of $P_2$, the toric variety associated to the normal fan of $P$ has the structure of a fibre bundle over $\PP^1$, in particular the fibres over $0$ and $\infty$ are isomorphic.
\end{proof}

\begin{lem}
	If $\alpha_1, \alpha_2 > -1$, then $P$ is isomorphic to $P(\alpha,l,j)$ for some $l \in \ZZ_{\geq 0}$ and $j \in \{1,2\}$.
\end{lem}
\begin{proof}
	The point $(1,1,0)$ cannot lie in the interior of $P$, and hence there is a $u \in N$ such that $\langle u, (1,1,0) \rangle \leq -1$, but $\langle u,p \rangle \geq -1$ for any point $p \in P$. In particular, writing $u = (u_1,u_2,u_3)$, and recalling that that $(0,0,\pm 1) \in P$, we have that $u_3 \in \{-1,0,1\}$. Similarly, $u_1 \geq -1+u_3$, $u_2 \geq -1+u_3$, $u_1 \geq -1-\alpha_1 u_3$ and $u_2 \geq -1-\alpha_2 u_3$. Hence, if $u_3 = 1$, $u_1 \geq 0$ and $u_2 \geq 0$, but no such points satisfy $u_1+u_2 \leq -1$. If $u_3=0$, we have the solutions $(u_1,u_2) = (-1,-1)$, $(-1,0)$, or $(0,1)$. These all define the Cayley sum of a pair of quadrilaterals, as $T$ is not a panel of $P$ by the assumption that $\alpha_i > -1$ for each $i \in \{1,2\}$. Since the panels of $P$ are Cayley polytopes (the sum of two line segments) -- and $P$ is unimodular -- $P$ is the Cayley sum of a pair of polyhedra of sections of ample divisors on a (fixed) Hirzebruch surface by Lemma~\ref{lem:uni_cayley}. Such a polytope is isomorphic to $P(\alpha,l,j)$ for some $\alpha$, $l$, and $j$.

	In the case $u_3=-1$ the bounds $\alpha_i > -1$ for each $i \in \{1,2\}$, together with the inequalities $u_1 \geq -1-\alpha_1 u_3$ and $u_2 \geq -1-\alpha_2 u_3$, ensure that there are no further cases.
\end{proof}

\begin{figure}
	\includegraphics[scale=1.5]{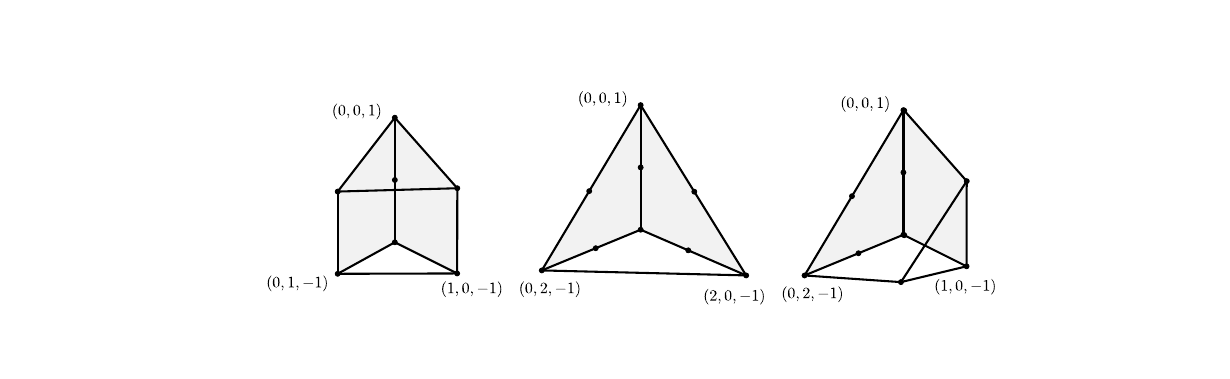}
	\caption{Exceptional pieces.}
	\label{fig:Exceptional}
\end{figure}

Note that $P(\alpha,0,1) = P(\alpha,0,2)$ and $P(\alpha,-1,1) = P(\alpha,-1,2)$. Note also that whenever $\alpha_1 = \alpha_2$, $P(\alpha,l,1) \cong P(\alpha,l,2)$, although these polytopes are not equal. The remaining cases are $(\alpha_1,\alpha_2) = (0,-1)$ and $(\alpha_1,\alpha_2) = (-1,-1)$. In the latter case $P$ is a sub-polytope of $\conv{-e_3,e_3,2e_1-e_3,2e_2-e_3}$, and hence there are three possible polytopes, illustrated in Figure~\ref{fig:Exceptional}. In the case $(\alpha_1,\alpha_2) = (0,-1)$, we introduce another infinite class of polytopes. Fixing a value of $l \in \ZZ_{\geq 1}$ define the `wedge' polytope $W(l)$ to be the convex hull of the points given by the columns of the following matrix,
\[
\begin{array}{cccccccc}
 0 & 0 & 0 & 1 & 1 & l & l & 2(l-1) \\
 0  & 0 & 2 & 0 & 0 & 1 & 1 & 2 \\
 1 & -1 & -1 & -1 & 1 & 0 & -1  & -1
\end{array}.
\]
See Figure~\ref{fig:type_W} for an illustration of such a polytope. We also define 
\[
W'(l) := W(l) \cap \{x : \langle (-1,1,0),x \rangle \leq 1\}
\]
for each $l$. There are additional cases which appear for small values of $l$; in particular we define the polytopes $W_0(l)$ to be the convex hull of the points given by the columns of the following matrix,
\[
\begin{array}{cccccc}
0 & 0 & 0 & 1 & 1 & 2l-1 \\
0  & 0 & 2 & 0 & 0 & 2 \\
1 & -1 & -1 & -1 & 1 & -1
\end{array}.
\]
and $W'_0(l) := W_0(l) \cap \{x : \langle (-1,1,0),x \rangle \leq 1\}$ for each $l \in \{1,2\}$.

		\begin{figure}
	\centering
	\begin{minipage}[b]{0.49\textwidth}
		\includegraphics{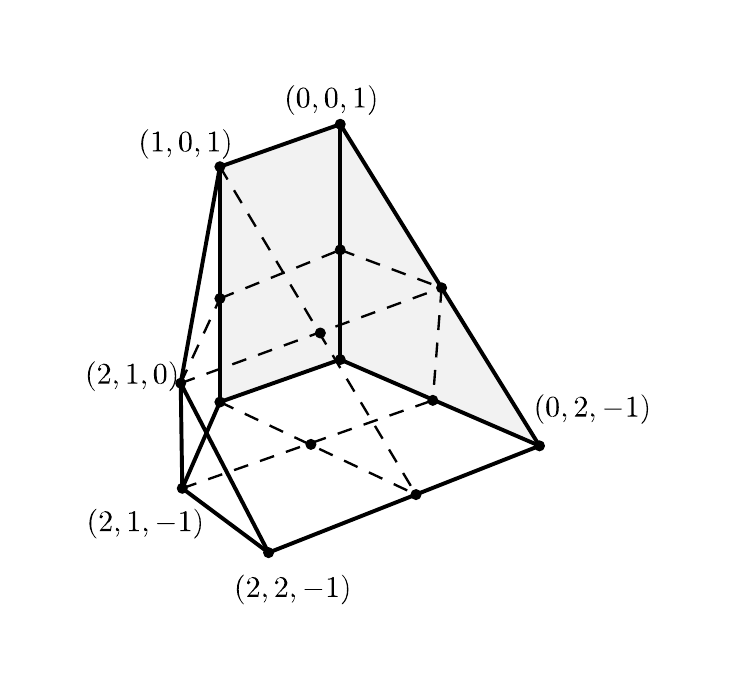}
	\end{minipage}
	\hfill
	\begin{minipage}[b]{0.49\textwidth}
		\includegraphics{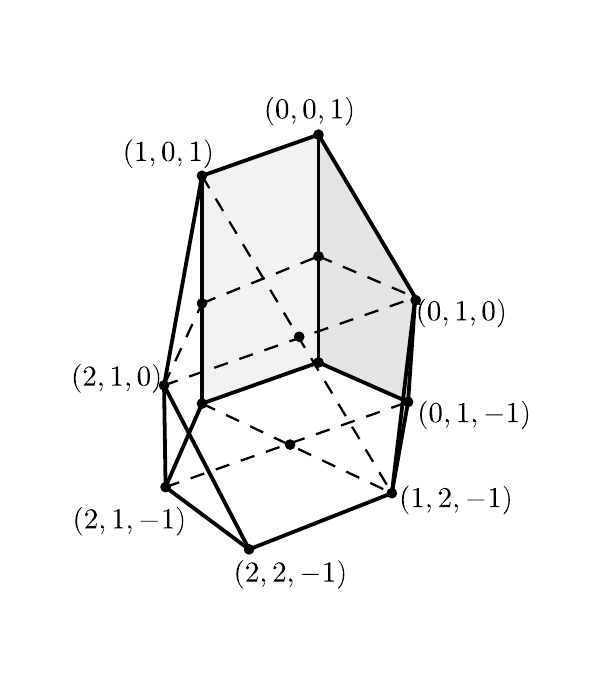}
	\end{minipage}
	\caption{Pieces $W(2)$ and $W'(2)$.}
	\label{fig:eg2-18}
\end{figure}

\begin{figure}

	\label{fig:type_W}
\end{figure}

\begin{lem}
	\label{lem:special_alpha}
	If $\alpha = (0,-1)$, then $P$ is isomorphic to one of
	\begin{enumerate}
		\item $W(l)$, for some $l \in \ZZ_{\geq 2}$,
		\item $W'(l)$, for some $l \in \ZZ_{\geq 2}$,
		\item $W_0(l)$ for $l \in \{1,2\}$,
		\item $W'_0(1)$; or,
		\item $P(\alpha,l,1)$ for some $l \in \ZZ_{\geq 0}$.
	\end{enumerate}
\end{lem}
\begin{proof}
	Since $\alpha = (0,-1)$ the polytope $P$ is contained in the half-space $\{x: \langle (0,1,1),x\rangle \leq 1\}$. Moreover $P$ is assumed to be contained in the positive orthant; that is,
	\[
	P \subset A := \RR_{\geq-1} \times \conv{(0,\pm1),(2,-1)}.
	\]
	We claim such pieces $P$ are determined by the facet $F = P \cap \{u : \langle u,e^\star_3 \rangle = -1 \}$. Indeed, fixing this polygon $F$ it is easy to verify that
	\[
	P = A \cap (F \times \RR).
	\]
	The possible polygons $F$ are also easily classified. Choose co-ordinates on $\RR^2$ such that the origin and $(1,0)$ are vertices of $F$. If $F \cap \{y=2\} = \varnothing$ both $F$ and $P$ are Cayley polytopes, and $P = P(\alpha,l,1)$ for some $l \geq 0$. Otherwise $F$ is a (possibly degenerate) hexagon with vertices given by the columns of
	\[
	\begin{array}{cccccc}
	0 & 1 & 1 & 2 & 2 & a \\
	0 & 0 & k_1 & k_2 & k_3 & 0,
	\end{array}
	\]
	where $a \in \{1,2\}$. Fix a value of $k_1 \geq 0$. By convexity and unimodularity of $F$ at $(1,k_1)$, we have that $k_2 = 2(k_1-1)$; unless $k_1 \in \{1,2\}$; which gives the additional cases $(k_1,k_2) = (1,1)$ and $(k_1,k_2) = (2,3)$. If $a=2$, $k_3=0$ and $P = W(l)$ for some $l \in \ZZ_{\geq 0}$ or $W_0(l)$ for some $l \in \{0,1\}$. Otherwise $a=1$ and we have that $k_3 = 1$ (note $k_3 \neq 0$ as $(1,0)$ is vertex of $F$) by unimodularity of $F$ at the point $(2,k_3)$. In these cases $P = W'(l)$ for some $l \in \ZZ_{\geq 1}$ or $W'_0(2)$. Note that $W(1)$ and $W'_0(1)$ are not unimodular. Moreover, $P(\alpha,l,1) = P(\alpha,l,1)$ for $l \in \{0,1\}$, while $P(\alpha,l,2)$ is not unimodular if $l > 1$. 
\end{proof}
We summarise the above calculations in the following proposition.

\begin{pro}
	If $P$ is a $3$-dimensional piece and the origin is contained in the relative interior of an edge of $P$, then $P$ belongs to one of the infinite families $P(\alpha,l,j)$, one of the three exceptional cases shown in Figure~\ref{fig:Exceptional}, or one of the polytopes listed in Lemma~\ref{lem:special_alpha}.
\end{pro}
Finally, assume that $k=0$. For each $l \in \ZZ_{\geq 0}$ and $j \in \{1,2\}$, we define the Cayley polytopes $Q(\alpha,l,j)$ to be the intersection of $P(\alpha,l,j)$ with the half-space $\{u \in \RR^3 : \langle e_3^\star,u \rangle \geq 0\}$.

\begin{pro}
	\label{pro:zero_k}
	If $P$ is a $3$-dimensional piece and the origin is a vertex of $P$, then $P$ belongs to the infinite family $Q(\alpha,l,j)$. The polytope $Q(\alpha,l,j)$ is a reflexive piece if and only if one of the following hold.
	\begin{enumerate}
		\item $\alpha_1 \geq 0$, $\alpha_2 \geq 0$, $j \in \{1,2\}$, and $l \in \ZZ_{\geq 0}$.
		\item $\alpha_1 = 0$, $\alpha_2 = -1$, $j = 1$ and $l \in \ZZ_{\geq 0}$.
		\item $\alpha_1 \geq 0$, $\alpha_2 = -1$, $j = 2$ and $l =\alpha_1+1$.
		\item $\alpha_1 = -1$, $\alpha_2 = -1$.
	\end{enumerate}
	Note that the only polytope which appears in the fourth case is the standard simplex.
\end{pro}
\begin{proof}

The vertex set of a piece $P$ contains the origin, and -- in a suitable co-ordinate system -- each of the three standard basis vectors. The polygon $F_i := \{e^\star_i=0\} \cap P$ is a two dimensional reflexive piece, which were classified above.

\begin{figure}
	\includegraphics{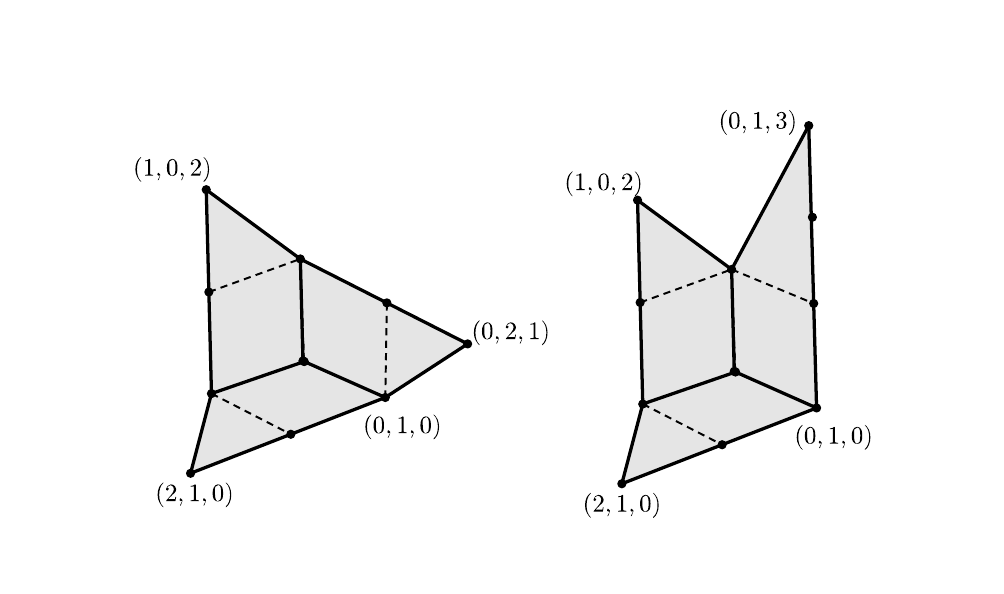}
	\caption{Relative arrangement of panels.}
	\label{fig:Q_types}
\end{figure}

\begin{figure}
	\includegraphics{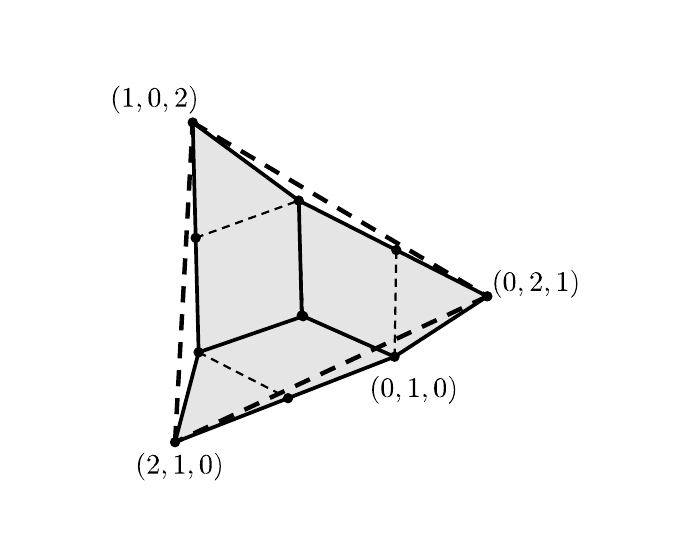}
	\caption{An impossible arrangement of panels.}
	\label{fig:no_type}
\end{figure}

Thus we may assume that each polygon $F_i$ is either a standard triangle or a Cayley sum of line segments. These polygons may be oriented relative to each other in two distinct ways, illustrated in Figure~\ref{fig:Q_types}. We show that the first case does not include any piece which is not a special case of the second. Polytopes in the first case contain vertices $(1,0,k_1)$, $(k_2,1,0)$, and $(0,k_3,1)$. Note that we can assume that $k_i \geq 2$. If $k_i > 2$ for any $i \in \{1,2,3\}$, the lattice point $(1,1,1)$ is in the interior of the convex hull of the vertices of $P$, and hence $k_1=k_2=k_3=2$. However, as $P$ is contained in the half space $\{u \in \RR^3 : (1,1,1)\cdot u \leq 3\}$, $P$ is a sub-polytope of the convex hull $P'$ of the vertices shown in Figure~\ref{fig:no_type}. Note that every vertex of this polytope is contained in a panel, and hence $P=P'$. Since $P'$ is not unimodular it does not contribute to the list of pieces.

\begin{figure}
	\includegraphics[scale=1.3]{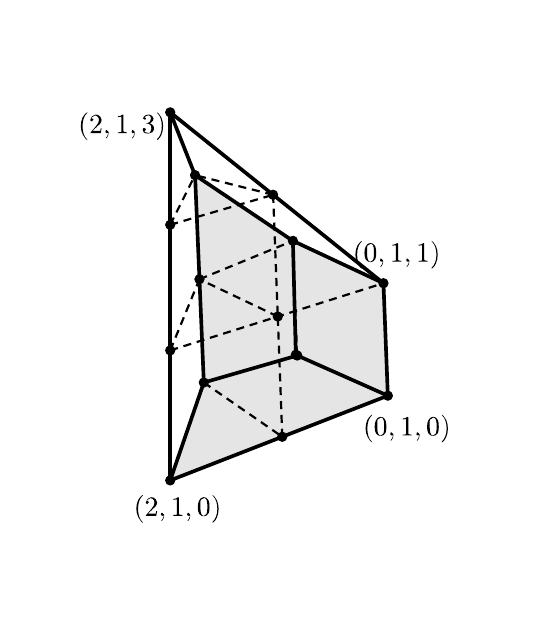}
	\caption{Piece $Q((1,0),1,1)$ }
	\label{fig:type_Q}
\end{figure}

In the second case illustrated in Figure~\ref{fig:type_Q}, we observe that $P$ is a Cayley polytope. Indeed, assuming that $P$ contains the vertices $(1,0,k_1)$, $(k_2,1,0)$, and $(0,1,k_3)$, $P$ is the Cayley sum of the facets contained in $H_0$ and $H_1$, where $H_k := \{u : \langle e_2^\star, u \rangle = k \}$. These are both $2$-dimensional if $\alpha_1 \geq 0$ and $\alpha_2\geq 0$; and in this case it follows from Lemma~\ref{lem:uni_cayley} that $P$ is of the form $Q(\alpha, l,j)$ for some $l \in \ZZ_{\geq 0}$ and $j \in \{1,2\}$. The classification of the remaining possible pieces follows from a case-by-case analysis. The case $\alpha = (-1,-1)$ is trivial. If $\alpha = (0,-1)$, $P$ is contained in the product of a standard simplex and a ray, and equal to some $Q(\alpha,1,l)$. If $\alpha_1 >0$ and $\alpha_2=-1$ we note that the polytopes $Q(\alpha,1,l)$ are not unimodular, while $Q(\alpha,2,l)$ is a Cayley polytope $P_1\star P_2$, such that $P_1$ is a standard simplex. $P_2$ is a dilate of a standard simplex by Lemma~\ref{lem:uni_cayley}, and hence $l = \alpha_1+1$.
\end{proof}
\section{Connection to the Gross--Siebert program}
\label{sec:gross_siebert}

The results and computations of this article fit into a larger program of research, directed toward a novel method of Fano classification. In particular, the authors of \cite{CCGK} construct a database of polytopes which support a mirror (\emph{Minkowski}) Laurent polynomial to a given Fano threefold, see \href{www.fanosearch.net}{www.fanosearch.net}. It is conjectured that this database describes precisely the toric varieties (associated to Minkowski polytopes) which smooth to a given Fano threefold.

In this article we have constructed degenerations proving part of this conjecture: every toric variety we obtain by degenerating a Fano threefold appears in the database generated in \cite{CCGK}. As discussed in the introduction, the Gross--Siebert program suggests a general approach to relate toric degeneration and mirror Laurent polynomials. Loosely, we first degenerate the toric variety $X_P$, associated to the Newton polytope $P$ of a Minkowski polynomial, to a union of toric varieties. Using methods from tropical and log geometry we can then (attempt to) generate both the smoothing of $X_P$ to a Fano threefold and the Laurent polynomial mirror.

More specifically, we expect that our families are fibrewise compactifications of families mirror to certain log Calabi--Yau varieties, which may themselves be constructed from a scaffolding. The two dimensional version of this program is current work in progress with Barrott and Kasprzyk, and we now outline the main features of the construction. As remarked in the introduction, if this program were complete in dimension three, the current work would relate the constructions of Mori--Mukai and the toric degenerations obtained via families mirror to a given log Calabi--Yau variety.

\subsection{Compactifying families of log Calabi--Yau varieties}

Fix a scaffolding $S$ of a Fano polytope $P$ with shape determined by the fan $\Sigma$. Assume, for simplicity, that $N_U = \{0\}$, and hence $\bar{N} = N \cong \ZZ^n$. The induced inclusion $\iota \colon X_P \hookrightarrow Y_S$ fits into the commutative diagram
\[
\xymatrix{
	X_P \ar@{^{(}->}[r] & Y_S \\
	\CC^n \ar[u] \ar@{^{(}->}[r] & \CC^{\Sigma(1)}, \ar[u],
}
\]
where the horizontal and vertical arrows are closed and open embeddings respectively. Using standard methods, we can degenerate $\CC^n$ into a union of copies of $\CC^n$, determined by the cones of the (unimodular) fan $\Sigma$. Moreover, there is a canonical embedding of this degeneration into $\CC^{\Sigma(1)}$. We propose to consider the extension of this degeneration over the base of a family of log Calabi--Yau varieties considered by Gross--Hacking--Keel (in two dimensions) \cite{GHK1}, and by Gross--Hacking--Siebert \cite{GHS} in higher dimensions.

Assume for now that $n = 2$, and $S$ is a full scaffolding of $P$. Let $Z$ denote the shape variety of $S$, the toric variety associated to $\Sigma$. We construct a log Calabi--Yau variety $U$ by blowing up points on the toric boundary of $Z$, and propose that the mirror family $\cV \to T$ -- constructed in \cite{GHK1} -- fits into the following commutative diagram, where $\cX \to T$ is projective and flat, and $\cV$ is an open subscheme of $\cX$:
\begin{equation}
\label{eq:families}
\xymatrix{
	\cX \ar@{^{(}->}[r] & Y_S\times T \\
	\cV \ar[u] \ar@{^{(}->}[r] & \CC^{\Sigma(1)}\times T \ar[u].
}
\end{equation}

We construct the variety $U$ using a notion of \emph{mutability} for the scaffolding $S$. We recall from \cite{ACGK} that a mutation of a polygon is determined by a \emph{weight vector} $w \in M \cong \ZZ^2$, and a \emph{factor} $F \subset w^\bot$. We refer to \cite{ACGK} for the full definition of polytope mutation, but recall that a polytope $P$ admits a mutation with respect to $(w,F)$ if and only if each 
\[
P_a := P \cap \{x : \langle x , w \rangle = a\}
\]
contains a translate of the polytope $\max(-a,0)F$ whenever $P_a \neq \varnothing$. Fixing a convention for the orientation of $w^\bot$, a mutation in two dimensions is determined by the weight vector $w$.

\begin{dfn}
	Given a pair $(w,F)$, we say that $S$ admits a mutation in $(w,F)$ if the polytope $P_D + \chi$ admits this mutation for each element $(D,\chi) \in S$.
\end{dfn}

Fix a scaffolding $S$ with shape $Z$, where $Z$ is a product of projective spaces. We recall from \cite{CKP17} that there is a standard choice of Laurent polynomials $f_s$, such that $\Newt(f_s) = P_D + \chi$, where $s = (D,\chi) \in S$. Thus there is a standard choice of Laurent polynomial
\[
f = \sum_{s \in S}{f_s}
\]
such that $\Newt(f) = P$. If $S$ is mutable, the Laurent polynomial $f$ admits an \emph{algebraic mutation} \cite{ACGK} (also called a \emph{symplectomorphism of cluster type} \cite{Katzarkov--Przyjalkowski}). Hence we expect that $f$ defines a global function on the variety $U$ defined (in the $2$ dimensional case) as follows.

\begin{cons}
Let $v_\rho$ denote the ray generator of the ray $\rho$ of $\Sigma$. Given a scaffolding $S$ of $P$, suppose that $S$ admits a mutation with weight vector $v_\rho$ and factor $F_\rho$ of lattice length $\ell_\rho$. Let $U$ be the complement of the strict transform of the toric boundary of $Z$ under the blow-up $\pi$ of $Z$ with $k_\rho$ distinct reduced centres on the boundary divisor of $Z$ corresponding to each ray $\rho$. 
\end{cons}

The following conjecture is the main result of \cite{BKP}.

\begin{conjecture}
	\label{conj:GHK_smoothing}
	The mirror family $\cV$ to the log Calabi--Yau $U$ constructed in \cite{GHK1} fits into the commutative diagram \eqref{eq:families}.
\end{conjecture}

Conjecture~\ref{conj:GHK_smoothing} offers a systematic way of constructing the deformations we build by hand throughout this article. The situation in higher dimensions is the subject of current and exciting research. We particularly refer here to ongoing work of Corti--Hacking--Petracci \cite{CHP}, which may be interpreted as an extension of Conjecture~\ref{conj:GHK_smoothing} to higher dimensions, in which the map $X_P \to Y_S$ is the anti-canonical embedding of the Gorenstein toric Fano variety $X_P$.

\subsection{Example: $A_2$ cluster variety}

We consider a particular case of the mirror family to a log Calabi--Yau in some detail. Let $Z$ be the toric variety $dP_7$, obtained by blowing up $\PP^1\times \PP^1$ in a single torus invariant point. Let $\bar{U}$ be the blow up of a (non-special) point on each of the pair of torus invariant curves $C$ in $Z$ such that $C^2 = 0$. Let $U$ be the complement of the strict transform of the toric boundary of $Z$ in $\bar{U}$. It is well known, see for example \cite{GHK2}, that $U$ is the $\cA$ cluster variety associated to an $A_2$ quiver. The mirror family, described by \cite{GHK1} using $\theta$-functions, is a family over $\Spec(\mathbf{k}[\NE(\bar{U})]) \cong \AA_{\mathbf{k}}^5$, for a choice of ground field $\mathbf{k}$. Specialising the parameters corresponding to $\NE(Z)$ to $1$, we obtain the $2$ parameter family defined by the $4 \times 4$ Pfaffians of the matrix
\begin{equation}
\label{eq:affine_pfaffians}
\begin{pmatrix}
1 & x_1 & x_2 & 1 \\
& t_1 & x_3 & x_4  \\
& & 1 & x_5 \\
& & & t_2 \\
\end{pmatrix},
\end{equation}
where $x_1,\ldots, x_5$ denote the theta functions corresponding to the five rays shown in Figure~\ref{fig:pfaffian_aff}. The parameters $t_1$ and $t_2$ correspond to the curve classes of the exceptional locus of the contraction $\bar{U} \to Z$. We associate an integral affine manifold (with singularities) to $U$, illustrated in Figure~\ref{fig:pfaffian_aff}. The singular locus consists of a pair of \emph{focus-focus} singularities. There is a monodromy operator conjugate to $\begin{pmatrix}
1 & 1 \\ 0 & 1
\end{pmatrix}$ associated to each focus-focus singularity such that the subspace invariant under each operator is parallel to the ray containing the corresponding singular point.

\begin{figure}
	\includegraphics{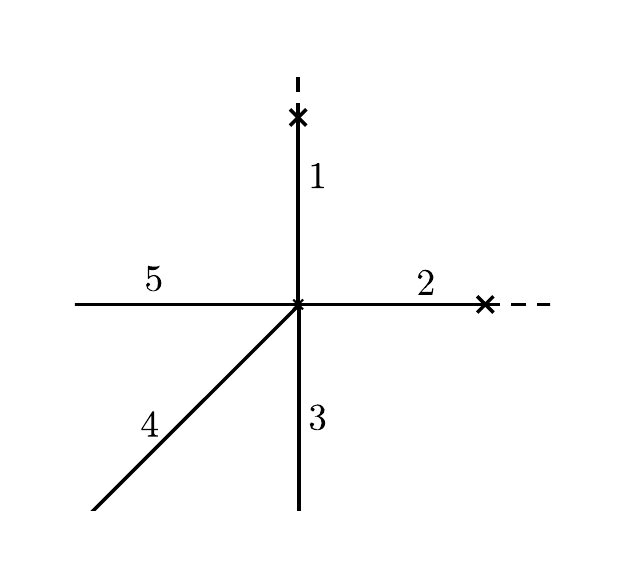}
	\caption{Affine manifold associated to $V$.}
	\label{fig:pfaffian_aff}
\end{figure}

Fix a Fano polygon $P$ together with a scaffolding $S$ which has shape $Z$. An example of such a scaffolding is shown in Figure~\ref{fig:pfaffian_example}. Given an element $s = (D,0) \in S$, let $s_i$ denote the $i$th co-ordinate of $D \in \Div_T(Z) = \ZZ^5$, using the ordering of the basis elements shown in Figure~\ref{fig:pfaffian_aff}. Recalling that the scaffolding $S$ defines a toric embedding $X_P \to Y_S$, let $x_i$ be the homogeneous coordinate corresponding to the $i$th basis element in $\Div_T(Z)$ for each $i \in \{1,\ldots,5\}$.

The scaffolding $S$ determines an embedding of $X_P$ of codimension $3$. Moreover, explicitly computing the ideal of this toric embedding, the image of $X_P$ in $Y_S$ is given by the $4\times 4$ Pfaffians of the matrix
\begin{equation}
\label{eq:degenerate_pfaffians}
	\begin{pmatrix}
	\prod_{s}y_s^{s_1+s_4-s_5} & x_1 & x_2 & \prod_{s}y_s^{s_2+s_4-s_3} \\
	& 0 & x_3 & x_4  \\
	& & & \prod_{s}y_s^{s_3+s_5-s_4} & x_5 \\
	& & & & 0 \\
	\end{pmatrix}.
\end{equation}

Note that each of the exponents of entries in the matrix appearing in \eqref{eq:degenerate_pfaffians} is non-negative as, writing $s = (D,0)$, $D$ is nef for any $s \in S$. In fact the nef cone of $dP_7$ is defined by the inequalities $s_3+s_5 \geq s_4$, $s_2+s_4\geq s_3$, and $s_1+s_4\geq s_5$. This can easily deduced from Figure~\ref{fig:pfaffian_strut}, which illustrates a general polygon with shape $Z$.

\begin{figure}
	\includegraphics[scale=0.8]{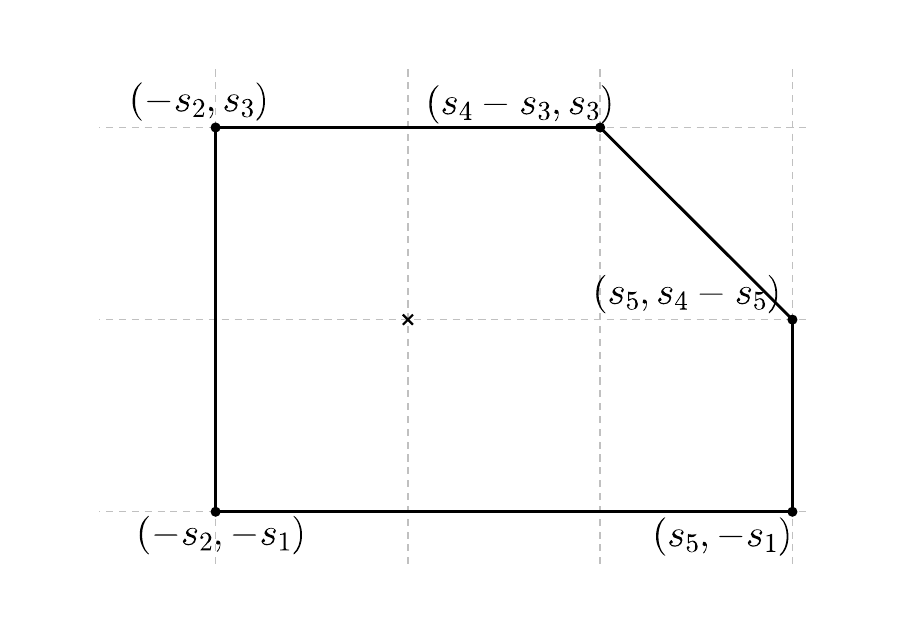}
	\caption{Example of a strut with shape $Z$.}
	\label{fig:pfaffian_strut}
\end{figure}

The variety $X_P$ fits into the two-parameter family defined by the $4\times 4$ Pfaffians of the matrix
\[
\begin{pmatrix}
\prod_{s}y_s^{s_1+s_4-s_5} & x_1 & x_2 & \prod_{s}y_s^{s_2+s_4-s_3} \\
& t_1\prod_{s}y_s^{s_1+s_3-s_2} & x_3 & x_4  \\
& & \prod_{s}y_s^{s_3+s_5-s_4} & x_5 \\
& & & t_2\prod_{s}y_s^{s_2+s_5-s_1} \\
\end{pmatrix}
\]
if and only if the exponents $s_1+s_3 \geq s_2$ and $s_2+s_5 \geq s_1$ are non-negative for each $s \in S$. However, this is immediately equivalent to the mutability $S$ with respect to the weight vectors $(1,0)$ and $(0,1)$. Hence, mutability of the scaffolding is precisely the condition required for the mirror family to admit a compactification in $Y_S$.

\begin{figure}
	\includegraphics[scale = 0.8]{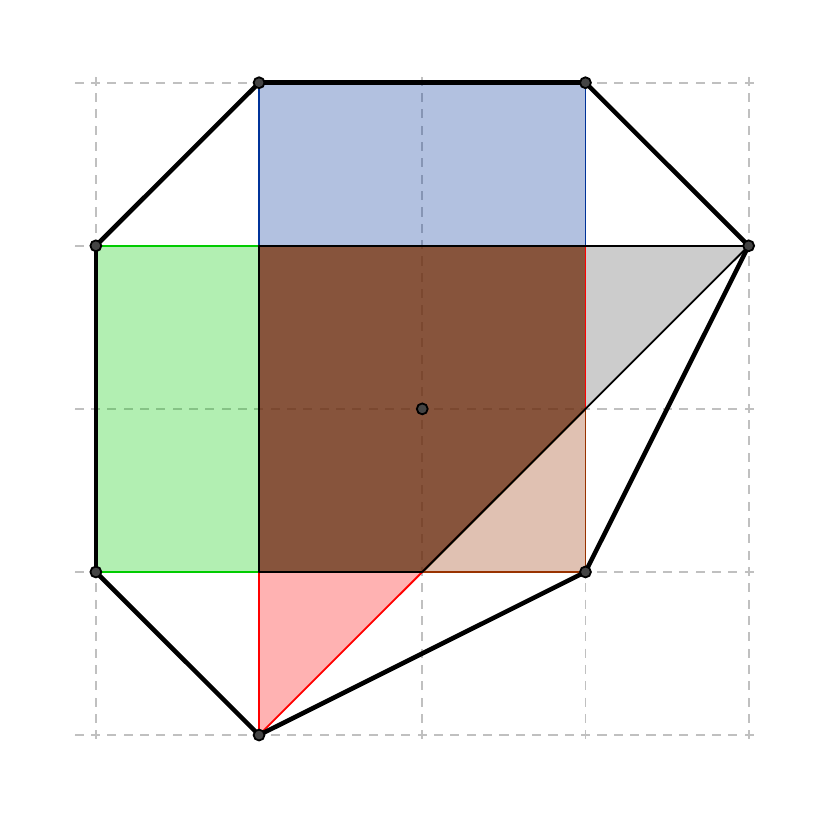}
	\caption{Example scaffolding with shape $Z$.}
	\label{fig:pfaffian_example}
\end{figure}

For example, consider the polygon $P$, with scaffolding $S$ shown in Figure~\ref{fig:pfaffian_example}. This is evidently a mutable scaffolding, and indeed a general fibre of the family has equations is identical to those used to construct the log del Pezzo surface $X_{5,5/3}$ in \cite{CH16}.

\subsection{The Gross--Siebert algorithm}

As well as the approach exploiting the results \cite{GHK1,GHS} detailed above, we may attempt to make direct use of the \emph{Gross--Siebert algorithm}, introduced in \cite{Gross--Siebert}. The existence of this algorithm entails a powerful smoothing result, namely that any \emph{locally rigid, positive, pre-polarized tori log Calabi--Yau space} arises from a formal degeneration of log Calabi--Yau pairs. Moreover, the extension of these families to families over an analytic base with canonical co-ordinates is known (at least in the Calabi--Yau context) and we refer to the article \cite{RS14} of Ruddat--Siebert -- and current work in progress of these authors -- for further details.


Therefore, if we can adapt our constructions to define such a toric log Calabi--Yau space we can define a smoothing using constructions in logarithmic geometry. The technical difficulties here are two-fold.

\begin{enumerate}
	\item Local rigidity is a strong condition, and is restrictive even in three-dimensions.
	\item Construction of a locally rigid toric log Calabi--Yau involves refining the triangulation of $P$, and we lose a reasonable ambient space for the resulting formal degeneration.
\end{enumerate}

In \cite{Gross--Siebert} the authors explain how toric log Calabi--Yau spaces may be constructed from certain integral affine manifolds, together with additional discrete data (such as a polyhedral decomposition). Local rigidity is related to the notion of \emph{simplicity} of the singularities of an integral affine manifold $B$ associated to a toric log Calabi--Yau space. In three dimensions, an integral affine manifold with simple singularities is a topological manifold with an integral affine structure in the complement of a trivalent graph $\Delta$, together with conditions on the monodromy of the integral affine structure around edges of $\Delta$.

In our context $B$ is the polytope $P^\star$, dual to $P$, and we note that an integral affine manifold with simple singularities corresponding to each family of Fano threefolds was constructed in the author's earlier work \cite{P:Fibrations}. Constructing a \emph{polyhedral decomposition} and \emph{polarization} compatible with these integral affine manifolds allows us to describe a locally rigid toric log Calabi--Yau space. Indeed, adapting the constructions in \cite{P:Fibrations}, we expect that the Gross--Siebert algorithm can be used to construct all families of Fano $3$-folds in this way. We note that the deformations of log Calabi--Yau spaces are still the subject of active research, and we hope that very recent work of Filip--Felton--Ruddat will allow us to overcome some of the technical difficulties presented by the requirement of local rigidity in the Gross--Siebert algorithm.

\appendix

\bibliographystyle{plain}
\bibliography{bibliography}
\end{document}